\documentclass[reqno,a4paper,11pt]{amsart}
\usepackage[ngerman, english]{babel}
\usepackage{graphicx}
\usepackage{amsmath}
\usepackage{amssymb}
\usepackage{amsfonts}
\usepackage{enumerate}
\usepackage[arrow, matrix, curve]{xy}
\usepackage{url}
\usepackage[pagebackref,colorlinks]{hyperref}
\usepackage{comment}
\usepackage{multirow}

\let\oldtocsection=\tocsection
\let\oldtocsubsection=\tocsubsection
\renewcommand{\tocsection}[2]{\hspace{0em}\oldtocsection{#1}{#2}}
\renewcommand{\tocsubsection}[2]{\hspace{1em}\oldtocsubsection{#1}{#2}}

\theoremstyle{plain}
\newtheorem {Lemma}{Lemma}
\newtheorem {Proposition}[Lemma]{Proposition}
\newtheorem {Theorem}[Lemma]{Theorem}
\newtheorem {Corollary}[Lemma]{Corollary}

\theoremstyle{definition}
\newtheorem{Definition}[Lemma]{Definition}

\newtheorem{Remark}[Lemma]{Remark}
\newtheorem {Example}[Lemma]{Example}

\numberwithin{equation}{section}
\title{The E-Normal Structure of Odd Dimensional Unitary Groups}
\author{Anthony Bak}
\address{Department of Mathematics,
University of Bielefeld, Germany}
\email{bak.biel@gmail.com, bak@math.uni-bielefeld.de}
\author{Raimund Preusser}
\address{Department of Mathematics,
University of Brasilia, Brazil}
\email{raimund.preusser@gmx.de}
\date{}
\newcommand{\m}{\mathfrak{m}}
\newcommand{\+}{\overset{.}{+}}
\newcommand{\minus}{\overset{.}{-}}
\newcommand{\h}{\mathfrak{H}}
\newcommand{\GL}{\operatorname{GL}}
\newcommand{\E}{\operatorname{E}}
\newcommand{\C}{\operatorname{C}}
\newcommand{\U}{\operatorname{U}}
\newcommand{\EU}{\operatorname{EU}}
\newcommand{\UEU}{\operatorname{UEU}}
\newcommand{\TEU}{\operatorname{TEU}}
\newcommand{\NU}{\operatorname{NU}}
\newcommand{\CU}{\operatorname{CU}}
\newcommand{\Ort}{\operatorname{O}}
\newcommand{\Sp}{\operatorname{Sp}}
\newcommand{\tr}{\operatorname{tr}}
\newcommand{\Center}{\operatorname{Center}}

\newcommand{\Normalizer}{\operatorname{Normalizer}}
\newcommand{\M}{\operatorname{M}}
\newcommand{\FP}{\operatorname{FP}}
\makeatletter

\begin{document}
\maketitle
\begin{abstract}In this paper we define odd dimensional unitary groups $\U_{2n+1}(R,\Delta)$. These groups contain as special cases the odd dimensional general linear groups $\GL_{2n+1}(R)$ where $R$ is any ring, the odd dimensional orthogonal and symplectic groups $\Ort_{2n+1}(R)$ and $\Sp_{2n+1}(R)$ where $R$ is any commutative ring and further the first author's even dimensional unitary groups $\U_{2n}(R,\Lambda)$ where $(R,\Lambda)$ is any form ring. We classify the E-normal subgroups of the groups $\U_{2n+1}(R,\Delta)$ (i.e. the subgroups which are normalized by the elementary subgroup $\EU_{2n+1}(R,\Delta)$), under the condition that $R$ is either a semilocal or quasifinite ring with involution and $n\geq 3$. Further we investigate the action of $\U_{2n+1}(R,\Delta)$ by conjugation on the set of all E-normal subgroups.
\end{abstract}

\tableofcontents
\section{Introduction}
Towards the end of the 19th century mathematicians in particular C. Jordan \cite{jordan} and L. E.
Dickson \cite{dickson, dickson_book} became interested in the normal subgroups of classical groups, at first over finite fields and then over infinite fields. The case of finite fields was of special interest because normal subgroups here yielded concrete examples of finite nonabelian simple groups which were important for Galois theory. In the 1940's, 50's and 60's, J. Dieudonne \cite{dieudonne_book_1, dieudonne_1, dieudonne_2,  dieudonne_book_2} and others extended the work above to include all classical groups over fields and generalized it to include skew fields. In the early 60's, W. Klingenberg \cite{klingenberg_1,klingenberg_2,klingenberg_4,klingenberg_3} generalized the results to classical groups over commutative local rings and in the case of the general linear group $\GL_n$ to noncommutative local rings. In 1964, H. Bass \cite{bass} described the E-normal subgroups (i.e. the subgroups normalized by the elementary subgroup) of general linear groups over rings satisfying a stable range condition. Namely he proved that if $H$ is a subgroup of the general linear group $\GL_n(R)$, then
\begin{align*}
H \text{ is E-normal }\Leftrightarrow \exists!\text{ ideal } I:\E_n(R,I) \subseteq H\subseteq \C_n(R,I)\tag{1.1}
\end{align*}
for all rings $R$ of finite stable range $r$ and $n\geq max(r + 1, 3)$. In (1.1), $\E_n(R,I)$ denotes the relative elementary subgroup of level $I$ and $\C_n(R,I)$ denotes the full
congruence subgroup of level $I$. In 1972, J. S. Wilson \cite{wilson} proved that (1.1) holds if $R$ is commutative and $n\geq 4$, without a stable range restriction. One year later, I. Z. Golubchik \cite{golubchik} added the case that $R$ is commutative and $n=3$. In 1981, L. N. Vaserstein \cite{vaserstein} proved that (1.1) holds if $R$ is almost commutative (i.e. finitely generated as module over its center) and $n\geq 3$.

A natural question is if similar results hold for other classical groups. In 1969, the first author \cite{bak_2} defined (even dimensional) hyperbolic unitary groups $\U_{2n}(R,\Lambda)$ which include as special cases the groups $\GL_{2n}(R)$ where $R$ is any ring and the classical groups $\Ort_{2n}(R)$ and $\Sp_{2n}(R)$ where $R$ is any commutative ring. He proved that if $H$ is a subgroup of $\U_{2n}(R,\Lambda)$, then
\begin{align*}
H \text{ is E-normal }\Leftrightarrow \exists!\text{ form ideal } (I,\Gamma):\EU_{2n}((R,\Lambda),(I,\Gamma)) \subseteq H\subseteq \CU_{2n}((R,\Lambda),(I,\Gamma))\tag{1.2}
\end{align*}
for all form rings $(R,\Lambda)$ such that $R$ is semilocal or has finite Bass-Serre dimension $d$ and $n \geq max(d + 2,3)$. In (1.2), $\EU_{2n}((R,\Lambda),(I,\Gamma))$ denotes the relative elementary subgroup of level $(I,\Gamma)$ and $\CU_{2n}((R,\Lambda),(I,\Gamma))$ denotes the full congruence subgroup of level $(I,\Gamma)$. In 1989, E. Abe \cite{abe} proved that (1.2) holds for the classical groups $\Sp_{2n}(R)$  and $\Ort_{2n}(R)$, namely those of Chevalley types $C$ and $D$,  for all $n\geq 3$ and all commutative rings $R$. In 1995, L. N. Vaserstein and H. You \cite{vaserstein-you} gave an incorrect proof of (1.2) for the case that $R$ is almost commutative and $n\geq 3$. The proof can be repaired when $2$ is invertible in $R$. In 2012, H. You \cite{you} proved that (1.2) holds if $R$ is commutative and $n\geq 4$ and one year later H. You and X. Zhou \cite{you-zhou} added the case that $R$ is commutative and $n=3$. Recently the second author \cite{preusser} proved that (1.2) holds if $R$ is semilocal or almost commmutative and $n\geq 3$. 

In this paper we define odd dimensional unitary groups $\U_{2n+1}(R,\Delta)$. These groups are isomorphic to Petrov's odd hyperbolic unitary groups such that $V_0=R$, see \cite{petrov}. They contain as special cases the groups $\GL_{2n+1}(R)$ where $R$ is any ring, the classical groups $\Ort_{2n+1}(R)$ and $\Sp_{2n+1}(R)$ where $R$ is any commutative ring and further all even dimensional unitary groups $\U_{2n}(R,\Lambda)$ (cf. example \ref{18}) and all of the first author's odd dimensional unitary groups $\U_{2n+1}(R,\Lambda)$ when the Hermitian form associated to $\U_{2n+1}(R,\Delta)$ is even. In fact, the groups $\U_{2n+1}(R,\Delta)$ are introduced precisely to eliminate the necessity of assuming that the Hermitian form associated to $\U_{2n+1}$ is even. However, we define the groups $\U_{2n+1}(R,\Delta)$ differently than Petrov does, namely as automorphism groups of doubly parametrized forms. These forms were introduced already in a paper \cite[Section 4]{bak-morimoto} of M. Morimoto and the first author, in the context of surgery theory in order to extend a quadratic form defined with respect to a classical form parameter on an even submodule of a noneven Hermitian module, to a quadratic form defined across the entire Hermitian module. This extension requires introducing a second form parameter. In the current paper we start with a (not necessarily even) Hermitian form on a $2n+1$ dimensional free module $R\oplus R^{2n}$ such that on $R^{2n}$, the Hermitian form is free hyperbolic (and therefore even) and is equipped with the standard classical quadratic form for this situation, taking values in $R/\Lambda$ where $\Lambda$ is a classical form parameter. In order to extend the standard quadratic form across the whole Hermitian module, namely across $R\oplus R^{2n}$, we take as a second form parameter an odd form parameter $\Delta$ of V. Petrov \cite{petrov}. $\Delta$ is a normal subgroup of the Heisenberg group $\h$ and the extended quadratic form takes values in $\h/\Delta$. The choice of $\Delta$ is not unique, but there is always a canonical injective map of $R/\Lambda$ into $\h/\Delta$, which allows us to extend the classical quadratic form on $R^{2n}$ to one on $R\oplus R^{2n}$. The pair $(R,\Delta)$ is called a Hermitian form ring. The automorphism group of the Hermitian module and extended quadratic form above is isomorphic to the Petrov hyperbolic unitary group such that his $V_0 = R$, see Remark \ref{17}(c). The main difference between a classical quadratic form and an extended quadratic form is that the classical trace condition for a classical quadratic form is replaced by a generalized trace condition. The deviation of a quadratic form from being additive is measured the same way in both the classical and extended situations, namely by the Hermitian form. Taking this point of view, the arguments of the current paper for $\U_{2n+1}(R,\Delta)$ run parallel to and are an extension of those in \cite{preusser_thesis} or \cite{preusser} for the case of the even dimensional unitary group $\U_{2n}(R,\Lambda)$.  

The main result of the current paper is the following: If $H$ is a subgroup of an odd dimensional unitary group $\U_{2n+1}(R,\Delta)$, then
\begin{align*}
H \text{ is E-normal }\Leftrightarrow \exists!\text{ odd form ideal } (I,\Omega):\EU_{2n+1}((R,\Delta),(I,\Omega)) \subseteq H\subseteq \CU_{2n+1}((R,\Delta),(I,\Omega))\tag{1.3}
\end{align*}
provided $R$ is semilocal or quasifinite (i.e. a direct limit of almost commutative rings) and $n \geq 3$. In (1.3), $\EU_{2n+1}((R,\Delta),(I,\Omega))$ denotes the relative elementary subgroup of level $(I,\Omega)$ and $\CU_{2n+1}((R,\Delta),(I,\Omega))$ denotes the full congruence subgroup of level $(I,\Omega)$. Further we investigate the action of $\U_{2n+1}(R,\Delta)$ on the set of all E-normal subgroups by conjugation.

The main result above ((1.3) for semilocal or quasifinite rings $R$ and $n \geq 3$) yields a description of the E-normal structure of $\GL_m(R)$ but only for $m\geq 6$, while Vaserstein in \cite{vaserstein} considered the more general case $m\geq 3$. For results on the case $m=2$, see \cite{costa-keller}. The main result above also yields a description of the E-normal structure of the classical groups $\Ort_{2n+1}(R)$, namely those of Chevalley type $B$, of rank $n\geq 3$ over arbitrary commutative rings, due to Abe \cite{abe}. The Hermitian form associated to these groups is even symmetric. There are also results on the E-normal structure of $\Sp_{4}(R)$ (cf. \cite{costa-keller_2}) which are not covered by the main result of our paper.

The authors expect that the results of the current paper can be generalized to include Petrov's odd hyperbolic unitary groups and beyond, but details will probably be even more complicated than in this paper. See comments following Definition \ref{8}.

The rest of the paper is organized as follows. In Section \ref{sec2} we recall some standard notation
which will be used throughout the paper. In Section \ref{sec3} we define odd dimensional unitary groups and some important subgroups. In Section \ref{sec4} we investigate the action of conjugation on important subgroups. The main result of this section is Theorem \ref{47}. At the end of the section, we specialize to the cases $R$ is semilocal or quasifinite and determine the effect of conjugation on the level of an E-normal subgroup. In Section \ref{sec6}, we prove the very important Extraction Theorem. In the last section we prove the main result (1.3), first for semilocal rings, then for certain Noetherian rings and then for quasifinite rings. 
\section{Notation}\label{sec2}
Let $G$ be a group and $H,K$ be subsets of $G$. The subgroup of $G$ generated by $H$ is denoted by $\langle H\rangle$. If $g,h\in G$, let $^hg:=hgh^{-1}$, $g^h:=h^{-1}gh$ and $[g,h]:=ghg^{-1}h^{-1}$. Set $^KH:=\langle \{^kh\mid h\in H, k\in K\}\rangle$ and $H^K:=\langle \{h^k\mid h\in H, k\in K\}\rangle$. Analogously define $[H,K]$ and $HK$. Instead of $^K\{g\}$ we write $^Kg$ (analogously we write $g^K$ instead of $\{g\}^K$, $^gH$ instead of $^{\{g\}}H $, $[g,K]$ instead of $[\{g\},K]$ and so forth).\\

In this paper, $\mathbb{N}=\{1,2,\dots\}$ denotes the set of all natural numbers. The word "ring" will always mean associative ring with $1\neq 0$, "ideal" will mean two-sided ideal. If $R$ is a ring and $m,n\in \mathbb{N}$, then the set of all invertible elements in $R$ is denoted by $R^*$ and the set of all $m\times n$ matrices with entries in $R$ is denoted by $\M_{m\times n}(R)$. If $a\in \M_{m\times n}(R)$, let $a_{ij}\in R$ denote the element in the $(i,j)$-th position. Let $a^t$ denote the transpose of $a$, thus $(a^t)_{ij}=a_{ji}$. In particular if $a$ denotes a column, resp. row, vector with coefficients in $R$, then $a^t$ is a row, resp. column, vector with coefficients in $R$. Denote the $i$-th row of $a$ by $a_{i*}$ and the $j$-th column of $a$ by $a_{*j}$. We set $\M_n(R):=\M_{n\times n}(R)$. The identity matrix in $\M_n(R)$ is denoted by $e$ or $e^{n\times n}$ and the matrix with a $1$ at position $(i,j)$ and zeros elsewhere is denoted by $e^{ij}$. If $a\in \M_n(R)$ is invertible, the entry 
of $a^{-1}$ at position $(i,j)$ is denoted by $a'_{ij}$, the $i$-th row of $a^{-1}$ by $a'_{i*}$ and the $j$-th column of $a^{-1}$ by $a'_{*j}$. Further we denote by $^n\!R$ the set of all rows $v=(v_1,\dots,v_n)$ with entries in $R$ and by $R^n$ the set of all columns $u=(u_1,\dots,u_n)^t$ with entries in $R$.
\section{Odd dimensional unitary groups}\label{sec3}
\subsection{Hermitian form rings and odd form ideals}\label{sec 3.1}
First we recall the definition of a ring with involution with symmetry and introduce the notion of a Hermitian ring.
\begin{Definition}\label{1}
Let $R$ be a ring and 
\begin{align*}
\bar{}:R&\rightarrow R\\
x&\mapsto \bar{x}
\end{align*}
an anti-isomorphism of $R$ (i.e. $\bar{}~$ is bijective, $\overline{x+y}=\bar x+\bar y$, $\overline{xy}=\bar y\bar x$ for any $x,y\in R$ and $\bar 1=1$). Further let $\lambda\in R$ such that $\bar{\bar x}=\lambda x\bar\lambda$ for any $x\in R$. Then $\lambda$ is called a {\it symmetry} for $~\bar{}~$, the pair $(~\bar{}~,\lambda)$ an {\it involution with symmetry} and the triple $(R,~\bar{}~,\lambda)$ a {\it ring with involution with symmetry} or a {\it ring with involution and symmetry}. A subset $A\subseteq R$ is called {\it involution invariant} iff $\bar x\in A$ for any $x\in A$. We call a quadruple $(R,~\bar{}~,\lambda,\mu )$ where $(R,~\bar{}~,\lambda)$ is a ring with involution with symmetry and $\mu \in R$ such that $\mu =\bar\mu \lambda$ a {\it Hermitian ring}.
\end{Definition}
\begin{Remark}\label{2}
Let $(R,~\bar{}~,\lambda,\mu )$ be a Hermitian ring.
\begin{enumerate}[(a)]
\item It is easy to show that $\bar\lambda=\lambda^{-1}$.
\item If $\lambda\in \Center(R)$ then one gets the usual concept of an involution which is for instance used in \cite{preusser}.
\item The map
\begin{align*}
\b{}:R&\rightarrow R\\
x&\mapsto \b{x}:=\bar\lambda \bar x\lambda
\end{align*}
is the inverse map of $~\bar{}~$. One checks easily that $(R,~\b{}~,\b{$\lambda$},\b{$\mu $})$ is a Hermitian ring.
\item For any $n\in\mathbb{N}$, $(\M_n(R),{}^*,\lambda e,\mu  e)$ is a Hermitian ring where $\sigma^*:=\bar \sigma^t=((\bar \sigma_{ij})_{ij})^t=(\bar\sigma_{ji})_{ij}$ for any $\sigma\in \M_n(R)$. The inverse map of ${}^*$ is the map ${}_*$ which associates to each $\sigma\in \M_n(R)$ the matrix $\sigma_*:=\underline{\sigma}^t$. It follows from (c) that $(\M_n(R),{}_*,\b{$\lambda$} e,\b{$\mu $} e)$ is also a Hermitian ring.  
\end{enumerate}
\end{Remark}

Below we introduce the notion of an $R^{\bullet}$-module. It will play an important role in the rest of the paper.
\begin{Definition}\label{3}
If $R$ is a ring, let $R^\bullet$ denote the underlying set of the ring equipped with the  
multiplication of the ring, but not the addition of the ring. A {\it (right) $R^{\bullet}$-module} is a not 
necessarily abelian group $(G,\+)$ equipped with a map 
\begin{align*}
\circ: G\times R^{\bullet}&\rightarrow G\\
(a,x) &\mapsto a\circ x
\end{align*}
such that the following holds:
\begin{enumerate}[(1)]
\item $a\circ 0=0$ for any $a\in G$,
\item $a\circ 1=a$ for any $a\in G$,
\item $(a\circ x)\circ y=a\circ (xy)$ for any $a\in G$ and $x,y\in R$ and
\item $(a\+ b)\circ x=(a\circ x)\+(b\circ x)$ for any $a,b\in G$ and $x\in R$.
\end{enumerate}
A left $R^{\bullet}$-module is defined analogously. An $R$-module is canonically an $R^{\bullet}$-module, but not conversely. Let $G$ and $G'$ be $R^{\bullet}$-modules. A group homomorphism $f:G\rightarrow G'$ satisfying $f(a\circ x)=f(a)\circ x$ for any $a\in G$ and $x\in R$ is called a {\it  homomorphism of $R^{\bullet}$-modules}. A subgroup $H$ of $G$ which is $\circ$-stable (i.e. $a\circ x\in H$ for any $a\in H$ and $x\in R$) is called an {\it $R^{\bullet}$-submodule}. Further, if $A\subseteq G$ and $B\subseteq R$, we denote by $A\circ B$ the subgroup of $G$ generated by $\{a\circ b\mid a\in A,b\in B\}$.
\end{Definition}
\begin{Remark}\label{4}
We treat $\circ$ as an operator with higher priority than $\+$.
\end{Remark}

Next we define the Heisenberg group. It is an $R^{\bullet}$-module and will play an important role in the rest of the paper. Certain $R^{\bullet}$-submodules called odd form parameters can be embedded
according to Lemma \ref{23} (E1) as subgroups of odd dimensional unitary groups. The role of these subgroups is analogous to the role of subgroups of even dimensional unitary groups corresponding to embeddings of classical form parameters.
\begin{Definition}\label{5}
Let $(R,~\bar{}~,\lambda,\mu )$ be a Hermitian ring. Define the map.
\begin{align*}
\+: (R\times R)\times (R\times R) &\rightarrow R\times R\\
((x_1,y_1),(x_2,y_2))&\mapsto (x_1,y_1)\+ (x_2,y_2):=(x_1+x_2,y_1+y_2-\bar x_1\mu  x_2).
\end{align*}
Then $(R\times R,\+)$ is a group, which we call {\it the Heisenberg group} and denote by $\h$. Equipped with the map
\begin{align*}
\circ:(R\times R)\times R^{\bullet}&\rightarrow R\times R\\
((x,y),a)&\mapsto (x,y)\circ a:=(xa,\bar aya)
\end{align*}
$\h$ becomes an $R^{\bullet}$-module.
\end{Definition}

The following lemma is straightforward to check.
\begin{Lemma}\label{7}
Let $(R,~\bar{}~,\lambda,\mu )$ be a Hermitian ring. Denote the inverse of an element $(x,y)\in \h$ by $\minus(x,y)$. Then the following identities hold in $\h$:
\begin{enumerate}[(1)]
\item[\textnormal{(1)}] $\minus(x,y)=(-x,-y-\bar x\mu  x)$ for any $(x,y)\in \h$,
\item[\textnormal{(2)}] $(x_1,y_1)\minus(x_2,y_2)=(x_1-x_2,y_1-y_2+\overline{(x_1-x_2)}\mu  x_2)$ for any $(x_1,y_1),(x_2,y_2)\in \h$,
\item[\textnormal{(3)}] $[(x_1,y_1),(x_2,y_2)]=(0,\bar x_2\mu x_1-\bar x_1\mu x_2)$ for any $(x_1,y_1),(x_2,y_2)\in \h$ and
\item[\textnormal{(4)}] $\overset{.}{\sum\limits_{1\leq i \leq n}}(x_i,y_i)=(\sum\limits_{i=1}^n x_i,\sum\limits_{i=1}^n y_i-\sum\limits_{\substack{i,j=1,\\i<j}}^n \bar x_i\mu  x_j)$ for any $n\in\mathbb{N}$ and $(x_1,y_1),\hdots,(x_n,$ $y_n)\in \h$.
\end{enumerate}
\end{Lemma}
\begin{Definition}\label{8}
Let $(R,~\bar{}~,\lambda,\mu )$ be a Hermitian ring. Let $(R,+)$ have the $R^{\bullet}$-module structure defined by $x\circ a = \bar{a}xa$. Define the {\it trace map}
\begin{align*}
\tr:\h&\rightarrow R\\
(x,y)&\mapsto \bar x\mu  x+y+\bar y\lambda.
\end{align*}
One checks easily that $\tr$ is a homomorphism of $R^{\bullet}$-modules. Set \[\Delta_{min}:=\{(0,x-\overline{x}\lambda)\mid x\in R\}\] and \[\Delta_{max}:=\ker(\tr).\] An $R^{\bullet}$-submodule $\Delta$ of $\h$ lying between $\Delta_{min}$ and $\Delta_{max}$ is called an {\it odd form parameter of $(R,~\bar{}~,\lambda,\mu )$}. A pair $((R,~\bar{}~,\lambda,\mu ),\Delta)$ is called a {\it Hermitian form ring}. We shall usually abbreviate it by $(R,\Delta)$. Since $\Delta_{min}$ and $\Delta_{max}$ are $R^{\bullet}$-submodules of $\h$, they are respectively the smallest and largest odd form parameters. Since $\Delta_{min}$ contains the commutator subgroup of $\h$, the quotient $\h/\Delta$ is an $R^{\bullet}$-module whose underlying group is abelian. We set $\Lambda(\Delta) :=\{ x\in R\mid (0,x)\in \Delta\}$. It is clearly a classical form parameter of $R$ in the sense of \cite[Section 1(B) and Section 13]{bak_book} or \cite{bak_2} and the canonical map 
\begin{align*}
R &\rightarrow \h,\\
x&\mapsto(0,x)
\end{align*}
induces an injection $R/\Lambda(\Delta)\rightarrow \h/\Delta$ of $R^\bullet$-modules. The pair $(R,\Lambda(\Delta))$ is called a {\it form ring}.
\end{Definition}

{\it Morphisms} of the rings above are defined in the ususal way. A morphism $(R,~\bar{}~,\lambda)
\rightarrow (R',~\bar{}~',\lambda')$ of rings with involution with symmetry is a ring homomorphism
$R\rightarrow R'$ which preserves involution and symmetry. A morphism $(R,~\bar{}~,\lambda, \mu)
\rightarrow (R',~\bar{}~',\lambda',\mu')$ of Hermitian rings is a morphism $(R,~\bar{}~,\lambda)\rightarrow (R',~\bar{}~',\lambda')$ of  rings with involution and symmetry which takes the 
Hermitian element $\mu$ to the Hermitian element $\mu'$. Any morphism  $(R,~\bar{}~,\lambda, \mu)
\rightarrow (R',~\bar{}~',\lambda',\mu')$ of Hermitian rings induces a morphism $\h\rightarrow
\h'$ of the corresponding Heisenberg groups.  A morphism $(R,~\bar{}~,\lambda, \mu,\Delta)\rightarrow (R',~\bar{}~',\lambda',\mu',\Delta')$ of Hermitian form rings is a morphism $(R,~\bar{}~,\lambda, \mu)
\rightarrow (R',~\bar{}~',\lambda',\mu')$ of Hermitian rings such that the induced map $\h\rightarrow\h'$ preserves the odd form parameters. Our definitions of odd dimensional unitary groups will be functorial in Hermitian form rings. 

The constructions Hermitian ring and Hermitian form ring have analogs and generalizations in surgery theory where they are called prepared algebras and parameter algebras. If one wants to extend
the results of the current paper regarding forms on $R\oplus R^{2n}$,  which restrict to the free hyperbolic form on $R^{2n}$, to forms on $R^{k}\oplus R^{2n}$ which restrict to the free hyperbolic form on $R^{2n}$, it is useful to carry over some of the constructions used in surgery theory to the setting of the current paper. Doing this, one gets for the notion of Hermitian ring, a ring with involution with symmetry which is equipped with a set of Hermitian elements, rather than just one Hermitian element. In this situation there is a suitable notion of Heisenberg group and one gets a good notion of Hermitian form ring by equipping the Hermitian ring above with an $R^{\bullet}$-submodule of the generalized Heisenberg group. These submodules include the odd form parameters of V. Petrov. 

Next we define an odd form ideal of a Hermitian form ring. They will be of central importance throughout the rest of the paper.
\begin{Definition}\label{11}
Let $(R,\Delta)$ be a Hermitian form ring and $I$ an involution invariant ideal of $R$. Set $J(\Delta):=\{y\in R\mid\exists z\in R:(y,z)\in \Delta\}$ and $\tilde I:=\{x\in R\mid\overline{J(\Delta)}\mu  x\subseteq I\}$. Further set \[\Omega^I_{min}:=\{(0,x-\bar x\lambda)\mid x\in I\}\+ \Delta\circ I\] and \[\Omega^I_{max}:=\Delta\cap (\tilde I\times  I).\]
An $R^{\bullet}$-submodule $\Omega$ of $\h$ lying between $\Omega^I_{min}$ and $\Omega^I_{max}$ is called a {\it relative odd form parameter for} $I$ or a {\it relative odd form parameter of level $I$}. Since $\Omega^I_{min}$ and $\Omega^I_{max}$ are $R^{\bullet}$-submodules of $\h$, they are respectively the smallest and the largest relative odd form parameters for $I$. If $\Omega$ is a relative odd form parameter for $I$, then $(I,\Omega)$ is called an {\it odd form ideal of $(R,\Delta)$}.
\end{Definition}
\begin{Remark}\label{12}
Let $\Omega$ be a relative odd form parameter for $I$.
\begin{enumerate}[(a)]
\item Obviously $\tilde I$ and $J(\Delta)$ are right ideals of $R$. Further $I\subseteq \tilde I$. Set $J(\Omega):=\{y\in R\mid\exists z\in R:(y,z)\in \Omega\}$. Then $J(\Omega)$ is a right ideal of $R$ and 
\[J(\Delta)I\subseteq J(\Omega)\subseteq J(\Delta)\cap \tilde I.\]
\item One checks easily that $\Omega$ is a normal subgroup of $\Delta$. Further $\Omega^I_{max}/\Omega$ is a right $R$-module (not just an $R^{\bullet}$-module) with scalar multiplication $(a\+\Omega)x=(a\circ x)\+\Omega$.
\item $\Gamma(\Omega):=\{x\in R\mid(0,x)\in\Omega\}$ is a classical relative form parameter for $I$ of the form ring $(R,\Lambda(\Delta))$ (cf. \cite{bak_2}, \cite{preusser} or \cite{shchegolev}). 
\end{enumerate}
\end{Remark}
\begin{Definition}\label{13}
Let $(R,\Delta)$ be a Hermitian form ring. Further let $Y\subseteq R$ and $Z\subseteq \Delta$ be subsets. Then we denote by $I(Y)$ the ideal of $R$ generated by $Y\cup\bar Y$. It is called the {\it involution invariant ideal defined by $Y$}. Further set $\Omega(Y):=\Omega^{I(Y)}_{min}$. It is called the {\it relative odd form parameter defined by $Y$} and $(I(Y),\Omega(Y))$ is called the {\it odd form ideal defined by $Y$}.  Set $Z_1:=\{x\in R\mid\exists y\in R: (x,y)\in Z\}$, $Z_2:=\{y\in R\mid\exists x\in R: (x,y)\in Z\}$ and $Z':=\overline{J(\Delta)}\mu  Z_1\cup Z_2$. We denote by $I(Z)$ the ideal of $R$ generated by $Z'\cup \overline{Z'}$. It is called the {\it involution invariant ideal defined by $Z$}. Further set $\Omega(Z):=\Omega^{I(Z)}_{min}\+ Z\circ R$. $\Omega(Z)$ is called the {\it relative odd form parameter defined by $Z$}. One checks easily that $(I(Z),\Omega(Z))$ is an odd form ideal of $(R,\Delta)$. It is called the {\it odd form ideal defined by $Z$}.
\end{Definition}
\begin{Remark}\label{14}
If $x\in R$ or $x\in\Delta$, we sometimes write $(I(x),\Omega(x))$ instead of $(I(\{x\}),\Omega(\{x\}))$.
\end{Remark}
\subsection{Odd dimensional unitary groups}\label{3.2}
Let $(R,\Delta)$ be a Hermitian form ring and $n\in \mathbb{N}$. Fix the ordered basis $(e_1,\dots,e_n,e_0,e_{-n},\dots,e_{-1})$ of the right $R$-module $M:=R^{2n+1}$ where $e_i$ 
is the column vector whose $i$-th entry $(1\leq i\leq -1)$ is $1$ and whose other entries are $0$. If $u\in M$, then we call $(u_1,\dots,u_n,u_{-n},\dots,u_{-1})^t\in R^{2n}$ the {\it hyperbolic part of $u$} and denote it by $u_{hb}$. Define the maps
\begin{align*}
b:M\times M&\rightarrow R\\
(u,v)&\mapsto u^*\begin{pmatrix} 0& 0 & p\\0&\mu &0\\ p\lambda &0 &0 \end{pmatrix}v=\sum\limits_{i=1}^{n}\bar u_i v_{-i}+\bar u_0\mu  v_0+\sum\limits_{i=-n}^{-1}\bar u_{i}\lambda v_{-i}
\end{align*}
and 
\begin{align*}
q:M&\rightarrow \h\\
u&\mapsto (q_1(u),q_2(u)):=(u_0,u_{hb}^*\begin{pmatrix} 0&p\\0&0 \end{pmatrix}u_{hb})=(u_0,\sum\limits_{i=1}^{n}\bar u_i u_{-i})
\end{align*}
where $u^*=\bar u^t$ and $p\in \M_n(R)$ is the matrix with ones on the second diagonal and zeros elsewhere.
\begin{Lemma}\label{15}
~\\
\vspace{-0.6cm}
\begin{enumerate}[(1)]
\item $b$ is a $\lambda$-Hermitian form, i.e. $b$ is biadditive, $b(ux,vy)=\bar x b(u,v) y~\forall u,v\in M,x,y\in R$ and $b(u,v)=\overline{b(v,u)}\lambda~\forall u,v\in M$.
\item $q(ux)=q(u)\circ x~\forall u\in M, x\in R$, $q(u+v)\equiv q(u)\+ q(v)\+(0,b(u,v))\bmod \Delta_{min}~\forall u,$ $v\in M$ and $tr(q(u))=b(u,u)~\forall u\in M$.
\end{enumerate}
\end{Lemma}
\begin{proof}
{Straightforward computation.}
\end{proof}
\begin{Definition}\label{16}
The group
\[\U_{2n+1}(R,\Delta):=\{\sigma\in \GL_{2n+1}(R)\mid b(\sigma u,\sigma v)=b(u,v)~\forall u,v\in M \text{ and }q(\sigma u)\equiv q(u)\bmod \Delta~\forall u\in M\}\]
is called the {\it odd dimensional unitary group}.
\end{Definition}
\begin{Remark}\label{17}
$~$\\
\vspace{-0.6cm}
\begin{enumerate}[(a)]
\item One checks easily that $b(\sigma u,\sigma u)=b(u,u)\Leftrightarrow q(\sigma u)\equiv q(u)\bmod \Delta_{max}$ for any $\sigma\in \M_{2n+1}(R)$ and $u\in M$. Hence 
\begin{align*}\U_{2n+1}(R,\Delta_{max})
=\{\sigma\in \GL_{2n+1}(R)\mid b(\sigma u,\sigma v)=b(u,v)~\forall u,v\in M\}.
\end{align*}
\item 
Set $\Lambda:=\Lambda(\Delta)$ and let $n\in \mathbb{N}$. There is an embedding of $\U_{2n}(R,\Lambda)$ (as defined in \cite{preusser} or \cite{shchegolev}) into $\U_{2n+1}(R,\Delta)$ namely
\begin{align*}
\phi_{2n}^{2n+1}: \U_{2n}(R,\Lambda)&\rightarrow \U_{2n+1}(R,\Delta)\\
\begin{pmatrix}A&B\\C&D\end{pmatrix}&\mapsto \begin{pmatrix}A&0&B\\0&1&0\\C&0&D\end{pmatrix}
\end{align*}
where $A,B,C,D\in \M_n(R)$. There is also an embedding of $\U_{2n+1}(R,\Delta)$ into $\U_{2n+3}(R,\Delta)$ namely
\begin{align*}
\phi^{2n+3}_{2n+1}: \U_{2n+1}(R,\Delta)&\rightarrow \U_{2n+3}(R,\Delta)\\
\begin{pmatrix}A&t&B\\v&z&w\\C&u&D\end{pmatrix}&\mapsto \begin{pmatrix}A&0&t&0&B\\0&1&0&0&0\\v&0&z&0&w\\0&0&0&1&0\\C&0&u&0&D\end{pmatrix}
\end{align*}
where $A,B,C,D\in \M_{n}(R)$, $t,u\in \M_{n \times 1}(R)$, $v,w\in \M_{1 \times n}(R)$ and $z\in R$.
\item The odd dimensional unitary groups defined in this article are isomorphic to Petrov's odd hyperbolic unitary groups with $V_0=R$ (see \cite{petrov}). Namely $\U_{2n+1}(R,\Delta)$ is isomorphic to Petrov's group $\U_{2l}(R,\mathfrak{L})$ where Petrov's pseudoinvolution $~\hat{}~$ is defined by $\hat x=-\bar x\lambda$, $V_0=R$, $B_0(a,b)=\hat a\hat 1^{-1}\mu  b$, $\mathfrak{L}=\{(x,y)\in R\times R\mid(x,-y)\in\Delta\}$ and $l=n$.
\end{enumerate}
\end{Remark}
\begin{Example}\label{18}
~\\
\vspace{-0.6cm}
\begin{enumerate}[(1)]
\item Let $(R,\Lambda)$ be a classical form ring. Choose $\mu$ arbitary and set $\Delta:=\{0\}\times\Lambda$. Then $\U_{2n+1}(R,\Delta)\cong \U_{2n}(R,\Lambda)$. In particular the classical groups $\Ort_{2n}(R)$ and $\Sp_{2n}(R)$ where $R$ is a commutative ring and $\GL_{2n}(R)$ where $R$ is an arbitrary ring are examples of odd dimensional unitary groups.
\item Let $S$ be a ring and $S^{op}$ its opposite ring. If $R=S\times S^{op}$ is the Hermitian ring such that $\overline{(x,y)}=(y,x)$, $\lambda=1$, and $\mu=1$ and if $\Delta=\Delta_{max}$, then $\U_{2n+1}(R,\Delta)\cong \GL_{2n+1}(S)$.
\item If $R$ is commutative, $\bar x=x$, $\lambda=1$, $\mu=2$ and $\Delta=\{(x,-x^2)\mid x\in R\}$, then $ \U_{2n+1}(R,\Delta)=\Ort_{2n+1}(R)$.
\item If $R$ is commutative, $\bar x=x$, $\lambda=-1$, $\mu=0$ and $\Delta=R\times R$, then $\U_{2n+1}(R,\Delta)$ is the automorphism group of the bilinear form represented by
\[\begin{pmatrix}0&0&p\\0&0&0\\-p&0&0\end{pmatrix}.\]
It makes sense to denote this group by $\Sp_{2n+1}(R)$ since it is the automorphism group of a skew-symmetric bilinear form of maximal possible rank (cf. \cite{proctor}).
\end{enumerate}
\end{Example}
\begin{Definition}\label{19}
Define $\Theta_+:=\{1,\dots,n\}$, $\Theta_-:=\{-n,\dots,-1\}$, $\Theta:=\Theta_+\cup\Theta_-\cup\{0\}$, $\Theta_{hb}:=\Theta\setminus \{0\}$ and the map 
\begin{align*}\epsilon:\Theta_{hb} &\rightarrow\{\pm 1\}\\
i&\mapsto\begin{cases} 1, &\mbox{if } i\in\Theta_+, \\ 
-1, & \mbox{if } i\in\Theta_-. \end{cases}
\end{align*}
\end{Definition}
\begin{Lemma}\label{20}
Let $\sigma\in \GL_{2n+1}(R)$. Then $\sigma\in \U_{2n+1}(R,\Delta)$ if and only if the conditions \textnormal{(1)} and \textnormal{(2)} below hold. 
\begin{enumerate}[(1)]
\item[\textnormal{(1)}] \begin{align*}
       \sigma'_{ij}&=\lambda^{-(\epsilon(i)+1)/2}\bar\sigma_{-j,-i}\lambda^{(\epsilon(j)+1)/2}~\forall i,j\in\Theta_{hb},\\
       \mu \sigma'_{0j}&=\bar\sigma_{-j,0}\lambda^{(\epsilon(j)+1)/2}~\forall j\in\Theta_{hb},\\
       \sigma'_{i0}&=\lambda^{-(\epsilon(i)+1)/2}\bar\sigma_{0,-i}\mu ~\forall i\in\Theta_{hb} \text { and}\\
       \mu \sigma'_{00}&=\bar\sigma_{00}\mu .
      \end{align*}
\item[\textnormal{(2)}] \begin{align*}
q(\sigma_{*j})\equiv (\delta_{0j},0)\bmod \Delta ~\forall j\in \Theta.
\end{align*} 
\end{enumerate}
\end{Lemma}
\begin{proof}
~\\
``$\Rightarrow$'':\\
Assume that $\sigma=\begin{pmatrix} A&B&C\\D&E&F\\G&H&I\end{pmatrix}\in \U_{2n+1}(R,\Delta)$ where $A,C,G,I\in \M_n(R)$, $B,H\in R^{n}$, $D,F\in{}^n\!R$ and $E\in R$. Then $b(\sigma u, \sigma v)=b(u,v)~\forall u,v \in M$ and $q(\sigma u)\equiv q(v)\bmod \Delta~\forall u\in M$. Let $\sigma^{-1}=\begin{pmatrix} A'&B'&C'\\D'&E'&F'\\G'&H'&I'\end{pmatrix}$ where $A'$ has the same size as $A$, $B'$ has the same 
size as $B$ and so on. One checks easily that $b(\sigma u, \sigma v)=b(u,v)~\forall u,v \in M$ is equivalent to
\begin{align*}
\begin{pmatrix} A'&B'&C'\\\mu D'&\mu E'&\mu F'\\ G'&H'&I'\end{pmatrix}=\begin{pmatrix} \bar\lambda p I^*p\lambda&\bar\lambda p F^*\mu &\bar\lambda p C^*p\\ H^*p\lambda&\E^*\mu &B^*p\\ p G^*p\lambda&pD^*\mu &pA^*p\end{pmatrix}
\end{align*}
which is clearly equivalent to condition (1). Further $q(\sigma u)\equiv q(u)\bmod \Delta~\forall u\in M$ implies condition (2), since $q(e_j)=(0,0)$ for any $j\neq 0$ and $q(e_0)=(1,0)$.\\
\\
``$\Leftarrow$'':\\
Assume that (1) and (2) hold. As shown under ``$\Rightarrow$'', (1) is equivalent to $b(\sigma u, \sigma v)=b(u,v)~\forall u,v \in M$. It remains to show that $q(\sigma u)\equiv q(u)\bmod  \Delta~\forall u\in M$. But this follows from (1), (2) and Lemma \ref{15}.
\end{proof}
\subsection{The elementary subgroup}\label{sec3.3}
We introduce the following notation. In Definition \ref{5} we defined an $R^{\bullet}$-module structure on $\h$. Let $(R,~\b{}~,\b{$\lambda$},\b{$\mu $})$ be the Hermitian ring defined in Remark \ref{2}(c). Let $\underline{\h}$ denote the Heisenberg group corresponding to $(R,~\b{}~,\b{$\lambda$},\b{$\mu $})$. The underlying set of both $\h$ and $\underline{\h}$ is $R\times R$. If we replace the Hermitian ring $(R,~\bar~,\lambda,\mu)$ by the Hermitian ring $(R,~\b{}~,\b{$\lambda$},\b{$\mu $})$, then we get another $R^{\bullet}$-module structure on $R\times R$. We denote the group operation (resp. scalar multiplication) defined by $(R,~\bar{}~,\lambda,\mu )$ on $R\times R$ by $\+_1$ (resp. $\circ_1$) and the group operation (resp. scalar multiplication) defined by $(R,~\b{}~,\b{$\lambda$},\b{$\mu $})$ on $R\times R$ by $\+_{-1}$ (resp. $\circ_{-1}$). Further we set
 $\Delta^{-1}:=\{(x,y)\in R\times R\mid(x,\bar y)\in \Delta\}$. One checks easily that  $\Delta^{-1}$ is an odd form parameter of $(R,~\b{}~,\b{$\lambda$},\b{$\mu $})$. Analogously, if $(I,\Omega)$ is an odd form ideal of $(R,\Delta)$, we set $\Omega^{-1}:=\{(x,y)\in R\times R\mid(x,\bar y)\in \Omega\}$. One checks easily that $\Omega^{-1}$ is an relative odd form parameter for $I$ of the Hermitian form ring $(R,\Delta^{-1})$. 
 
If $k,l\in \Omega$, let $e^{kl}$ denote the $(2n+1)\times(2n+1)$ matrix with $1$ in the $(k,l)$-th position and $0$ in all other positions.
\begin{Definition}\label{21}
If $i,j\in \Theta_{hb}$, $i\neq\pm j$ and $x\in R$, the element  \[T_{ij}(x):=e+xe^{ij}-\lambda^{(\epsilon(j)-1)/2}\bar x\lambda^{(1-\epsilon(i))/2}e^{-j,-i}\] of $\U_{2n+1}(R,\Delta)$ is called an {\it (elementary) short root matrix}. 
If $i\in \Theta_{hb}$ and $(x,y)\in \Delta^{-\epsilon(i)}$, the element \[T_{i}(x,y):=e+xe^{0,-i}-\lambda^{-(1+\epsilon(i))/2}\bar x\mu e^{i0}+ye^{i,-i}\] of $\U_{2n+1}(R,\Delta)$ is called an {\it (elementary) extra short root matrix}. The extra short root matrices of the kind \[T_{i}(0,y)=e+ye^{i,-i}\] are called {\it (elementary) long root matrices}. If an element of $\U_{2n+1}(R,\Delta)$ is a short or extra short root matrix, then it is called {\it elementary matrix}. The subgroup of $\U_{2n+1}(R,\Delta)$ generated by all elementary matrices is called the {\it elementary subgroup} and is denoted by $\EU_{2n+1}(R,\Delta)$. 
\end{Definition}
\begin{Remark}\label{22}
Set $\Lambda:=\Lambda(\Delta)$. One checks easily that $\phi_{2n}^{2n+1}(\EU_{2n}(R,\Lambda))\subseteq \EU_{2n+1}(R,\Delta)$ and $\phi_{2n+1}^{2n+3}(\EU_{2n+1}($ $R,\Delta))\subseteq \EU_{2n+3}(R,\Delta)$.
\end{Remark}
\begin{Lemma}\label{23}
The following relations hold for elementary matrices.
\begin{align*}
&T_{ij}(x)=T_{-j,-i}(-\lambda^{(\epsilon(j)-1)/2}\bar x\lambda^{(1-\epsilon(i))/2}), \tag{S1}\\
&T_{ij}(x)T_{ij}(y)=T_{ij}(x+y), \tag{S2}\\
&[T_{ij}(x),T_{kl}(y)]=e \text{ if } k\neq j,-i \text{ and } l\neq i,-j, \tag{S3}\\
&[T_{ij}(x),T_{jk}(y)]=T_{ik}(xy) \text{ if } i\neq\pm k, \tag{S4}\\
&[T_{ij}(x),T_{j,-i}(y)]=T_{i}(0,xy-\lambda^{(-1-\epsilon(i))/2}\bar y\bar x\lambda^{(1-\epsilon(i))/2}), \tag{S5}\\
&T_{i}(x_1,y_1)T_{i}(x_2,y_2)=T_{i}((x_1,y_1)\+_{-\epsilon(i)}(x_2,y_2)), \tag{E1}\\
&[T_{i}(x_1,y_1),T_{j}(x_2,y_2)]=T_{i,-j}(-\lambda^{-(1+\epsilon(i))/2}\bar x_1\mu x_2) \text{ if } i\neq\pm j, \tag{E2}\\
&[T_{i}(x_1,y_1),T_{i}(x_2,y_2)]=T_{i}(0,-\lambda^{-(1+\epsilon(i))/2}(\bar x_1\mu x_2-\bar x_2\mu x_1)), \tag{E3}\\
&[T_{ij}(x),T_{k}(y,z)]=e \text{ if } k\neq j,-i \text{ and} \tag{SE1}\\
&[T_{ij}(x),T_{j}(y,z)]=T_{j,-i}(z\lambda^{(\epsilon(j)-1)/2}\bar x\lambda^{(1-\epsilon(i))/2})T_{i}(y\lambda^{(\epsilon(j)-1)/2}\bar x\lambda^{(1-\epsilon(i))/2},xz\lambda^{(\epsilon(j)-1)/2}\bar x\lambda^{(1-\epsilon(i))/2}).\tag{SE2}\\
\end{align*}
\end{Lemma}
\begin{proof}
Straightforward computation.
\end{proof}
\begin{Remark}
It follows from Lemma \ref{23} (E1) that for each $i\neq 0$, there is an embedding
\begin{align*}
T_i:\Delta^{-\epsilon(i)}&\rightarrow \U_{2n+1}(R,\Delta)\\
(x,y)&\mapsto T_i(x,y).
\end{align*}
For $\epsilon(i)=1$ (resp. $-1$), the groups $T_i(\Delta^{-\epsilon(i)})$ are upper (resp. lower) triangular matrices in $\U_{2n+1}(R,\Delta)$. Such triangular matrix groups are historically called Heisenberg groups. See \cite[Section 7($1^0$)]{bak-vavilov}. Thus the groups $\h^1=\h$ and $\h^{-1}=\underline{\h}$ which we are calling Heisenberg groups, actually contain isomorphic copies of all the groups historically called Heisenberg groups, although they themselves historically are not Heisenberg groups and in particular cannot in general be embedded in unitary groups.
\end{Remark}
\begin{Definition}\label{24}
Let $i,j\in\Theta_{hb}$ such that $i\neq\pm j$. Define
\begin{align*}
P_{ij}:=&e-e^{ii}-e^{jj}-e^{-i,-i}-e^{-j,-j}+e^{ij}-e^{ji}+\lambda^{(\epsilon(i)-\epsilon(j))/2}e^{-i,-j}-\lambda^{(\epsilon(j)-\epsilon(i))/2}e^{-j,-i}\\
=&T_{ij}(1)T_{ji}(-1)T_{ij}(1)\in \EU_{2n+1}(R,\Delta).                                                                                                                                                                                                                                                                                                                                                                                                                                                                                             
\end{align*}
It is easy to show that $(P_{ij})^{-1}=P_{ji}$. 
\end{Definition}
\begin{Lemma}\label{25}
Let $i,j,k\in\Theta_{hb}$ such that $i\neq \pm j$ and $k\neq \pm i,\pm j$. Let $x\in R$ and $(y,z)\in\Delta^{-\epsilon(i)}$. Then 
\begin{enumerate}[(1)]
\item $^{P_{ki}}T_{ij}(x)=T_{kj}(x)$,
\item $^{P_{kj}}T_{ij}(x)=T_{ik}(x)$ and
\item $^{P_{-k,-i}}T_{i}(y,z)=T_{k}(y,\lambda^{(\epsilon(i)-\epsilon(k))/2}z)$.
\end{enumerate}
\end{Lemma}
\begin{proof}
Straightforward.
\end{proof}
\subsection{Relative elementary subgroups}
\begin{Definition}
Let $(I,\Omega)$ denote an odd form ideal of $(R,\Delta)$. A short root matrix $T_{ij}(x)$ is called {\it $(I,\Omega)$-elementary} if $x\in I$. An extra short root matrix $T_{i}(x,y)$ is called {\it $(I,\Omega)$-elementary} if $(x,y)\in \Omega^{-\epsilon(i)}$. If an element of $\U_{2n+1}(R,\Delta)$ is an $(I,\Omega)$-elementary short or extra short root matrix, then it is called {\it $(I,\Omega)$-elementary matrix}. The subgroup $\EU_{2n+1}(I,\Omega)$ of $\EU_{2n+1}(R,\Delta)$ generated by the $(I,\Omega)$-elementary matrices is called the {\it preelementary subgroup of level $(I,\Omega)$}. Its normal closure $\EU_{2n+1}((R,\Delta),(I,\Omega))$ in $\EU_{2n+1}(R,\Delta)$ is called {\it elementary subgroup of level $(I,\Omega)$}.
\end{Definition}
\section{Congruence subgroups}\label{sec4}
In this section $(R,\Delta)$ denotes a Hermitian form ring and $n$ a natural number. $M$ denotes the free module $R^{2n+1}$ with the ordered basis $(e_1,\dots,e_n,e_0,e_{-n},\dots,e_{-1})$ of Subsection \ref{3.2}.
\subsection{Principal, normalized principal and full congruence subgroups}
In this subsection $(I,\Omega)$ denotes an odd form ideal of $(R,\Delta)$. If $\sigma\in \M_{2n+1}(R)$, we call the matrix $(\sigma_{ij})_{i,j\in\Theta_{hb}}\in \M_{2n}(R)$ the {\it hyperbolic part of $\sigma$} and denote it by $\sigma_{hb}$. Set $M(R,\Delta):=\{u\in M\mid u_0\in J(\Delta)\}$, $M(I):=\{u\in M\mid u_i\in I~\forall i\in\Theta_{hb}\}$ and $M(I,\Omega):=\{u\in M(I)\mid u_0\in J(\Omega)\}$. $M(R,\Delta)$, $M(I)$ and $M(I,\Omega)$ are clearly submodules of $M$. Note that if $\sigma\in \U_{2n+1}(R,\Delta)$, $u\in M(R,\Delta)$, $v\in M(I)$ and $w\in M(I,\Omega)$, then $\sigma u\in M(R,\Delta)$ and $\sigma v\in M(I)$, but it is not necessary that $\sigma w\in M(I,\Omega)$. Define 
\[I_0:=\{x\in R\mid xJ(\Delta)\subseteq I\}\]
and
\[\tilde I_0:=\{x\in R\mid\overline{J(\Delta)}\mu x\in I_0\}.\]
Then $I_0$ is a left ideal of $R$ and $\tilde I_0$ an additive subgroup of $R$. If $J(\Delta)=R$ then $I_0=I$ and $\tilde I_0= \tilde I$. Let $u,v\in M$ and $\sigma,\tau\in \U_{2n+1}((R,\Delta),(I,\Omega))$. We write $u\equiv v\bmod I,\tilde I$ (resp. $u\equiv v\bmod I_0,\tilde I_0$) if and only if $u_i\equiv v_i\bmod I~\forall i\in \Theta_{hb}$ and $u_0\equiv v_0\bmod \tilde I$ (resp. $u_i\equiv v_i\bmod I_0~\forall i\in \Theta_{hb}$ and $u_0\equiv v_0\bmod \tilde I_0$). We write $\sigma\equiv \tau\bmod I,\tilde I,I_0,\tilde I_0$ if and only if $\sigma_{*j}\equiv \tau_{*j}\bmod I,\tilde I~\forall j\in \Theta_{hb}$ and $\sigma_{*0}\equiv \tau_{*0}\bmod I_0,\tilde I_0$. Further if $w,x\in M^t:={}^{2n}\!R$, we write $w\equiv x\bmod I,I_0$ (resp. $w\equiv x\bmod \tilde I,\tilde I_0$) if and only if $w_j\equiv x_j\bmod I~\forall j\in \Theta_{hb}$ and $w_0\equiv x_0\bmod I_0$ (resp. $w_j\equiv x_j\bmod \tilde I~\forall j\in \Theta_{hb}$ and $w_0\equiv x_0\bmod \tilde I_0$).
\begin{Definition}\label{36}
The subgroup $\U_{2n+1}((R,\Delta),(I,\Omega)):=$
\[\{\sigma\in \U_{2n+1}(R,\Delta)\mid\sigma_{hb}\equiv e_{hb}\bmod  I\text{ and }q(\sigma u)\equiv q(u)\bmod \Omega~\forall u\in M(R,\Delta)\}\]
of $\U_{2n+1}(R,\Delta)$ is called {\it the principal congruence subgroup of level $(I,\Omega)$}. The subgroup $\NU_{2n+1}((R,\Delta),$ $(I,\Omega)): =$
\[ \Normalizer_{\U_{2n+1}(R,\Delta)}(\U_{2n+1}((R,\Delta),(I,\Omega)))\]
of $\U_{2n+1}(R,\Delta)$ is called the {\it normalized principal congruence subgroup of level $(I,\Omega)$}. The subgroup $\CU_{2n+1}((R,\Delta),(I,\Omega)):=$
\[ \{\sigma\in \NU_{2n+1}((R,\Delta),(I,\Omega))\mid [\sigma,\EU_{2n+1}(R,\Delta)]\subseteq \U_{2n+1}((R,\Delta),(I,\Omega))\}\]
of $\U_{2n+1}(R,\Delta)$ is called the {\it full congruence subgroup of level $(I,\Omega)$}.
\end{Definition}
\begin{Remark}\label{42}
In many interesting situations, $\NU_{2n+1}((R,\Delta),(I,\Omega))$ equals $\U_{2n+1}(R,\Delta)$, e.g. in the situations (1)-(3) in Example \ref{18}. Further the equality holds if $\Omega=\Omega^I_{max}$ or $J(\Omega)\subseteq I$ (true e.g. if $\Omega=\Omega^I_{min}$). But it can also happen that $\NU_{2n+1}((R,\Delta),(I,\Omega))\neq \U_{2n+1}(R,\Delta)$, see Example \ref{174}.
\end{Remark}
\subsection{The action of $\U_{2n+1}(R,\Delta)$ on odd form ideals} 
\begin{Lemma}\label{37}
Let $I$ denote an involution invariant ideal of $R$. Let $u,v\in M$ such that $u\in M(I)$ or $v\in M(I)$. Then 
\[q(u+v)\equiv q(u)\+ q(v)\+(0,b(u,v))\bmod \Omega^I_{\min}.\]
\end{Lemma}
\begin{proof}
A straightforward computation shows that
\[q(u+v)\minus(q(u)\+ q(v)\+(0,b(u,v)))=(0,\sum\limits_{i=1}^n\bar v_iu_{-i}-\overline{\sum\limits_{i=1}^n\bar v_iu_{-i}}\lambda)\in\Omega^I_{\min}.\]
\end{proof}
\begin{Lemma}\label{39}
Let $\sigma\in \U_{2n+1}(R,\Delta)$ and $(I,\Omega)$ denote an odd form ideal of $(R,\Delta)$. Then $\sigma\in \U_{2n+1}((R,\Delta),$ $(I,\Omega))$ if and only if the conditions \textnormal{(1)} and \textnormal{(2)} below hold. 
\begin{enumerate}[(1)]
\item[\textnormal{(1)}] $\sigma_{hb}\equiv e_{hb}\bmod  I$.
\item[\textnormal{(2)}] $q(\sigma_{*j})\in\Omega~\forall j\in \Theta_{hb}$ and $(q(\sigma_{*0})\minus (1,0))\circ x\in\Omega~\forall x\in J(\Delta)$.
\end{enumerate}
\end{Lemma}
\begin{proof}
~\\
``$\Rightarrow$'':\\
Assume that $\sigma\in \U_{2n+1}((R,\Delta),(I,\Omega))$. Then (1) holds and $q(\sigma u)\minus q(u)\in \Omega~\forall u\in M$. Clearly $q(\sigma u)\minus q(u)\in \Omega~\forall u\in M$ implies (2), since $q(e_j)=(0,0)$ for any $j\neq 0$ and $q(e_0)=(1,0)$.\\
\\
``$\Leftarrow$'':\\
Assume that (1) and (2) hold. We have to show that $q(\sigma u)\equiv q(u)\bmod  \Omega~\forall u\in M$. But this follows from (1), (2) and Lemma \ref{37}.
\end{proof}
\begin{Remark}\label{40}
Let $\sigma\in \U_{2n+1}(R,\Delta)$. Lemma \ref{39} implies that $\sigma\in \U_{2n+1}((R,\Delta),(I,\Omega^I_{max}))$ if and only if the conditions (1) and (2) below hold. 
\begin{enumerate}[(1)]
\item $\sigma_{hb}\equiv e_{hb}\bmod  I$.
\item $\sigma_{0*}\equiv e_0^t\bmod \tilde I,\tilde I_0$.
\end{enumerate}
It follows from (2) that if $\sigma\in \U_{2n+1}((R,\Delta),(I,\Omega^I_{max}))$, then $\sigma_{*0}\equiv e_0\bmod I_0,\tilde I_0$ and therefore $\sigma\equiv e\bmod I,\tilde I,I_0,\tilde I_0 $.
\end{Remark}
\begin{Definition}\label{999}
Let $(I,\Omega)$ be an odd form ideal and $\sigma\in \U_{2n+1}(R,\Delta)$. Then 
\begin{align*}
^{\sigma}\Omega:=&\{q(\sigma_{*0}x)\+(0,y)\mid (x,y)\in\Omega\}\+\Omega_{min}^I\\
=&\{(q(\sigma_{*0})\minus (1,0))\circ x\+(x,y)\mid (x,y)\in\Omega\}\+\Omega_{min}^I
\end{align*}
is called {\it the relative odd form parameter for $I$ defined by $\Omega$ and $\sigma$}.
\end{Definition}
\begin{Remark}
~\\
\vspace{-0.6cm}
\begin{enumerate}[(a)]
\item One checks easily that ${}^{\sigma}\Omega$ is a relative odd form parameter for $I$.
\item $^{\sigma}\Omega$ depends not only on $\sigma$ and $\Omega$ but also on $I$, although this is not expressed in the notation.
\end{enumerate}
\end{Remark}
\begin{Lemma}\label{38}
Let $I$ denote an involution invariant ideal, $\sigma\in \U_{2n+1}(R,\Delta)$ and $u\in M(I)$. Then \[q(\sigma u)\minus q(u)\equiv (q(\sigma_{*0})\minus (1,0))\circ u_0\bmod \Omega_{min}^I.\]
\end{Lemma}
\begin{proof}
The assertion of the lemma can be shown easily using Lemma \ref{37}.
\end{proof}
\begin{Lemma}\label{45}
Let $I$ denote an involution invariant ideal and $\FP(I)$ the set of all relative odd form parameters for $I$. Then the map 
\begin{align*}
\U_{2n+1}(R,\Delta)\times \FP(I)&\rightarrow \FP(I)\\
(\sigma,\Omega)&\mapsto {}^{\sigma}\Omega
\end{align*}
is a (left) group action.
\end{Lemma}
\begin{proof}
In Definition \ref{5} we have defined the Heisenberg group $\h$ whose underlying set is $R\times R$. Similarly one can define an $R^{\bullet}$-module structure on the set $M\times R$ (cf. \cite[p. 4753]{petrov}). The addition is given by 
\begin{align*}
\+: (M\times R)\times (M\times R) &\rightarrow M\times R\\
((u,x),(v,y))&\mapsto (u,x)\+ (v,y):=(u+v,x+y-b(u,v))
\end{align*}
and the scalar mutliplication by
\begin{align*}
\circ:(M\times R)\times R&\rightarrow M\times R\\
((u,x),a)&\mapsto (u,x)\circ a:=(ua,\bar axa).
\end{align*}
We call this $R^{\bullet}$-module the {\it big Heisenberg group} and denote it by $\h^*$. Set $\Delta^*:=\{(u,x)\in \h^*\mid q(u)\+(0,x)\in\Delta\}$, $J(\Delta^*):=\{u\in M\mid \exists x\in R:(u,x)\in \Delta^*\}$ and $\tilde I^*:=\{v\in M\mid b(u,v)\in I~\forall u\in J(\Delta^*)\}$. Further set \[(\Omega^I_{min})^*:=\{(0,x-\bar x\lambda)\mid  x\in I\}\+ \Delta^*\circ I\] and \[(\Omega^I_{max})^*:=\Delta^*\cap (\tilde I^*\times  I).\]
We call an $R^{\bullet}$-submodule $\Omega^*$ of $\h^*$ lying between $(\Omega^I_{min})^*$ and $(\Omega^I_{max})^*$ a {\it big relative odd form parameter for} $I$. Denote the set of all big relative odd form parameters for $I$ by $\FP^*(I)$. There is a one-to-one correspondence between $\FP(I)$ and $\FP^*(I)$ (compare \cite[p. 4758]{petrov}). Namely the maps
\begin{align*}
\hspace{2.8cm}f:\FP(I)&\rightarrow \FP^*(I)\\
\Omega&\mapsto f(\Omega):=\{(u,x)\mid u\in M(I), x\in R, q(u)\+(0,x)\in \Omega\}.
\end{align*}
and
\begin{align*}
f^{-1}:\FP^*(I)&\rightarrow \FP(I)\\
\Omega^*&\mapsto f^{-1}(\Omega^*):=\{q(u)\+(0,x)\mid (u,x)\in \Omega^*\}.
\end{align*}
are inverse to each other. Define the map
\begin{align*}
\psi:\U_{2n+1}(R,\Delta)\times \FP^*(I)&\rightarrow \FP^*(I)\\
(\sigma,\Omega^*)&\mapsto \psi((\sigma,\Omega^*)):={}^{\sigma}\Omega^*.
\end{align*}
where $^{\sigma}\Omega^*=\{(\sigma u,x)\mid (u,x)\in \Omega^*\}$.
Then clearly $\psi$ is a (left) group action.  Consider the commutative diagram
\xymatrixcolsep{5pc}
\xymatrixrowsep{5pc}
\[\xymatrix{
\U_{2n+1}(R,\Delta)\times \FP(I) \ar[d]_{(id,f)} \ar[r]^-{\phi} &\FP(I)\\
\U_{2n+1}(R,\Delta)\times \FP^*(I)\ar[r]^-{\psi} &\FP^*(I)\ar[u]_{f^{-1}}}\]
where $\phi=f^{-1}\circ \psi\circ (id,f)$. Clearly $\phi$ is a group action since $\psi$ is a group action. Further 
\[\phi(\sigma,\Omega)=\{q(\sigma u)\+(0,x)\mid u\in M(I),x\in R, q(u)\+(0,x)\in\Omega\}.\]
It follows from Lemma \ref{38} that $\phi(\sigma,\Omega)={}^{\sigma}\Omega$.
\end{proof}
\subsection{The action of $\U_{2n+1}(R,\Delta)$ on principal and normalized principal congruence subgroups}
In this subsection $(I,\Omega)$ denotes an odd form ideal of $(R,\Delta)$.
\begin{Theorem}\label{47}
Let $\sigma\in \U_{2n+1}(R,\Delta)$. Then ${}^{\sigma}\U_{2n+1}((R,\Delta),(I,\Omega))=\U_{2n+1}((R,\Delta),(I,{}^{\sigma}\Omega))$ and \[\U_{2n+1}((R,\Delta),(I,{}^{\sigma}\Omega))=\U_{2n+1}((R,\Delta),(I,\Omega))\Leftrightarrow {}^{\sigma}\Omega=\Omega.\]
\end{Theorem}
\begin{proof}
Assume we have shown that \[{}^{\sigma}\U_{2n+1}((R,\Delta),(I,\Omega))\subseteq \U_{2n+1}((R,\Delta),(I,{}^{\sigma}\Omega)).\tag{34.1}\]
Since $\sigma$ and $\Omega$ were arbitrarily chosen, it follows that 
\begin{align*}
{}^{\sigma^{-1}}\U_{2n+1}((R,\Delta),(I,{}^{\sigma}\Omega))&\subseteq \U_{2n+1}((R,\Delta),(I,\Omega))\\
\Leftrightarrow \hspace{0.4cm}\U_{2n+1}((R,\Delta),(I,{}^{\sigma}\Omega))&\subseteq {}^{\sigma}\U_{2n+1}((R,\Delta),(I,\Omega)).
\end{align*}
and hence ${}^{\sigma}\U_{2n+1}((R,\Delta),(I,\Omega))=\U_{2n+1}((R,\Delta),(I,{}^{\sigma}\Omega))$. Thus it suffices to show (34.1) in order to show the first assertion of the theorem.\\
Let $\rho\in {}^{\sigma}\U_{2n+1}((R,\Delta),(I,\Omega))$. Then there is a $\tau\in \U_{2n+1}((R,\Delta),(I,\Omega))$ such that $\rho={}^{\sigma}\tau$. We have to show that $\rho\in \U_{2n+1}((R,\Delta),(I,{}^{\sigma}\Omega))$, i.e. 
\begin{enumerate}[(1)]
\item $\rho_{hb}\equiv e_{hb}\bmod I $ and
\item $q(\rho u)\equiv q(u)\bmod {}^{\sigma}\Omega~\forall u\in M(R,\Delta)$.
\end{enumerate} 
One checks easily that (1) holds. It remains to show that (2) holds. Let $u\in M(R,\Delta)$. Set $\xi:=\tau-e$. Clearly $\sigma^{-1}u\in M(R,\Delta)$ since $u\in M(R,\Delta)$. It follows that $\xi\sigma^{-1}u\in M(I,\Omega)$ since $\tau\in \U_{2n+1}((R,\Delta),(I,\Omega))$. This implies $\sigma\xi\sigma^{-1}u\in M(I)$. Hence
\begin{align*}
&q(\rho u)\\
=&q((e+\sigma\xi\sigma^{-1})u)\\
=&q(u+\sigma\xi\sigma^{-1} u)\\
\overset{L. \ref{37}}{\equiv}&q(u)\+ q(\sigma\xi\sigma^{-1}u)\+(0,b(u,\sigma\xi\sigma^{-1}u))\bmod \Omega^I_{\min}.
\end{align*}
In order to show (2) it suffices to show that $q(\sigma\xi\sigma^{-1}u)\+(0,b(u,\sigma\xi\sigma^{-1}u))\in {}^{\sigma}\Omega$. But 
\begin{align*}
&q(\sigma\xi\sigma^{-1}u)\+(0,b(u,\sigma\xi\sigma^{-1}u))\\
\equiv& (q(\sigma_{*0})\minus(1,0))\circ (\xi\sigma^{-1}u)_0\+ q(\xi\sigma^{-1}u)\+ (0,b(u,\sigma\xi\sigma^{-1}u))\bmod \Omega^I_{min}
\end{align*}
by Lemma \ref{38}. Hence it suffices to show that $q(\xi\sigma^{-1}u)\+ (0,b(u,\sigma\xi\sigma^{-1}u))\in\Omega$ in order to show that $q(\sigma\xi\sigma^{-1}u)\+(0,b(u,\sigma\xi\sigma^{-1}u))\in{}^{\sigma}\Omega$. But
\begin{align*}
&q(\xi\sigma^{-1}u)\+(0,b(u,\sigma\xi\sigma^{-1}u))\\
=&q(\xi\sigma^{-1}u)\+(0,b(\sigma^{-1}u,\xi\sigma^{-1}u))\\
\overset{L.\ref{37}}{\equiv}& q(\sigma^{-1}u+\xi\sigma^{-1}u)\minus q(\sigma^{-1}u)\bmod \Omega^I_{min}\\
=&q((e+\xi)\sigma^{-1}u)\minus q(\sigma^{-1}u)\\
=&q(\tau\sigma^{-1}u)\minus q(\sigma^{-1}u)\in\Omega
\end{align*}
since $\tau\in \U_{2n+1}((R,\Delta),(I,\Omega))$. Thus (2) holds.\\
The second assertion of the theorem follows from the fact that if $(I,\Omega)$ and $(I, \Omega')$ are odd form ideals then $(I,\Omega)=(I, \Omega')\Leftrightarrow \U_{2n+1}((R,\Delta),(I,\Omega))=\U_{2n+1}((R,\Delta),(I,\Omega'))$.
\end{proof}
\begin{Corollary}\label{10000}
$\NU_{2n+1}((R,\Delta),(I,\Omega))=\{\sigma\in \U_{2n+1}(R,\Delta)\mid \Omega={}^{\sigma}\Omega\}$ and for all $\sigma\in \U_{2n+1}(R,\Delta)$ we have ${}^{\sigma}\NU_{2n+1}((R,\Delta),(I,\Omega))=\NU_{2n+1}((R,\Delta),(I,{}^{\sigma}\Omega))$.
\end{Corollary}
\begin{proof}
Straightforward.
\end{proof}
\begin{Corollary}\label{10001}
$\EU_{2n+1}(R,\Delta)\subseteq \NU_{2n+1}((R,\Delta),(I,\Omega))$.
\end{Corollary}
\begin{proof}
Let $\sigma\in \EU_{2n+1}(R,\Delta)$ be an elementary matrix. Then clearly $q(\sigma_{*0})=(1,0)$ and hence $\Omega={}^{\sigma}\Omega$. The result follows now from the previous corollary.
\end{proof}
\begin{Corollary}\label{utilde}
\[\NU_{2n+1}((R,\Delta),(I,\Omega))=\{\sigma\in \U_{2n+1}(R,\Delta)\mid q(\sigma u)\equiv q(\sigma^{-1}u)\equiv q(u)\bmod \Omega~\forall u\in M(I,\Omega)\}.\]
\end{Corollary}
\begin{proof}
Let $\sigma \in \U_{2n+1}(R,\Delta)$. Then, by Corollary \ref{10000}, $\sigma\in \NU_{2n+1}((R,\Delta),(I,\Omega))$ iff $\Omega={}^{\sigma}\Omega$. One checks easily that $\Omega={}^{\sigma}\Omega$ iff  \[q(\sigma_{*0}x)\equiv q(\sigma'_{*0}x)\equiv (x,0)\bmod \Omega~\forall x\in J(\Omega).\tag{37.1}\] 
It follows from Lemma \ref{38} that (37.1) is equivalent to
\[q(\sigma u)\equiv q(\sigma^{-1}u)\equiv q(u)\bmod \Omega~\forall u\in M(I,\Omega).\tag{37.2}\]
\end{proof}

The example below shows that in general (even if $R$ is semisimple and $n\geq 3$)
\begin{enumerate}[(1)]
\item $\NU_{2n+1}((R,\Delta),(I,\Omega))\neq \U_{2n+1}(R,\Delta)$,
\item $J(\Omega)\neq J({}^{\sigma}\Omega)$ and
\item the group action $(\sigma,\Omega)\mapsto {}^{\sigma}\Omega$ is not transitive.
\end{enumerate} 
\begin{Example}\label{174}
Suppose $R=\M_2(\mathbb{F}_2)$, $\bar{x}={x}^t~\forall x\in R$, $\lambda=1$, $\mu=0$ and $\Delta:=\Delta_{max}=\{(x,y)\in R\times R\mid y=y^t\}$. Let $I=\{0\}$. Then
\[\Omega_{min}^I=\{(0,x-\bar x\lambda)\mid x\in I\}\overset{.}{+} (\Delta\circ I)=\{0\}\times\{0\}\]
and
\[\Omega_{max}^I=\Delta\cap(\tilde I\times I)=R\times\{0\}.\]
Hence the relative odd form parameters for $I$ correspond to right ideals of $R$ (any relative odd form parameter for $I$ is of the form $J\times \{0\}$ for some right ideal $J$ of $R$ and conversely if $J$ is a right ideal of $R$, then $J\times \{0\}$ is a relative odd form parameter for $I$). But there are only 5 right ideals of $R$ (they are in 1-1 correspondence to the subspaces of the $\mathbb{F}_2$-vector space $\mathbb{F}_2\times \mathbb{F}_2$), namely
\begin{align*}
J_1&=\{0\},\\
J_2&=\{\begin{pmatrix}a&b\\0&0\end{pmatrix}\mid a,b\in\mathbb{F}_2\},\\
J_3&=\{\begin{pmatrix}0&0\\a&b\end{pmatrix}\mid a,b\in\mathbb{F}_2\},\\
J_4&=\{\begin{pmatrix}a&b\\a&b\end{pmatrix}\mid a,b\in\mathbb{F}_2\}\text{ and}\\
J_5&=R.
\end{align*}
Set $\Omega_i:=J_i\times\{0\}~(1\leq i\leq 5)$. Then \[\FP(I)=\{\Omega_1=\Omega^I_{min}
,\Omega_2,\Omega_3,\Omega_4,\Omega_5=\Omega_{max}^I\}.\]
It is easy to show that 
\begin{align*}
orbit(\Omega_1)&=\{\Omega_1\},\\
orbit(\Omega_2)&=orbit(\Omega_3)=orbit(\Omega_4)=\{\Omega_2,\Omega_3,\Omega_4\}\text{ and}\\
orbit(\Omega_5)&=\{\Omega_5\}.
\end{align*}
where for any $i\in \{1,\dots,5\}$, $orbit(\Omega_i)$ denotes the orbit of $\Omega_i$ with respect to the group action
\begin{align*}
\U_{2n+1}(R,\Delta)\times \FP(I)&\rightarrow \FP(I)\\
(\sigma,\Omega)&\mapsto {}^{\sigma}\Omega.
\end{align*}
For example, set
\[\sigma:=\left(\begin{array}{ccc}e^{n\times n}&0&0\\0&\scalebox{0.8}{$\left(\begin{array}{cc}0&1\\1&0\end{array}\right)$}&0\\0&0&e^{n\times n}\end{array}\right)\in \M_{2n+1}(R).\]
By Lemma \ref{20}, $\sigma\in \U_{2n+1}(R,\Delta)$. Clearly \begin{align*}
&^{\sigma}\Omega_2\\
=&\{(\sigma_{00}x,q_2(\sigma_{*0}x)+y)\mid (x,y)\in\Omega_2\}\\
=&\{(\sigma_{00}x,0)\mid x\in J_2\}\\
=&J_3\times\{0\}\\
=&\Omega_3.
\end{align*}
\end{Example}
\subsection{The effect of conjugation on level}
In this subsection we assume that $R$ is semilocal or quasifinite and $n\geq 3$. A subgroup $H$ of $\U_{2n+1}(R,\Delta)$ is called {\it E-normal} iff it is normalized by the elementary subgroup $\EU_{2n+1}(R,\Delta)$. Suppose $H$ is an E-normal subgroup of $\U_{2n+1}(R,\Delta)$. Set
\[I:=\{x\in R\mid T_{ij}(x)\in H \text{ for some }i,j\in\Theta_{hb}\}\]
and
\[\Omega:=\{(x,y)\in \Delta\mid T_{i}(x,y)\in H \text{ for some }i\in\Theta_{-}\}.\]
Then $(I,\Omega)$ is an odd form ideal such that $\EU_{2n+1}((R,\Delta),(I,\Omega))\subseteq H$. It is called the {\it level of H} and $H$ is called an {\it E-normal subgroup of level $(I,\Omega)$}. 

The next result of V. A. Petrov \cite{petrov} is a generalization of an analogous result \cite[Theorem 1.1 and Lemma 5.2]{bak-vavilov} for $\U_{2n}(R,\Lambda(\Delta))$.
\begin{Theorem}[V. A. Petrov]\label{51}
$\EU_{2n+1}((R,\Delta),(I,\Omega))$ is a normal subgroup of $\NU_{2n+1}((R,\Delta),(I,\Omega))$ and the standard commutator formulas
\begin{align*}
&[\CU_{2n+1}((R,\Delta),(I,\Omega)),\EU_{2n+1}(R,\Delta)]\\
=&[\EU_{2n+1}((R,\Delta),(I,\Omega)),\EU_{2n+1}(R,\Delta)]\\
=&\EU_{2n+1}((R,\Delta),(I,\Omega))
\end{align*}
hold. In particular from the absolute case $(I,\Omega)=(R,\Delta)$, it follows that $\EU_{2n+1}(R,\Delta)$ is perfect and normal in $\U_{2n+1}(R,\Delta)$. 
\end{Theorem}
\begin{proof}
By Corollary \ref{utilde}, $\NU_{2n+1}((R,\Delta),(I,\Omega))=\{\sigma\in \U_{2n+1}(R,\Delta)\mid q(\sigma u)\equiv q(\sigma^{-1}u)\equiv q(u)\bmod \Omega~\forall u\in M(I,\Omega)\}$. This is the group Petrov denotes by $\tilde U$ in \cite{petrov}. It follows from \cite[Theorems 1 and 4]{petrov} (semilocal case) resp. \cite[corollary on page 4765]{petrov} (quasifinite case), that $\EU_{2n+1}((R,\Delta),(I,\Omega))$ is a normal subgroup of $\NU_{2n+1}((R,\Delta),(I,\Omega))$ and \begin{align*}
&[\U_{2n+1}((R,\Delta),(I,\Omega)),\EU_{2n+1}(R,\Delta)]\\
\subseteq &\EU_{2n+1}((R,\Delta),(I,\Omega)). \tag{39.1}
\end{align*}
(Note that in \cite{petrov} the full congruence subgroup is defined a little differently). By \cite[Proposition 4]{petrov}, 
\begin{align*}
&[\EU_{2n+1}((R,\Delta),(I,\Omega)),\EU_{2n+1}(R,\Delta)]\\
=&\EU_{2n+1}((R,\Delta),(I,\Omega)).\tag{39.2}
\end{align*}
Hence
\begin{align*}
&[\CU_{2n+1}((R,\Delta),(I,\Omega)),\EU_{2n+1}(R,\Delta)]\\
=&[\EU_{2n+1}(R,\Delta),\CU_{2n+1}((R,\Delta),(I,\Omega))]\\
=&[[\EU_{2n+1}(R,\Delta),\EU_{2n+1}(R,\Delta)],\CU_{2n+1}((R,\Delta),(I,\Omega))]\\
\subseteq &\EU_{2n+1}((R,\Delta),(I,\Omega))\tag{39.3}
\end{align*}
by the definition of $\CU_{2n+1}((R,\Delta),(I,\Omega))$, (39.1) and the three subgroups lemma. (39.2) and (39.3) imply the assertion of the theorem. 
\end{proof}
\begin{Corollary}\label{165}
If $\sigma\in \U_{2n+1}(R,\Delta)$, then 
\[^{\sigma}\EU_{2n+1}((R,\Delta),(I,\Omega))=\EU_{2n+1}((R,\Delta),(I,{}^{\sigma}\Omega))\]
and
\[^{\sigma}\CU_{2n+1}((R,\Delta),(I,\Omega))=\CU_{2n+1}((R,\Delta),(I,{}^{\sigma}\Omega)).\]
Further $\NU_{2n+1}((R,\Delta),(I,\Omega))$ is the normalizer of $\EU_{2n+1}((R,\Delta),(I,\Omega))$ and of $\CU_{2n+1}((R,\Delta),(I,\Omega))$.
\end{Corollary}
\begin{proof}
Follows from Theorem \ref{47} and the standard commutator formulas in Theorem \ref{51}. 
\end{proof}
\begin{Corollary}
Let $H$ be an E-normal subgroup of level $(I,\Omega)$. Then for any $\sigma\in \U_{2n+1}(R,\Delta)$, $^{\sigma}H$ is an E-normal subgroup of level $(I,{}^{\sigma}\Omega)$. It follows that $H\subseteq \NU_{2n+1}((R,\Delta),(I,\Omega))$.
\end{Corollary}
\begin{proof}
Follows from the previous corollary and Corollary \ref{10000}.
\end{proof}
\section{The Extraction Theorem}\label{sec6}
In this section $(R,\Delta)$ denotes a Hermitian form ring and $n\geq 3$ a natural number. The main result of this section is Theorem \ref{extraction}, which is found at the very end of the section. It will be used in the next section to classify the E-normal subgroups of an odd dimensional unitary group.

Theorem \ref{extraction} says the following: Suppose $R$ is semilocal. Let $H$ be an E-normal subgroup of $\U_{2n+1}(R,\Delta)$ such that $\EU_{2n+1}(I,\Omega)\subseteq H$. If $H\not\subseteq \CU_{2n+1}((R,\Delta),(I,\Omega))$, then $H$ contains an elementary matrix which is not $(I,\Omega)$-elementary.

Although Theorem \ref{extraction} is for semilocal rings, many of the results needed to prove the theorem will also be required in the next section, where it is not assumed that $R$ is semilocal. So for most of the current section, no assumptions will be put on $R$. 

The results that are needed are always of the same kind. We start with a Hermitian form ring $(R,\Delta)$, some odd form ideal $(I,\Omega)$ of $(R,\Delta)$, and an element $h$ of $\U_{2n+1}(R,\Delta)$
which is not in $\CU_{2n+1}((R,\Delta),(I,\Omega))$. Let $^{E}\langle h\rangle$ denote the $\EU_{2n+1}(R,\Delta)$-normal closure of $\langle h\rangle$ in $\U_{2n+1}(R,\Delta)$. The result needed is that $^{E}\langle h\rangle \EU_{2n+1}((R,\Delta),(I,\Omega))$ contains an elementary matrix, say $g$, such that $g$ is not $(I,\Omega)$-elementary. This is always proved by showing that $h$ or some element of $h\EU_{2n+1}((R,\Delta),(I,\Omega))$, denoted again by $h$, satisfies a {\it generator extraction equation (GEE)}
\[ ^{\epsilon_{k}}([^{\epsilon_{k-1}}(\dots^{\epsilon_2}([^{\epsilon_1}([^{\epsilon_0}h,g_0]^{l_1}),g_1]^{l_2})\dots),g_{k-1}]^{l_{k}})\xi=g_k\] 
where $k\in\mathbb{N}$, $\xi\in \EU_{2n+1}((R,\Delta),(I,\Omega))$, $g_0,\dots,g_k,\epsilon_0,\dots,\epsilon_k\in \EU_{2n+1}(R,\Delta)$, $l_1,\dots,l_k\in \{\pm 1\}$ and $g_k=g$. However, it is typically the case that this is not shown in just one go. It is first shown that $h$ satisfies an {\it extraction equation (EE)} 
\[ ^{\epsilon_{k}}([^{\epsilon_{k-1}}(\dots^{\epsilon_2}([^{\epsilon_1}([^{\epsilon_0}h,g_0]^{l_1}),g_1]^{l_2})\dots),g_{k-1}]^{l_{k}})\xi=h_1\] 
where $k\in\mathbb{N}$, $\xi\in \EU_{2n+1}((R,\Delta),(I,\Omega))$, $g_0,\dots,g_{k-1},\epsilon_0,\dots,\epsilon_k\in \EU_{2n+1}(R,\Delta)$, $l_1,\dots,l_k\in \{\pm 1\}$ and $h_1 \not\in  \CU_{2n+1}((R,\Delta),(I,\Omega))$. The element $h_1$ is broke into several cases, covering all possibilities, and in each case it is shown that $h_1$ satisfies an EE with target, say $h_2$. In those cases where $h_2$ is an elementary matrix which is not $(I,\Omega)$-elementary, we are finished. If this happens in all cases, we have reached our goal. In the cases where $h_2$ is not an elementary matrix which is not $(I,\Omega)$-elementary, we repeat the modus operandi above. Doing this as often as necessary, we finally show that $^{E}\langle h\rangle EU_{2n+1}((R,\Delta),(I,\Omega))$ contains an elementary matrix $g$, such that $g$ is not $(I,\Omega)$-elementary.

\subsection{Preliminary results I}
In this subsection $I$ denotes an involution invariant ideal of $R$.
\begin{Lemma}\label{n1}
Let $H$ be an E-normal subgroup of $\U_{2n+1}(R,\Delta)$ which contains a matrix $\sigma$ of the form
\[\sigma=\begin{pmatrix} 1 &0&0 \\ *&e^{(2n-1)\times (2n-1)}& 0\\ *& *&1\end{pmatrix}\]
such that $\sigma_{-1,j}\not\in I$ for some $j\in\{2,\dots,-2\}\setminus\{0\}$ or $\sigma_{-1,0}\not\in I_0$. Then $H$ contains an elementary matrix which is not $(I,\Omega^I_{max})$-elementary. 
\end{Lemma}
\begin{proof}
~\\
\underline{case 1} Assume that $\sigma_{-1,j}\not\in I$ for some $j\in\{2,\dots,-2\}\setminus\{0\}$.\\
Choose a $k\neq 0,\pm 1, \pm j$. Then 
\begin{align*}
&T_{-1,-j}(\sigma_{-1,j})=[[\sigma,T_{jk}(1)],T_{k,-j}(1)]\in H.
\end{align*}
\\
\underline{case 2} Assume that $\sigma_{-1,0}\not\in I_0$.\\
By case 1 we can additionally assume that $\sigma_{-1,-2}\in I$. By the definition of $I_0$ there is a $c\in J(\Delta)$ such that $\sigma_{-1,0}c\not\in I$. Choose a $d\in R$ such that $(c,d)\in\Delta$. Then $\tau:=[\sigma,T_{j}(c,d)]$ has the form 
\[\tau=\begin{pmatrix} 1 &0&0 \\ *&e^{(2n-1)\times (2n-1)}& 0\\ *&*&1\end{pmatrix}\]
and $\tau_{-1,2}=\sigma_{-1,0}c-\sigma_{-1,-2}\bar d\lambda\not\in I$. We can finish by applying case 1 to $\tau$.
\end{proof}
\begin{Lemma}\label{n2}
Let $H$ be an E-normal subgroup of $\U_{2n+1}(R,\Delta)$ which contains a matrix $\sigma$ of the form
\[\sigma=\begin{pmatrix}1&0&0\\ *&E&0\\ *& *&1\end{pmatrix}\]
where $E\in \U_{2n-1}(R,\Delta)\setminus \U_{2n-1}((R,\Delta),(I,\Omega^I_{max}))$. Then $H$ contains an elementary matrix which is not $(I,\Omega^I_{max})$-elementary. 
\end{Lemma}
\begin{proof}
By Remark \ref{40}, there are $i,j\in\{2,\dots,-2\}\setminus\{0\}$ such that $\sigma_{ij}\not\equiv \delta_{ij}\bmod I $ or there is a $j\in\{2,\dots,-2\}\setminus\{0\}$ such that $\sigma_{0j}\not\in \tilde I$ or $\sigma_{00}\not\equiv 1\bmod \tilde I_0$.\\
\\
\underline{case 1} Assume there are $i,j\in\{2,\dots,-2\}\setminus\{0\}$ such that $\sigma_{ij}\not\equiv \delta_{ij}\bmod I $.\\
One checks easily that $\tau:=[\sigma^{-1},T_{-1,i}(1)]$ has the form 
\[\tau=\begin{pmatrix} 1 &0&0 \\ *&e^{(2n-1)\times (2n-1)}& 0\\ *& w&1\end{pmatrix}\]
where $w=(\sigma_{i2},\dots,\sigma_{i,-2})-e_i^t$. Hence $\tau_{-1,j}=\sigma_{ij}-\delta_{ij}\not\in I$. It follows from the previous lemma that $H$ contains an elementary matrix which is not $(I,\Omega^I_{max})$-elementary.\\
\\
\underline{case 2} Assume there is a $j\in\{2,\dots,-2\}\setminus\{0\}$ such that $\sigma_{0j}\not\in \tilde I$.\\
By the definition of $\tilde I$ there is an $a\in J(\Delta)$ such that $\bar a\mu \sigma_{0j}\not\in I$. Choose a $b\in R$ such that $(a,b)\in \Delta$. Then $\tau:=[\sigma^{-1},T_{-1}(\minus(a,b))]$ has the form 
\[\tau=\begin{pmatrix} 1 &0&0 \\ *&e^{(2n-1)\times (2n-1)}& 0\\ *& w&1\end{pmatrix}\]
where $w=\bar a\mu ((\sigma_{02},\dots,\sigma_{0,-2})-e_0^t)$. Hence $\tau_{-1,j}=\bar a\mu \sigma_{0j}\not\in I$. It follows from the previous lemma that $H$ contains an elementary matrix which is not $(I,\Omega^I_{max})$-elementary.\\
\\
\underline{case 3} Assume that $\sigma_{00}\not\equiv 1\bmod \tilde I_0$.\\
By case 2 we can additionally assume that $\sigma_{0,-2}\in \tilde I$. By the definition of $\tilde I_0$ there is an $a\in J(\Delta)$ such that $\bar a\mu (\sigma_{00}-1)\not\in I_0$. Choose a $b\in R$ such that $(a,b)\in \Delta$. Then $\tau:=[\sigma^{-1},T_{-1}(\minus(a,b))]$ has the form 
\[\tau=\begin{pmatrix} 1 &0&0 \\ *&e^{(2n-1)\times (2n-1)}& 0\\ *& w&1\end{pmatrix}\]
where $w=\bar a\mu ((\sigma_{02},\dots,\sigma_{0,-2})-e_0^t)$. Clearly $\tau_{-1,-2}=\bar a \mu\sigma_{0,-2}\in I$ and $\tau_{-1,0}=\bar a \mu(\sigma_{00}-1)\not\in I_0$. By the definition of $I_0$ there is a $c\in J(\Delta)$ such that $\tau_{-1,0}c\not\in I$. Choose a $d\in R$ such that $(c,d)\in\Delta$. Then $\omega:=[\tau,T_{-2}(c,d)]$ has the form 
\[\omega=\begin{pmatrix} 1 &0&0 \\ *&e^{(2n-1)\times (2n-1)}& 0\\ *& *&1\end{pmatrix}\]
and $\omega_{-1,2}=\tau_{-1,0}c-\tau_{-1,-2}\bar d\lambda\not\in I$. It follows from the previous lemma that $H$ contains an elementary matrix which is not $(I,\Omega^I_{max})$-elementary.
\end{proof}
\begin{Lemma}\label{n3}
Let $H$ be an E-normal subgroup of $\U_{2n+1}(R,\Delta)$ such that $\EU_{2n+1}(I,\Omega^I_{min})\subseteq H$. If $H$ contains a matrix $\sigma\not\in \U_{2n+1}((R,\Delta),(I,\Omega^I_{max}))$ of the form
\[\sigma=\begin{pmatrix}1&0&0\\ *&E&0\\ *& *&1\end{pmatrix}\]
where $E\in \U_{2n-1}((R,\Delta),(I,\Omega^I_{max}))$ and $\sigma_{i1}\in I$ for any $i\in\{2,\dots,n\}$, then $H$ contains an elementary matrix which is not $(I,\Omega^I_{max})$-elementary. 
\end{Lemma}
\begin{proof}
Throughout this proof, we use for matrices in $\U_{2n+1}(R,\Delta)$ the partitioning
\[\left(\begin{array}{cc|c|cc} x_1 &u_1&x_2&u_2&x_3\\v_1&A&v_2&B&v_3\\\hline x_4&u_3&x_5&u_4&x_6\\\hline v_4&C&v_5&D&v_6\\x_7 &u_5&x_8&u_6&x_9\end{array}\right)\]
where $x_1,\dots,x_9\in R$, $u_1,\dots,u_6\in \M_{1\times (n-1)}(R)$, $v_1,\dots,v_6\in \M_{(n-1)\times 1}(R)$ and $A,B,C,D\in \M_{n-1}(R)$.\\ Set $\xi:=\prod\limits_{i=2}^{n} T_{i1}(-\sigma_{i1})\in \EU_{2n+1}(I,\Omega^I_{min})\subseteq H$ and $\tau:=\xi\sigma$. Then $\tau$ has the form
\[\tau=\left(\begin{array}{cc|c|cc} 1 &0&0&0&0\\0&A&v_1&B&0\\\hline *&u_1&x&u_2&0\\\hline *&C&v_2&D&0\\ *&*&*&u_3&1\end{array}\right)\]
where $A-e,B, C-e,D\equiv 0\bmod I $, $u_1,u_2\equiv 0\bmod \tilde I$, $v_1,v_2\equiv 0\bmod I_0$ and $x-1\in \tilde I_0$. Further $\tau_{-1,j}\equiv b(\tau_{*1},\tau_{*j})=b(e_1,e_j)=0 \bmod I ~\forall j\in\{-n,\dots,-2\}$ and hence $u_3\equiv 0\bmod I $. Clearly $\tau\not\in \U_{2n+1}((R,\Delta),(I,\Omega^I_{max}))$ since $\sigma\not\in \U_{2n+1}((R,\Delta),(I,\Omega^I_{max}))$ and $\xi\in \EU_{2n+1}(I,\Omega^I_{min})\subseteq \U_{2n+1}((R,\Delta),(I,$ $\Omega^I_{max}))$. Suppose that $\tau_{-1,j}\in I$ for any $j\in\{1,\dots,n\}$ and $\tau_{-1,0}\in I_0$. Then, by Lemma \ref{20}, $\tau'_{i1}\in I~\forall i\in\{-n,\dots,-2\}$. It follows that $\tau_{i1}\in I~\forall i\in\{-n,\dots,-2\}$ since $\tau_{i*}\tau'_{*1}=0~\forall i\in\{-n,\dots,-2\}$. Further $\tau_{01}\in\tilde I$ (consider $b(\tau_{*1},\tau_{*0})$). Hence, by Remark \ref{40}, we have $\tau\in \U_{2n+1}((R,\Delta),(I,\Omega^I_{max}))$ and therefore a contradiction. Thus either there is a $j\in\{1,\dots,n\}$ such that $\tau_{-1,j}\not\in I$ or $\tau_{-1,0}\not\in I_0$. \\
\\
\underline{case 1} Assume that there is a $j\in\{3,\dots,n\}$ such that $\tau_{-1,j}\not\in I$.\\ Set $\omega:={}^{T_{12}(-1)}\tau$. One checks easily that 
\[\omega=\left(\begin{array}{cc|c|cc} 1 &\epsilon&*&*&*\\0&A&*&*&*\\\hline *&*&*&*&*\\\hline * &*&*&*&*\\ * &*&*&*&*\end{array}\right)\]
where $\epsilon \equiv 0\bmod I $ (and $A$ is the same matrix as above, hence $A\equiv e\bmod I $). Further $\omega_{-2,j}\equiv \tau_{-1,j}\bmod I $ and hence $\omega_{-2,j}\not\in I$. Set $\zeta:=\prod\limits_{k=2}^{n} T_{1k}(-\omega_{1k})\in \EU_{2n+1}(I,\Omega^I_{min})\subseteq H$ and $\rho:=\omega\zeta$. Then 
\[\rho=\left(\begin{array}{cc|c|cc} 1 &0&*&*&*\\0&A&*&*&*\\\hline *&*&*&*&*\\\hline * &*&*&*&*\\ * &*&*&*&*\end{array}\right).\]
Further $\rho_{-2,j}\equiv \omega_{-2,j}\bmod I $ and hence $\rho_{-2,j}\not\in I$. Set $\psi:=[\rho,T_{2,-j}(1)]$. Then
\begin{align*}
\psi=&[\rho,T_{2,-j}(1)]\\
=&(e+\rho_{*2}\rho'_{-j,*}-\rho_{*j}\bar\lambda\rho'_{-2,*})T_{2,-j}(-1)\\
=&\left(e+\begin{pmatrix}0\\ \rho_{22} \\ \vdots \\\rho_{n2} \\\rho_{02}\\\rho_{-n,2}\\\vdots \\\rho_{-1,2}\end{pmatrix}
\begin{pmatrix}\bar \rho_{-1,j}\lambda&\dots&\bar\rho_{-n,j}\lambda &\bar\rho_{0j}\mu & \bar \rho_{nj} & \dots& \bar \rho_{2j}& 0\end{pmatrix}\right.\\
&\left.-\begin{pmatrix} 0 \\ \rho_{2j} \\ \vdots \\ \rho_{nj} \\ \rho_{0j}\\ \rho_{-n,j} \\ \vdots \\\rho_{-1,j}\end{pmatrix}\bar\lambda
\begin{pmatrix}\bar\rho_{-1,2}\lambda&\dots&\bar\rho_{-n,2}\lambda&\bar\rho_{02}\mu  & \bar \rho_{n2} & \dots& \bar \rho_{22}& 0\end{pmatrix}\right)T_{2,-j}(-1).
\end{align*}
Clearly $\psi_{1*}=e^t_1$, $\psi_{*,-1}=e_{-1}$ and 
$\psi_{22}=1+\rho_{22}\bar\rho_{-2,j}\lambda-\rho_{2j}\bar\lambda \bar \rho_{-2,2}\lambda$. Since $A\equiv e\bmod I $, we have $\rho_{22}\equiv 1\bmod I $ and $\rho_{2j}\in I$. It follows that $\psi_{22}\not\equiv 1\bmod I $ since $\rho_{-2,j}\not\in I$. Thus one can apply the previous lemma to $\psi$.\\ 
\\
\underline{case 2} Assume that $\tau_{-1,2}\not\in I$.\\ 
Set $\omega:={}^{T_{13}(-1)}\sigma$. One checks easily that 
\[\omega=\left(\begin{array}{cc|c|cc} 1 &\epsilon&*&*&*\\0&A&*&*&*\\\hline *&*&*&*&*\\\hline * &*&*&*&*\\ * &*&*&*&*\end{array}\right)\]
where $\epsilon \equiv 0\bmod I $. Further $\omega_{-3,2}\equiv \tau_{-1,2}\bmod I $ and hence $\omega_{-3,2}\not \in I$. Set $\zeta:=\prod\limits_{k=2}^{n} T_{1k}(-\omega_{1k})\in \EU_{2n+1}(I,\Omega^I_{min})\subseteq H$ and $\rho:=\omega\zeta$. Then 
\[\rho=\left(\begin{array}{cc|c|cc} 1 &0&*&*&*\\0&A&*&*&*\\\hline *&*&*&*&*\\\hline * &*&*&*&*\\ * &*&*&*&*\end{array}\right).\]
Further $\rho_{-3,2}\equiv \omega_{-3,2}\bmod I $ and hence $\rho_{-3,2}\not\in I$. Set $\psi:=[\rho,T_{3,-2}(1)]$. Then
\begin{align*}
\psi
=&[\rho,g_2]\\
=&(e+\rho_{*3}\rho'_{-2,*}-\rho_{*2}\bar\lambda\rho'_{-3,*})T_{3,-2}(-1)\\
=&\left(e+\begin{pmatrix} 0 \\ \rho_{23} \\  \vdots \\ \rho_{n3} \\\rho_{03}\\ \rho_{-n,3} \\ \vdots \\ \rho_{-1,3}\end{pmatrix}
\begin{pmatrix}\bar \rho_{-1,2}\lambda&\dots&\bar\rho_{-n,2}\lambda &\bar\rho_{02}\mu & \bar \rho_{n2} & \dots& \bar \rho_{22}& 0\end{pmatrix}\right.\\
&\left.-\begin{pmatrix}0 \\ \rho_{22} \\  \vdots \\ \rho_{n2} \\ \rho_{02}\\ \rho_{-n,2} \\ \vdots \\ \rho_{-1,2}\end{pmatrix}\bar\lambda
\begin{pmatrix}\bar\rho_{-1,3}\lambda&\dots&\bar\rho_{-n,3}\lambda&\bar\rho_{03}\mu  & \bar \rho_{n3} & \dots& \bar \rho_{23}& 0\end{pmatrix}\right)T_{3,-2}(-1).
\end{align*}
Clearly $\psi_{1*}=e^t_1$, $\psi_{*,-1}=e_{-1}$ and 
$\psi_{33}=1+\rho_{33}\bar\rho_{-3,2}\lambda-\rho_{32}\bar\lambda \bar \rho_{-3,3}\lambda$. Since $A\equiv e\bmod I$, we have $\rho_{33}\equiv 1\bmod I $ and $\rho_{32}\in I$. It follows that $\psi_{33}\not\equiv 1\bmod I $ since $\rho_{-3,2}\not\in I$. Thus one can apply the previous lemma to $\psi$.\\ 
\\
\underline{case 3} Assume that $\tau_{-1,1}\not\in I$.\\
By cases 1 and 2 we can additionally assume that $\tau_{-1,j}\in I~\forall j\in\{2,\dots,n\}$. It follows that $\tau_{j1}\in I~\forall j\in\{-n,\dots,-2\}$. Set $\omega:={}^{T_{12}(-1)}\tau$. One checks easily that 
\[\omega=\left(\begin{array}{cc|c|cc} 1 &\epsilon&*&*&*\\0&A&*&*&*\\\hline *&*&*&*&*\\\hline * &*&*&*&*\\ * &*&*&*&*\end{array}\right)\]
where $\epsilon \equiv 0\bmod I $. Further $\omega_{-2,2}=\tau_{-2,2}+\tau_{-1,2}+\tau_{-2,1}+\tau_{-1,1}\equiv \tau_{-1,1}\bmod I $ and hence $\omega_{-2,2}\not\in I$. Set $\zeta:=\prod\limits_{k=2}^{n} T_{1k}(-\omega_{1k})\in \EU_{2n+1}(I,\Omega_{min}^I)\subseteq H$ and $\rho:=\omega\zeta$. Then 
\[\rho=\left(\begin{array}{cc|c|cc} 1 &0&*&*&*\\0&A&*&*&*\\\hline *&*&*&*&*\\\hline * &*&*&*&*\\ * &*&*&*&*\end{array}\right).\]
Further $\rho_{-2,2}\equiv \omega_{-2,2}\bmod I $ and hence $\rho_{-2,2}\not\in I$. Set $\psi:=[\rho,T_{2,-3}(1)]$. Then
\begin{align*}
\psi
=&[\rho,T_{2,-3}(1)]\\
=&(e+\rho_{*2}\rho'_{-3,*}-\rho_{*3}\bar\lambda\rho'_{-2,*})T_{2,-3}(-1)\\
=&\left(e+\begin{pmatrix} 0 \\\rho_{22} \\ \vdots \\ \rho_{n2} \\\rho_{02}\\ \rho_{-n,2} \\\vdots\\\rho_{-1,2}\end{pmatrix}
\begin{pmatrix}\bar \rho_{-1,3}\lambda&\dots&\bar\rho_{-n,3}\lambda &\bar\rho_{03}\mu & \bar \rho_{n3} & \dots& \bar \rho_{23}& 0\end{pmatrix}\right.\\
&\left.-\begin{pmatrix}  0 \\ \rho_{23} \\  \vdots \\ \rho_{n3} \\ \rho_{03}\\ \rho_{-n,3} \\ \vdots \\\rho_{-1,3}\end{pmatrix}\bar\lambda
\begin{pmatrix}\bar\rho_{-1,2}\lambda&\dots&\bar\rho_{-n,2}\lambda&\bar\rho_{02}\mu  & \bar \rho_{n2} & \dots& \bar \rho_{22}& 0\end{pmatrix}\right)T_{2,-3}(-1).
\end{align*}
Clearly $\psi_{1*}=e^t_1$, $\psi_{*,-1}=e_{-1}$ and 
$\psi_{32}=\rho_{32}\bar\rho_{-2,3}\lambda-\rho_{33}\bar\lambda \bar \rho_{-2,2}\lambda$. Since $A\equiv e\bmod I $, we have $\rho_{33}\equiv 1\bmod I $ and $\rho_{32}\in I$. It follows that $\psi_{32}\not\in I$ since $\rho_{-2,2}\not\in I$. Thus one can apply the previous lemma to $\psi$.\\ 
\\
\underline{case 4} Assume that $\tau_{-1,0}\not\in I_0$.\\
By cases 1-3 we can additionally assume that $\tau_{-1,j}\in I~\forall j\in\Theta_+$. By the definition of $I_0$ there is an $a\in J(\Delta)$ such that $\tau_{-1,0}a\not\in I$. Choose a $b\in R$ such that $(a,b)\in \Delta$ and set $\omega:={}^{T_{-3}(a,b)}\tau$. Then $\omega$ has the form 
\[\omega=\left(\begin{array}{cc|c|cc} 1 &0&0&0&0\\0&*&*&*&0\\\hline *&*&*&*&0\\\hline * &*&*&*&0\\ * &*&*&*&1\end{array}\right).\]
Further $\omega_{-1,3}\equiv -\tau_{-1,0}a\bmod I $ and therefore $\omega_{-1,3}\not\in I$. Thus one can either apply the previous lemma or case 1 to $\omega$.
\end{proof}
\begin{Proposition}\label{n4}
Let $H$ be an E-normal subgroup of $\U_{2n+1}(R,\Delta)$ such that $\EU_{2n+1}(I,\Omega^I_{min})\subseteq H$. If $H$ contains a matrix $\sigma\not\in \U_{2n+1}((R,\Delta),(I,\Omega^I_{max}))$ such that $\sigma_{*j}=e_j$ for some $j\in \Theta_{hb}$, then $H$ contains an elementary matrix which is not $(I,\Omega^I_{max})$-elementary. 
\end{Proposition}
\begin{proof}
In view of Definition \ref{24} we may assume that $\sigma_{*,-1}=e_{-1}$. It follows from Lemma \ref{20} that $\sigma_{1*}=e_1^t$. Hence $\sigma$ has the form
\[\sigma=\begin{pmatrix}1&0&0\\ *&E&0\\ *& *&1\end{pmatrix}\]
where $E\in \U_{2n-1}(R,\Delta)$. By Lemma \ref{n2} and Lemma \ref{n3} we can assume that $E\in  \U_{2n-1}((R,\Delta),(I,\Omega^I_{max}))$ and that there is an $i\in\{2,\dots,n\}$ such that $\sigma_{i1}\not\in I$. Then $\sigma_{-1,-i}\not\in I$ (consider $b(\sigma_{*1},\sigma_{*,-i})$). Choose a $j\in \{2,\dots, n\}\setminus\{i\}$ and set $\tau:=[\sigma,T_{-i,j}(1)]$. Then
\begin{align*}
\tau
=&[\sigma,T_{-i,j}(1)]\\
      =&(e+\sigma_{*,-i}\sigma'_{j*}-\sigma_{*,-j}\lambda\sigma'_{i*})T_{-i,j}(-1)\\
      =&\left(e+\begin{pmatrix}0 \\\sigma_{2,-i} \\  \vdots \\ \sigma_{n,-i}\\ \sigma_{0,-i} \\ \sigma_{-n,-i} \\ \vdots \\ \sigma_{-1,-i}\end{pmatrix}\bar\lambda
      \begin{pmatrix}\bar \sigma_{-1,-j}\lambda&\dots&\bar \sigma_{-n,-j}\lambda &\bar \sigma_{0,-j}\mu &\bar\sigma_{n,-j}&\dots&\bar\sigma_{2,-j}&0\end{pmatrix}\right.\\
      &\left.-\begin{pmatrix} 0 \\ \sigma_{2,-j} \\ \vdots \\ \sigma_{n,-j}\\ \sigma_{0,-j} \\ \sigma_{-n,-j} \\ \vdots \\ \sigma_{-1,-j}\end{pmatrix}\lambda\bar\lambda
            \begin{pmatrix}\bar \sigma_{-1,-i}\lambda&\dots&\bar \sigma_{-n,-i}\lambda&\bar\sigma_{0,-i}\mu  & \bar \sigma_{n,-i} & \dots& \bar \sigma_{2,-i}& 0\end{pmatrix}\right)T_{-i,j}(-1).
\end{align*}
Clearly $\tau_{1*}=e^t_1$ and $\tau_{*,-1}=e_{-1}$. Further \[\tau_{-1,j}=\sigma_{-1,-i}\bar\lambda\bar \sigma_{-j,-j}\lambda-\sigma_{-1,-j}\bar\sigma_{-j,-i}\lambda-(\sigma_{-1,-i}\bar\lambda\bar \sigma_{i,-j}-\sigma_{-1,-j}\bar\sigma_{i,-i}).\]
Since $\sigma_{-j,-i},\sigma_{i,-j},\sigma_{i,-i}\in I$, $\sigma_{-j,-j}\equiv 1\bmod I $ and $\sigma_{-1,-i}\not\in I$, it follows that $\tau_{-1,j}\not\in I$ and hence $\tau\not\in \U_{2n+1}((R,\Delta),(I,\Omega^I_{max}))$. Further $\tau_{k1}\in I$ for any $k\in\{2,\dots,n\}$. Thus one can apply Lemma \ref{n2} or Lemma \ref{n3} to $\tau$.
\end{proof}
\subsection{Preliminary results II}
In this subsection $I$ denotes an involution invariant ideal of $R$.
\begin{Definition}\label{68}
The subgroup of $\EU_{2n+1}(R,\Delta)$ consisting of all $f\in \EU_{2n+1}(R,\Delta)$ of the form 
\[f=\left(\begin{array}{c|c|c}A &u&B\\\hline 0&1&v\\\hline 0&0&C\end{array}\right)\]
where $A,B,C\in \M_{n}(R)$, $u\in M_{n\times 1}(R)$ and $v\in M_{1\times n}(R)$ is denoted by $\UEU_{2n+1}(R,\Delta)$. The subgroup of $\UEU_{2n+1}(R,\Delta)$ consisting of all upper triangular matrices $f\in \EU_{2n+1}(R,\Delta)$ with ones on the diagonal is denoted by $\TEU_{2n+1}(R,\Delta)$. 
\end{Definition}
\begin{Lemma}\label{69}
Suppose that $R$ is semilocal and let $\sigma\in \U_{2n+1}(R,\Delta)$. Then there is an $f\in \TEU_{2n+1}(R,\Delta)$ such that $(f\sigma)_{11}$ is left invertible.
\end{Lemma}
\begin{proof}
Since $R$ is semilocal, we have that $sr R=1$ where $sr R$ denotes the stable rank of $R$ (cf. \cite{vaserstein_stable}). Set $\Lambda:=\Lambda(\Delta)$. Then $R$ satisfies the $\Lambda$-stable range condition $\Lambda S_m$ for any $m\in\mathbb{N}$ (cf. \cite{bak-tang}).\\
\underline{step 1}\\Let $u$ be the first column of $\sigma$ and $v$ the first row of $\sigma^{-1}$. Since $vu=1$,
\[(u_1,\dots,u_n,v_0 u_0, u_{-n},\dots,u_{-1})^t\]
is unimodular. Since $sr R=1$, there is an $x\in R$ such that \[(u_1+xv_0 u_0,u_2,\dots,u_n,u_{-n},\dots,u_{-1})^t\] is unimodular. Since $\sigma\in \U_{2n+1}(R,\Delta)$, $q(\sigma_{*,-1})=(\sigma_{0,-1},q_2(\sigma_{*,-1}))\in\Delta$. It follows that $(-\sigma_{0,-1}\bar x,$ $\bar{\bar x}q_2(\sigma_{*,-1})\bar x)\in\Delta$ and hence $a:=(-\sigma_{0,-1}\bar x,\underline{\bar{\bar x}q_2(\sigma_{*,-1})\bar x})\in\Delta^{-1}$. Set $f_1:=T_{1}(a)\in \TEU_{2n+1}(R,\Delta)$ and $u^{(1)}:=f_1u$. One checks easily that 
\[(u^{(1)})_{hb}=(u_1+xv_0 u_0+yu_{-1},u_2,\dots,u_n,u_{-n},\dots,u_{-1})^t\]
where $y=\underline{\bar{\bar x}q_2(\sigma_{*,-1})\bar x}$ (note that $v_0=\sigma'_{10}=\bar \lambda \bar\sigma_{0,-1}\mu$ by Lemma \ref{20}). Therefore $(u^{(1)})_{hb}$ is unimodular.\\
\underline{step 2}\\
Since $(u^{(1)})_{hb}$ is unimodular and $R$ satisfies $\Lambda S_{n-1}$, there is a matrix 
\[\rho=\begin{pmatrix}e&\gamma\\0&e\end{pmatrix}\in \EU_{2n}(R,\Lambda)\]
where $\gamma\in \M_n(R)$ such that $(w_1,\dots,w_n)^t$ is unimodular where $w=\rho(u^{(1)})_{hb}$. Set
\[f_2=\phi^{2n+1}_{2n}(\rho)=\begin{pmatrix}e&0&\gamma\\0&1&0\\0&0&e\end{pmatrix}\in \TEU_{2n+1}(R,\Delta)\]
(where $\phi^{2n+1}_{2n}$ is defined as in Remark \ref{17}(b)) and $u^{(2)}:=f_2u^{(1)}$. Then clearly $(u^{(2)}_1,\dots,u^{(2)}_n)^t=(w_1,\dots,w_n)^t$ and hence $(u^{(2)}_1,\dots,$ $u^{(2)}_n)^t$ is unimodular.\\
\underline{step 3}\\
Since $(u^{(2)}_1,\dots,u^{(2)}_n)^t$ is unimodular and $sr R=1$, there is an $f_3\in \TEU_{2n+1}(R,\Delta)$ such that if $f_3u^{(2)}=u^{(3)}$, then $u^{(3)}_1$ is left invertible.\\
\\
Hence if $f=f_3f_2f_1\in \TEU_{2n+1}(R,\Delta)$, then $(f\sigma)_{11}$ is left invertible.  \end{proof}
\begin{Lemma}\label{70}
Suppose that $R$ is semilocal and let $\sigma\in \U_{2n+1}(R,\Delta)$. Then there is a $f\in \UEU_{2n+1}(R,\Delta)$ such that $((f\sigma)_{11},(f\sigma)_{21},\dots,$ $(f\sigma)_{n1})=(1,0,\dots,0)$ and $((f\sigma)_{12},(f\sigma)_{22},(f\sigma)_{32},\dots,(f\sigma)_{n2})=(0,1,0,\dots,0)$.
\end{Lemma}
\begin{proof}
~\\
\underline{step 1}\\Let $u$ be the first column of $\sigma$. By steps 1 and 2 in the previous lemma, there is an $f_1\in \UEU_{2n+1}(R,\Delta)$ such that if $u^{(1)}=f_1u$, then $(u^{(1)}_1,\dots,u^{(1)}_n)^t$ is unimodular.\\
\underline{step 2}\\
Since $(u^{(1)}_1,\dots,u^{(1)}_n)^t$ is unimodular, there is an $f_2\in \UEU_{2n+1}(R,\Delta)$ such that if $f_2u^{(1)}=u^{(2)}$, then $(u^{(2)}_1,\dots,u^{(2)}_{n})^t=(1,0,\dots,0)^t$ (see \cite[Chapter V, (3.3)(1)]{bass_book}). Set $\tau:=f_2f_1\sigma$. Then the first two columns of $\tau$ equal
\[\begin{pmatrix}1&F\\0&G\\C&H\\D&I\\E&J\end{pmatrix}\]
for some $C,E,F,H,J\in R$ and $D,G,I\in \M_{(n-1)\times 1}(R)$.\\
\underline{step 3}\\
Set $\epsilon_1:=\prod\limits_{j=-n}^{-2}T_{j1}(-u^{(2)}_j)$ and $u^{(3)}:=\epsilon_1u^{(2)}$. Then clearly $u^{(3)}=e_1+e_0x+e_{-1}y$ for some $x,y\in R$. Since $u^{(3)}$ is the first column of a unitary matrix, $(x,y)=q(u^{(3)})\in\Delta$.
Set $\epsilon_2:=T_{-1}(\minus(x,y))$ and $u^{(4)}:=\epsilon_2u^{(3)}$. One checks easily that $u^{(4)}=e_1=(1,0,\dots,0)^t$. Set $\epsilon:=\epsilon_2\epsilon_1$. Since the first column of $\epsilon\tau$ equals $u^{(4)}=e_1$, its last row equals $e^t_{-1}$ (follows from Lemma \ref{20}). Hence the first two columns of $\epsilon\tau$ equal
\[\begin{pmatrix}1&F\\0&G\\0&H-CF\\0&I-DF\\0&0\end{pmatrix}.\]
It follows that $(G, z(H-CF), I-DF)^t$ is unimodular for some $z\in \overline{J(\Delta)}\mu$ (see step 1 in the previous lemma). Set $\xi:=T_{12}(-F)$. Then the first two columns of $\tau\xi$ equal
\[\begin{pmatrix}1&0\\0&G\\C&H-CF\\D&I-DF\\E&J-EF\end{pmatrix}.\]
Since $(G, z(H-CF), I-DF)^t$ is unimodular, there is a $\zeta\in \UEU_{2n-1}(R,\Delta)$ such that the first $n-1$ coefficients of $\zeta(G, H-CF, I-DF)^t$ equal $(1,0,\dots,0)^t$ (see steps 1 and 2). Set $f_3:=\psi_{2n-1}^{2n+1}(\zeta)\in \UEU_{2n+1}(R,\Delta)$. Then there are $C', H'\in R$ and $B',D',I'\in \M_{(n-1)\times 1}(R)$ such that the first two columns of $f_3\tau\xi$ equal
\[\begin{pmatrix}1&0\\B'&G'\\C'&H'\\D'&I'\\E&J-EF\end{pmatrix}\]
where $G'=(1,0,\dots,0)^t\in \M_{(n-1)\times 1}$. \\
\underline{step 4}\\
Set $f_4:=\prod\limits_{j=2}^n T_{j1}(-B'_j)\in \UEU_{2n+1}(R,\Delta)$. Then the first two columns of $f_4f_3\tau\xi$ equal
\[\begin{pmatrix}1&0\\0&G'\\C'&H'\\D'&I'\\E'&J'\end{pmatrix}\]
for some $E',J'\in R$. Hence the first two columns of $f_4f_3\tau$ equal
\[\begin{pmatrix}1&F\\0&G'\\C'&H''\\D'&I''\\E'&J''\end{pmatrix}\]
for some $H'', J''\in R$ and a $I''\in \M_{(n-1)\times 1}(R)$.\\
\underline{step 5}\\
Set $f_5:=T_{12}(-F)\in \UEU_{2n+1}(R,\Delta)$ and $f:=f_5\dots f_1\in \UEU_{2n+1}(R,\Delta)$. Then the first two columns of $f\sigma=f_5f_4f_3\tau$ equal
\[\begin{pmatrix}1&0\\0&G'\\C'&H''\\D''&I'''\\E'&J''\end{pmatrix}\]
for some $D'',I''\in \M_{(n-1)\times 1}(R)$.
\end{proof}
\begin{Lemma}\label{72}
Let $h\in \U_{2n+1}(R,\Delta)\setminus \CU_{2n+1}((R,\Delta),(I,\Omega^{I}_{max}))$. Then either
\begin{enumerate}[(1)]
\item there are an $f\in \EU_{2n+1}(R,\Delta)$ and an $x\in R$ such that \[[^fh,T_{1,-2}(x)]\not\in \U_{2n+1}((R,\Delta),(I,\Omega^{I}_{max}))\]
or
\item there are an $f\in \EU_{2n+1}(R,\Delta)$, an $x\in R$, a $k\in \Theta_{hb}$ and a $(y,z)\in \Delta^{-\epsilon(k)}$ such that \[[^f[h,T_{k}(y,z)],T_{1,-2}(x)]\not\in \U_{2n+1}((R,\Delta),(I,\Omega^{I}_{max})).\]
\end{enumerate} 
\end{Lemma}
\begin{proof}~\\
\underline{case 1} Assume that $[h,T_{ij}(x)]\not\in \U_{2n+1}((R,\Delta),(I,\Omega^{I}_{max}))$ for some $x\in R$ and $i,j\in \Theta_{hb}$ such that $i\neq\pm j$.\\
By Lemma \ref{25} there is an $f\in \EU_{2n+1}(R,\Delta)$ such that $^fT_{ij}(x)=T_{1,-2}(x)$. Hence $[^fh,T_{1,-2}(x)]={}^f[h,T_{ij}(x)]\not\in \U_{2n+1}((R,\Delta),(I,\Omega^{I}_{max}))$. Thus (1) holds.\\
\\
\underline{case 2} Assume that $[h,T_{ij}(x)]\in \U_{2n+1}((R,\Delta),(I,\Omega^{I}_{max}))$ for any $x\in R$ and $i,j\in \Theta_{hb}$ such that $i\neq\pm j$.\\
One checks easily that in this case $h_{ij}\in I$ for any $i,j\in \Theta_{hb}$ such that $i\neq j$ and $h_{0j}\in \tilde I$ for any $j\in\Theta_{hb}$. Since $h\not\in \CU_{2n+1}((R,\Delta),(I,\Omega^{I}_{max}))$, there are a $k\in \Theta_{hb}$ and a $(y,z)\in \Delta^{-\epsilon(k)}$ such that $[h,T_{k}(y,z)]\not\in \U_{2n+1}((R,\Delta),(I,\Omega^{I}_{max}))$. One checks easily that $[h,T_{k}(y,z)]_{0,-k}\not\in \tilde I$ or $[h,T_{k}(y,z)]_{k,-k}\not\in I$. It follows that $[[h,T_{k}(y,z)],T_{1k}(1)]\not\in \U_{2n+1}((R,\Delta),(I,\Omega^{I}_{max}))$. Hence $[h,T_{k}(y,z)]\not\in \CU_{2n+1}((R,\Delta),(I,\Omega^{I}_{max}))$. Applying case 1 to $[h,T_{k}(y,z)]$ we get (2).
\end{proof}
\begin{Proposition}\label{123}
Suppose that $R$ is semilocal and let $H$ be an E-normal subgroup of $\U_{2n+1}(R,\Delta)$. If $H\not\subseteq \CU_{2n+1}((R,\Delta),(I,\Omega_{max}^I))$, then $H$ contains a matrix $\sigma\not\in \U_{2n+1}((R,\Delta),(I,\Omega^{I}_{max}))$ of the form 
\[\sigma=\left(\begin{array}{cc|c|c} A_1 &A_2&t_1&\multirow{2}{*}{B}\\0&e^{(n-2)\times (n-2)}& 0&\\\hline \multicolumn{2}{c|}{v}&z&w\\\hline \multicolumn{2}{c|}{C}&u&D\end{array}\right)\]
where $A_1\in \M_{2}(R)$, $A_2\in \M_{2 \times (n-2)}(R)$, $t_1\in \M_{2\times 1}(R)$, $B,C,D\in \M_{n}(R)$, $u\in \M_{n \times 1}(R)$, $v,w\in \M_{1 \times n}(R)$ and $z\in R$. 
\end{Proposition}
\begin{proof}
Suppose $H\not\subseteq \CU_{2n+1}((R,\Delta),(I,\Omega_{max}^I))$. Then, by the previous lemma there are an $h\in H$, an $f_0\in \EU_{2n+1}(R,\Delta)$ and an $x\in R$ such that $[^{f_0}h,T_{1,-2}(x)]\not\in \U_{2n+1}((R,\Delta),(I,\Omega^{I}_{max}))$. Set $g:=T_{1,-2}(x)$. By Lemma \ref{70} there is an $f_1\in \UEU_{2n+1}(R,\Delta)$ such that the first $n$ coefficients of $f_1(^{f_0}h)_{*1}$ equal $(1,0,\dots,0)^t$ and the first $n$ coefficients of $f_1(^{f_0}h)_{*2}$ equal $(0,1,0,$ $\dots,0)^t$. Set $\sigma:={}^{f_1}[^{f_0}h,g]$. Since $[^{f_0}h,g]\not\in \U_{2n+1}((R,\Delta),(I,\Omega^{I}_{max}))$ and $\U_{2n+1}((R,\Delta),(I,\Omega^{I}_{max}))$ is a normal subgroup of $\U_{2n+1}(R,\Delta)$, we have $\sigma\not\in \U_{2n+1}((R,\Delta),(I,\Omega^{I}_{max}))$. Clearly 
\begin{align*}
\sigma
=&{}^{f_1}[^{f_0}h,g]\\
=&{}^{f_1}(g^{-1}+(^{f_0}h)_{*1}x((^{f_0}h)^{-1})_{-2,*}g^{-1}-(^{f_0}h)_{*2}\bar\lambda\bar x((^{f_0}h)^{-1})_{-1,*}g^{-1})\\
=&{}^{f_1}g^{-1}+{}^{f_1}((^{f_0}h)_{*1}x((^{f_0}h)^{-1})_{-2,*}g^{-1})-{}^{f_1}((^{f_0}h)_{*2}\bar\lambda\bar x((^{f_0}h)^{-1})_{-1,*}g^{-1}).
\end{align*}
Hence the $i$-th row of $\sigma$ equals the $i$-th row of $^{f_1}g^{-1}$ for any $i\in \{3,\dots,n\}$. But $^{f_1}g^{-1}$ has the form 
\[^{f_1}g^{-1}=\left(\begin{array}{c|c|c} e^{n\times n} &0&*\\\hline 0&1&0\\\hline 0&0&e^{n\times n}\end{array}\right)\]
since $f_1\in \UEU_{2n+1}(R,\Delta)$. The assertion of the proposition follows. 
\end{proof}
\subsection{The case $\Omega=\Omega_{max}^I$ and $n\geq 4$}
In this subsection $I$ denotes an involution invariant ideal of $R$.
\begin{Lemma}\label{n5}
Suppose $n\geq 4$. Let $H$ be an E-normal subgroup of $\U_{2n+1}(R,\Delta)$ such that $\EU_{2n+1}(I,\Omega^I_{min})\subseteq H$. If $H$ contains a matrix $\sigma$ of the form
\[\sigma=\left(\begin{array}{c|c|c} A &t&B\\\hline v&z&w\\\hline C&u&D\end{array}\right)=\left(\begin{array}{cc|c|cc} A_1 &A_2&t_1&B_1&B_2\\0&e^{2\times 2}& 0&B_3&B_4\\\hline \multicolumn{2}{c|}{v}&z&\multicolumn{2}{|c}{w}\\\hline \multicolumn{2}{c|}{C}&u&\multicolumn{2}{|c}{D}\end{array}\right)\]
where $A,B,C,D\in \M_{n}(R)$, $t,u\in \M_{n \times 1}(R)$, $v,w\in \M_{1 \times n}(R)$, $z\in R$, $A_1,B_2\in \M_{n-2}(R)$, $A_2,B_1\in \M_{(n-2) \times 2}(R)$, $B_3\in \M_2(R)$, $B_4\in \M_{2 \times (n-2)}(R)$ and $t_1\in \M_{(n-2) \times 1}(R)$ such that either $A\not\equiv e\bmod I $ or $v\not\equiv 0\bmod \tilde I$ or $C\not\equiv 0\bmod I $, then $H$ contains an elementary matrix which is not $(I,\Omega^I_{max})$-elementary. 
\end{Lemma}
\begin{proof}
Clearly there is a $j\in\Theta_+$ such that $\sigma_{*j}\not\equiv e_j\bmod I,\tilde I $.\\
\\
\underline{case 1} Assume there is a $j\in\{1,\dots,n-2\}$ such that $\sigma_{*j}\not\equiv e_j\bmod I,\tilde I $.\\
Set $\tau:=[\sigma,T_{j,-(n-1)}(1)]$. Clearly the $(n-1)$-th row of
\begin{align*}
\tau
=&[\sigma,T_{j,-(n-1)}(1)]\\
=&(e+\sigma_{*j}\sigma'_{-(n-1),*}-\sigma_{*,n-1}\bar\lambda\sigma'_{-j,*})T_{j,-(n-1)}(-1)\\
=&\left(e+\begin{pmatrix} \sigma_{1j} \\ \vdots \\ \sigma_{n-2,j} \\ 0\\ 0\\ \sigma_{0j} \\\sigma_{-n,j}\\ \vdots\\ \sigma_{-1,j}\end{pmatrix}
\begin{pmatrix}\bar \sigma_{-1,n-1}\lambda&\dots&\bar \sigma_{-n,n-1}\lambda&~\bar \sigma_{0,n-1}\mu  & 0 & 1~ &\bar\sigma_{n-2,n-1}  &\dots&\bar \sigma_{1,n-1} \end{pmatrix}\right.\\
&\left.-\begin{pmatrix} \sigma_{1,n-1} \\ \vdots \\ \sigma_{n-2,n-1} \\ 1 \\ 0 \\ \sigma_{0,n-1} \\ \sigma_{-n,n-1}\\ \vdots\\ \sigma_{-1,n-1}\end{pmatrix}\bar\lambda
\begin{pmatrix}\bar \sigma_{-1,j}\lambda&\dots&\bar \sigma_{-n,j}\lambda&~\bar \sigma_{0j}\mu  & 0 &0~& \bar\sigma_{n-2,j}  & \dots& \bar \sigma_{1j}\end{pmatrix}\right)T_{j,-(n-1)}(-1)
\end{align*}
is not congruent to $e^t_{n-1}$ modulo $I,I_0$ since $\sigma_{*j}\not\equiv e_j \bmod I,\tilde I $. Hence $\tau\not\in \U_{2n+1}((R,\Delta),(I,\Omega^I_{max}))$. Further $\tau_{*,-n}=e_{-n}$ and thus one can apply Proposition \ref{n4} to $\tau$.\\
\\
\underline{case 2} Assume $\sigma_{*,n-1}\not\equiv e_{n-1}\bmod I,\tilde I $.\\
By case 1 we can additionally assume that $\sigma_{*j}\equiv e_j\bmod I,\tilde I ~\forall j\in\{1,\dots,n-2\}$.\\
\\
\underline{case 2.1} Assume that $\sigma_{1,n-1}\in I$.\\
Consider the first row of 
\begin{align*}
\tau
:=&[\sigma,T_{1,-(n-1)}(1)]\\
=&(e+\sigma_{*1}\sigma'_{-(n-1),*}-\sigma_{*,n-1}\bar\lambda\sigma'_{-1,*})T_{1,-(n-1)}(-1)\\
=&\left(e+\begin{pmatrix} \sigma_{11} \\ \vdots \\ \sigma_{n-2,1}\\ 0\\0\\ \sigma_{01} \\ \sigma_{-n,1}\\\vdots\\\sigma_{-1,1}\end{pmatrix}
\begin{pmatrix}\bar \sigma_{-1,n-1}\lambda&\dots&\bar \sigma_{-n,n-1}\lambda & ~\bar\sigma_{0,n-1}\mu &0 & 1~ & \bar \sigma_{n-2,n-1}& \dots &\bar \sigma_{1,n-1} \end{pmatrix}\right.\\
&\left.-\begin{pmatrix}\sigma_{1,n-1} \\\vdots \\\sigma_{n-2,n-1} \\1 \\ 0\\ \sigma_{0,n-1} \\ \sigma_{-n,n-1}\\ \vdots\\ \sigma_{-1,n-1}\end{pmatrix}\bar\lambda
\begin{pmatrix}\bar \sigma_{-1,1}\lambda&\dots&\bar \sigma_{-n,1}\lambda &~ \bar\sigma_{01}\mu &0  &0~& \bar\sigma_{n-2,1} &\dots & \bar \sigma_{11}\end{pmatrix}\right)T_{1,-(n-1)}(-1)
\end{align*}
which equals 
\setlength\jot{0.5cm}
\begin{align*}
&e^t_1+\sigma_{11}\begin{pmatrix}\bar \sigma_{-1,n-1}\lambda&\dots&\bar \sigma_{-n,n-1}\lambda & ~\bar\sigma_{0,n-1}\mu &0 & 1~ & \bar \sigma_{n-2,n-1}& \dots &\bar \sigma_{1,n-1} \end{pmatrix}\\
&-\sigma_{1,n-1}\bar\lambda\begin{pmatrix}\bar \sigma_{-1,1}\lambda&\dots&\bar \sigma_{-n,1}\lambda &~ \bar\sigma_{01}\mu &0  &0~& \bar\sigma_{n-2,1} &\dots & \bar \sigma_{11}\end{pmatrix}+e^t_{-(n-1)}x_1+e^t_{-1}x_2
\end{align*}
for some  $x_1,x_2\in R$. It is clearly not
congruent to $e^t_1$ modulo $I,I_0$ since $\sigma_{11}\equiv 1\bmod I $, $\sigma_{*,n-1}\not\equiv e_{n-1} \bmod I,\tilde I$ and $\sigma_{1,n-1}\in I$. Hence $\tau\not\in \U_{2n+1}((R,\Delta),(I,\Omega_{max}^I))$. Further $\tau_{*,-n}=e_{-n}$ and thus one can apply Proposition \ref{n4} to $\tau$.\\
\setlength\jot{3pt}
\\
\underline{case 2.2} Assume $\sigma_{1,n-1}\not\in I$.\\
Consider the second row of 
\begin{align*}
\tau
:=&[\sigma,T_{2,-(n-1)}(1)]\\
=&(e+\sigma_{*2}\sigma'_{-(n-1),*}-\sigma_{*,n-1}\bar\lambda\sigma'_{-2,*})T_{2,-(n-1)}(-1)\\
=&\left(e+\begin{pmatrix} \sigma_{12} \\\vdots \\\sigma_{n-2,2}\\0 \\ 0\\ \sigma_{02} \\ \sigma_{-n,2}\\\vdots\\ \sigma_{-1,2}\end{pmatrix}\begin{pmatrix}\bar \sigma_{-1,n-1}\lambda&\dots&\bar \sigma_{-n,n-1}\lambda & ~\bar\sigma_{0,n-1}\mu &0 & 1~&\bar\sigma_{n-2,n-1}&\dots \bar \sigma_{1,n-1}\end{pmatrix}\right.
\end{align*}
\begin{align*}
&\left.-\begin{pmatrix}\sigma_{1,n-1} \\\vdots \\ \sigma_{n-2,n-1} \\ 1 \\ 0\\ \sigma_{0,n-1} \\ \sigma_{-n,n-1}\\ \vdots\\ \sigma_{-1,n-1}\end{pmatrix}\bar\lambda
\begin{pmatrix}\bar \sigma_{-1,2}\lambda&\dots&\bar \sigma_{-n,2}\lambda &~ \bar\sigma_{02}\mu &0 &0~ &\bar\sigma_{n-2,2}&\dots& \bar \sigma_{12}\end{pmatrix}\right)T_{2,-(n-1)}(-1)
\end{align*}
which equals 
\setlength\jot{0.5cm}
\begin{align*}
&e_2^t+\sigma_{22}\begin{pmatrix}\bar \sigma_{-1,n-1}\lambda&\dots&\bar \sigma_{-n,n-1}\lambda & ~\bar\sigma_{0,n-1}\mu &0 & 1~&\bar\sigma_{n-2,n-1}&\dots &\bar \sigma_{1,n-1} \end{pmatrix}\\
&-\sigma_{2,n-1}\bar\lambda\begin{pmatrix}\bar \sigma_{-1,2}\lambda&\dots&\bar \sigma_{-n,2}\lambda &~ \bar\sigma_{02}\mu &0 &0~ &\bar\sigma_{n-2,2}&\dots& \bar \sigma_{12}\end{pmatrix}+e^t_{-(n-1)}x_1+e^t_{-2}x_2
\end{align*}
for some  $x_1,x_2\in R$. Its last entry clearly does not lie in $I$. Hence $\tau\not\in \U_{2n+1}((R,\Delta),(I,\Omega_{max}^I))$. Clearly $\tau_{*,-n}=e_{-n}$ an thus one can apply Proposition \ref{n4} to $\tau$.\\
\setlength\jot{3pt}
\\
\underline{case 3} Assume $\sigma_{*n}\not\equiv e_{n}\bmod I,\tilde I $.\\
This case is very similar to case 2 and hence is omitted.
\end{proof}
\begin{Proposition}\label{p1}
Suppose that $R$ is semilocal and $n\geq 4$. Let $H$ be an E-normal subgroup of $\U_{2n+1}(R,\Delta)$ such that $\EU_{2n+1}(I,\Omega^I_{min})\subseteq H$. If $H\not\subseteq \CU_{2n+1}((R,\Delta),(I,\Omega_{max}^I))$, then $H$ contains an elementary matrix which is not $(I,\Omega^I_{max})$-elementary.
\end{Proposition}
\begin{proof}
By Proposition \ref{123}, $H$ contains a matrix $\sigma\not\in \U_{2n+1}((R,\Delta),(I,\Omega_{max}^I))$ of the form \[\sigma=\left(\begin{array}{c|c|c} A &t&B\\\hline v&z&w\\\hline C&u&D\end{array}\right)=\left(\begin{array}{cc|c|cc} A_1 &A_2&t_1&B_1&B_2\\0&e^{2\times 2}& 0&B_3&B_4\\\hline \multicolumn{2}{c|}{v}&z&\multicolumn{2}{|c}{w}\\\hline \multicolumn{2}{c|}{C}&u&\multicolumn{2}{|c}{D}\end{array}\right)\]
where $A, B,\dots$ have the same size as in the previous lemma. By the previous lemma we may assume that $A\equiv e\bmod I $, $v\equiv 0\bmod \tilde I$ and $C\equiv 0\bmod I $. Clearly 
\[\delta_{ij}=(\sigma^{-1}\sigma)_{ij}=\sum\limits_{k=1}^{-1}\sigma'_{ik}\sigma_{kj}\equiv\sigma'_{ij}\bmod I ~\forall i,j\in\Theta_+ \]
where $\delta_{ij}$ is the Kronecker delta (note that $\sigma'_{i0}\sigma_{0j}\in I$ since $\sigma'_{i0}\in\overline{J(\Delta)}\mu$ by Lemma \ref{20} and $\sigma_{0j}\in \tilde I$). It follows from Lemma \ref{20}, that $\sigma_{ij}\equiv \delta_{ij}\bmod I ~\forall i,j\in\Theta_-$, i.e. $D\equiv e\bmod I $. Hence $B\not\equiv 0\bmod I $ or $w\not\equiv 0\bmod \tilde I$ or $z\not\equiv 1\bmod \tilde I_0$ since $\sigma\not\in \U_{2n+1}((R,\Delta),(I,\Omega_{max}^I))$. \\
\\
\underline{case 1}  Assume that $B\not\equiv 0\bmod I $.\\
\\
\underline{case 1.1} Assume that $B_3\not\equiv 0 \bmod I $ or $B_4\not\equiv 0 \bmod I $.\\
Clearly
\begin{align*}
 \tau
 :=&[\sigma^{-1},T_{-n,n-1}(1)]\\
 =&(e+\sigma'_{*,-n}\sigma_{n-1,*}-\sigma'_{*,-(n-1)}\lambda\sigma_{n*})T_{-n,n-1}(-1)\\
 =&\left(e+\begin{pmatrix}\bar\lambda\bar \sigma_{n,-1} \\ \vdots \\ \bar\lambda\bar \sigma_{n,-n}\\ \sigma'_{0,-n} \\ 1 \\ 0 \\ \vdots \\ 0\end{pmatrix}
 \begin{pmatrix}0&\dots&0& 1 & 0 &0&\sigma_{n-1,-n}& \dots & \sigma_{n-1,-1} \end{pmatrix}\right.\\
 &\left.-\begin{pmatrix}\bar\lambda\bar \sigma_{n-1,-1} \\\vdots \\\bar\lambda\bar \sigma_{n-1,-n}\\\sigma'_{0,-(n-1)} \\ 0\\1 \\0 \\\vdots \\ 0\end{pmatrix}\lambda
 \begin{pmatrix}0&\dots&0& 1&0&\sigma_{n,-n}& \dots & \sigma_{n,-1}\end{pmatrix}\right)T_{-n,n-1}(-1)
\end{align*}
Since $B_3\not\equiv 0 \bmod I $ or $B_4\not\equiv 0 \bmod I $, we have $(\tau_{-n,-n},\dots,\tau_{-n,-1})\not\equiv (1,0,\dots,0)\bmod I $ or
$(\tau_{-(n-1),-n},\dots,$ $\tau_{-(n-1),-1})\not\equiv (0,1,0,\dots,0)\bmod I $. 
Hence $\tau\not\in \U_{2n+1}((R,\Delta),(I,$ $\Omega_{max}^I))$.
Further $\tau_{*2}=e_2$ and thus one can apply Proposition \ref{n4} to $\tau$.\\
\\
\underline{case 1.2} Assume that $B_2\not\equiv 0\bmod I $.\\
By case 1.1 we can additionally assume that $B_3\equiv 0 \bmod I $ and $B_4\equiv 0 \bmod I $.
Set 
\[\xi:=\prod\limits_{k=-(n-2)}^{-1}T_{n-1,k}(-\sigma_{n-1,k})\prod\limits_{k=-(n-2)}^{-1}T_{nk}(-\sigma_{nk})\in \EU_{2n+1}(I,\Omega^I_{min})\subseteq H\]
and $\tau:=\sigma\xi$. Then $\tau$ has the form 
\[\tau=\left(\begin{array}{c|c|c} A' &t'&B'\\\hline v'&z'&w'\\\hline C'&u'&D'\end{array}\right)=\left(\begin{array}{cc|c|cc} A'_1 &A'_2&t'_1&B'_1&B'_2\\0&e^{2\times 2}& 0&B'_3&0\\\hline \multicolumn{2}{c|}{v'}&z'&\multicolumn{2}{|c}{w'}\\\hline \multicolumn{2}{c|}{C'}&u'&\multicolumn{2}{|c}{D'}\end{array}\right)\]
where $A'$ has the same size as $A$, $B'$ has the same size as $B$ and so on.
Clearly $\tau\equiv \sigma\bmod I,\tilde I,I_0,\tilde I_0 $ and hence $B'_2\not\equiv 0\bmod I $. Therefore there are an $i\in\{1,\dots,n-2\}$ and a $j\in\{-(n-2),\dots,-1\}$ such that $\tau_{ij}\not\in I$. Choose a $k\in \{1,\dots,n-2\}\setminus\{-j\}$ and set $\omega:={}^{T_{jk}(-1)}\tau$. Then $\omega$ has the form 
\[\omega=\left(\begin{array}{c|c|c} A'' &t''&B''\\\hline v''&z''&w''\\\hline C''&u''&D''\end{array}\right)=\left(\begin{array}{cc|c|cc} A''_1 &A''_2&t''_1&B''_1&B''_2\\0&e^{2\times 2}& 0&B''_3&0\\\hline \multicolumn{2}{c|}{v''}&z''&\multicolumn{2}{|c}{w''}\\\hline \multicolumn{2}{c|}{C''}&u''&\multicolumn{2}{|c}{D''}\end{array}\right)\]
where $A''$ has the same size as $A'$, $B''$ has the same size as $B'$ and so on. Further $\omega_{ik}\not\equiv \delta_{ik}\bmod I $. Hence $A''\not\equiv e\bmod I $ and thus one can apply the previous lemma to $\omega$.\\
\\
\underline{case 1.3} Assume that $B_1\not\equiv 0\bmod I $.\\
By cases 1.1 and 1.2 we can additionally assume that $B_3\equiv 0 \bmod I $, $B_4\equiv 0 \bmod I $ and $B_2\equiv 0\bmod I $. Since $B_1\not\equiv 0\bmod I $, there are an $i\in\{1,\dots,n-2\}$ and a $j\in\{-n,-(n-1)\}$ such that $\sigma_{ij}\not\in I$. Set $\tau:={}^{T_{j,-1}(-1)}\sigma$. Then $\tau$ has the form 
\[\tau=\left(\begin{array}{c|c|c} A' &t'&B'\\\hline v'&z'&w'\\\hline C'&u'&D'\end{array}\right)=\left(\begin{array}{cc|c|cc} A'_1 &A'_2&t'_1&B'_1&B'_2\\0&e^{2\times 2}& 0&B'_3&B'_4\\\hline \multicolumn{2}{c|}{v'}&z'&\multicolumn{2}{|c}{w'}\\\hline \multicolumn{2}{c|}{C'}&u'&\multicolumn{2}{|c}{D'}\end{array}\right)\]
where $A'$ has the same size as $A$, $B'$ has the same size as $B$ and so on.
Clearly $A'\equiv e\bmod I $, $v'\equiv 0\bmod \tilde I$, $C'\equiv 0\bmod I $, $D'\equiv e\bmod I $, $ B'_3\equiv 0 \bmod I $, $B'_4\equiv 0 \bmod I $. Further $\tau_{i,-1}\not\in I$ and hence $B'_2\not\equiv 0\bmod I $. Thus one can apply case 1.2 to $\tau$.\\
\\
\underline{case 2} Assume that $w\not\equiv 0\bmod \tilde I$.\\
By case 1 we may additionally assume that $B\equiv 0\bmod I $. Set 
\[\xi:=\prod\limits_{k=-(n-2)}^{-1}T_{n-1,k}(-\sigma_{n-1,k})\prod\limits_{k=-(n-2)}^{-1}T_{nk}(-\sigma_{nk})\in \EU_{2n+1}(I,\Omega_{min}^I)\subseteq H\]
and $\tau:=\sigma\xi$. Then $\tau$ has the form 
\[\tau=\left(\begin{array}{c|c|c} A' &t'&B'\\\hline v'&z'&w'\\\hline C'&u'&D'\end{array}\right)=\left(\begin{array}{cc|c|cc} A'_1 &A'_2&t'_1&B'_1&B'_2\\0&e^{2\times 2}& 0&B'_3&0\\\hline \multicolumn{2}{c|}{v'}&z'&\multicolumn{2}{|c}{w'}\\\hline \multicolumn{2}{c|}{C'}&u'&\multicolumn{2}{|c}{D'}\end{array}\right)\]
where $A'$ has the same size as $A$, $B'$ has the same size as $B$ and so on.
Clearly $\tau\equiv \sigma\bmod I,\tilde I,I_0,\tilde I_0 $ and hence $w'\not\equiv 0\bmod \tilde I$. Therefore there is a $j\in \Theta_{-}$ such that $w'_j\not\in \tilde I$. By the definition of $\tilde I$ there is an $a\in J(\Delta)$ such that $\bar a\mu w'_j\not \in I$. Choose a $b\in R$ such that $(a,b)\in \Delta$ and set $\omega:={}^{T_{-1}(a,b)}\tau$. Then $\omega$ has the form 
\[\omega=\left(\begin{array}{c|c|c} A'' &t''&B''\\\hline v''&z''&w''\\\hline C''&u''&D''\end{array}\right)=\left(\begin{array}{cc|c|cc} A''_1 &A''_2&t''_1&B''_1&B''_2\\0&e^{2\times 2}& 0&B''_3&0\\\hline \multicolumn{2}{c|}{v''}&z''&\multicolumn{2}{|c}{w''}\\\hline \multicolumn{2}{c|}{C''}&u''&\multicolumn{2}{|c}{D''}\end{array}\right)\]
where $A''$ has the same size as $A'$, $B''$ has the same size as $B'$ and so on. Further $\omega_{-1,j}\not\equiv\delta_{-1,j}\bmod I$. Hence $D''\not\equiv e\bmod I $. It follows that either $A''\not\equiv e\bmod I $ or $v''\not\equiv 0\bmod \tilde I$ or $C''\not\equiv 0\bmod I $ and thus one can apply the previous lemma to $\omega$.\\
\\
\underline{case 3} Assume that $z\not\equiv 1\bmod \tilde I_0$.\\
By cases 1 and 2 we may additionally assume that $B\equiv 0\bmod I $ and $w\equiv 0\bmod \tilde I$. Set 
\[\xi:=\prod\limits_{k=-(n-2)}^{-1}T_{n-1,k}(-\sigma_{n-1,k})\prod\limits_{k=-(n-2)}^{-1}T_{nk}(-\sigma_{nk})\in \EU_{2n+1}(I,\Omega_{min}^I)\subseteq H\]
and $\tau:=\sigma\xi$. Then $\tau$ has the form 
\[\tau=\left(\begin{array}{c|c|c} A' &t'&B'\\\hline v'&z'&w'\\\hline C'&u'&D'\end{array}\right)=\left(\begin{array}{cc|c|cc} A'_1 &A'_2&t'_1&B'_1&B'_2\\0&e^{2\times 2}& 0&B'_3&0\\\hline \multicolumn{2}{c|}{v'}&z'&\multicolumn{2}{|c}{w'}\\\hline \multicolumn{2}{c|}{C'}&u'&\multicolumn{2}{|c}{D'}\end{array}\right)\]
where $A'$ has the same size as $A$, $B'$ has the same size as $B$ and so on. Clearly $\tau\equiv \sigma\bmod I,\tilde I,I_0,\tilde I_0 $ and hence $z'\not\equiv 1\bmod \tilde I_0$. It follows from the definition of $\tilde I_0$ that there is an $a\in J(\Delta)$ such that $(z'-1)a\not\in \tilde I$. Choose a $b\in R$ such that $(a,b)\in \Delta$ and set $\omega:={}^{T_{-1}(a,b)}\tau$. Then $\omega$ has the form 
\[\omega=\left(\begin{array}{c|c|c} A'' &t''&B''\\\hline v''&z''&w''\\\hline C''&u''&D''\end{array}\right)=\left(\begin{array}{cc|c|cc} A''_1 &A''_2&t''_1&B''_1&B''_2\\0&e^{2\times 2}& 0&B''_3&0\\\hline \multicolumn{2}{c|}{v''}&z''&\multicolumn{2}{|c}{w''}\\\hline \multicolumn{2}{c|}{C''}&u''&\multicolumn{2}{|c}{D''}\end{array}\right)\]
where $A''$ has the same size as $A'$, $B''$ has the same size as $B'$ and so on. Further $\omega_{01}\not\in\tilde I$. Hence $v''\not\equiv 0\bmod \tilde I$ and thus one can apply the previous lemma to $\omega$.
\end{proof}
\subsection{The case $\Omega=\Omega_{max}^I$ and $n=3$}
In this subsection $I$ denotes an involution invariant ideal of $R$. If in a matrix the symbol $+$ appears at a position $(i,j)$, it means that the entry at this position is congruent to $\delta_{ij}$ modulo $I$ if $i,j\in \Theta_{hb}$, congruent to $\delta_{0j}$ modulo $\tilde I$ if $i=0,j\in \Theta_{hb}$, congruent to $\delta_{i0}$ modulo $I_0$ if $i\in \Theta_{hb},j=0$ and congruent to $1$ modulo $\tilde I_0$ if $i=j=0$. A $*$ stands for an arbitrary ring element. If positions in a matrix are marked by a $-$, it means that the entry at one of these positions does not lie in $I$ (resp. $\tilde I$, $I_0$, $\tilde I_0$) respectively is not congruent to $1$ modulo $I$ (resp. $\tilde I$, $I_0$, $\tilde I_0$). For example, given a matrix $\sigma\in \U_7(R,\Delta)$, an equation like
\[\sigma=\left(\begin{array}{ccc|c|ccc} *&*&*&*&*&*&*\\ *&*&*&*&*&*&*\\+&+&*&*&*&-&-\\\hline +&+&*&*&*&-&-\\\hline +&+&*&*&*&-&-\\ *&*&*&*&*&*&*\\ *&*&*&*&*&*&*\end{array}\right)\]
means that $\sigma_{31}, \sigma_{32},\sigma_{-3,1},\sigma_{-3,2}\in I$, $\sigma_{01},\sigma_{02}\in \tilde I$ and further that one of the entries $\sigma_{3,-2}, \sigma_{3,-1},\sigma_{-3,-2},\sigma_{-3,-1}$ does not lie in $I$ or one of the entries $\sigma_{0,-2},\sigma_{0,-1}$ does not lie in $\tilde I$.
\begin{Lemma}\label{o1}
Let $H$ be an E-normal subgroup of $\U_{7}(R,\Delta)$ such that $\EU_{7}(I,\Omega^I_{min})\subseteq H$. If $H$ contains a matrix $\sigma$ of the form 
\[\sigma=\left(\begin{array}{ccc|c|ccc} *&*&*&*&*&*&*\\ *&*&*&*&*&*&*\\ 0&0&*&*&*&*&*\\\hline -&-&*&*&*&*&*\\\hline -&-&*&*&*&*&*\\-&-&*&*&*&*&*\\-&-&*&*&*&*&*\end{array}\right),\]
then $H$ contains an elementary matrix which is not $(I,\Omega^I_{max})$-elementary.
\end{Lemma}
\begin{proof}
Set $\tau:=[\sigma,T_{1,-2}(1)]$. Then 
\begin{align*}
\tau
=&[\sigma,T_{1,-2}(1)]\\
=&(e+\sigma_{*1}\sigma'_{-2,*}-\sigma_{*2}\bar\lambda\sigma'_{-1,*})T_{1,-2}(-1)\\
=&\left(e+\begin{pmatrix}\sigma_{11}\\\sigma_{21}\\0\\ \sigma_{01}\\\sigma_{-3,1}\\ \sigma_{-2,1} \\ \sigma_{-1,1}\end{pmatrix}
\begin{pmatrix}\bar\sigma_{-1,2}\lambda&\bar\sigma_{-2,2}\lambda&\bar\sigma_{-3,2}\lambda&\bar \sigma_{02}\mu &0&\bar\sigma_{22}&\bar\sigma_{12}\end{pmatrix}\right.\\
&\left.-\begin{pmatrix}\sigma_{12}\\\sigma_{22}\\0\\ \sigma_{02}\\\sigma_{-3,2}\\ \sigma_{-2,2} \\ \sigma_{-1,2}\end{pmatrix}\bar\lambda
\begin{pmatrix}\bar\sigma_{-1,1}\lambda&\bar\sigma_{-2,1}\lambda&\bar\sigma_{-3,1}\lambda&\bar\sigma_{01}\mu &0&\bar\sigma_{21}&\bar\sigma_{11}\end{pmatrix}\right)T_{1,-2}(-1).
\end{align*}
Assume that $\tau\in \U_{7}((R,\Delta),(I,\Omega_{max}^I))$. Then $\tau\equiv e\bmod I,\tilde I,I_0,\tilde I_0 $ by Remark \ref{40}. Hence $\sigma_{*1}\bar\sigma_{-j,2}\lambda-\sigma_{*2}\bar\lambda\bar\sigma_{-j,1}\lambda=\tau_{*j}-e_j\equiv 0\bmod I,\tilde I ~\forall j\in\{1,2,3\}$ and $\sigma_{*1}\bar\sigma_{02}\mu-\sigma_{*2}\bar\lambda\bar\sigma_{01}\mu=\tau_{*0}-e_0\equiv 0\bmod I_0,\tilde I_0 $. By multiplying $\sigma'_{1*}$ resp. $\sigma'_{2*}$ from the left we get $\sigma_{-3,1},\sigma_{-2,1},\sigma_{-1,1},\sigma_{-3,2},\sigma_{-2,2},\sigma_{-1,2}\in I$ and $\sigma_{01},\sigma_{02}\in \tilde I$ (note that $\sigma'_{10},\sigma'_{20}\in \overline{J(\Delta)}\mu$ by Lemma \ref{20}). Since this is a contradiction, $\tau\not\in \U_{7}((R,\Delta),(I,\Omega_{max}^I))$. Further $\tau_{*,-3}=e_{-3}$ and thus one can apply Proposition \ref{n4} to $\tau$.
\end{proof}
\begin{Lemma}\label{o2}
Let $H$ be an E-normal subgroup of $\U_{7}(R,\Delta)$ such that $\EU_{7}(I,\Omega^I_{min})\subseteq H$. If $H$ contains a matrix $\sigma$ of the form 
\[\sigma=\left(\begin{array}{ccc|c|ccc} *&*&*&*&*&-&-\\ *&*&*&*&*&-&-\\ *&*&*&*&*&0&0\\\hline *&*&*&*&*&-&-\\\hline *&*&*&*&*&-&-\\ *&*&*&*&*&*&*\\ *&*&*&*&*&*&*\end{array}\right),\]
then $H$ contains an elementary matrix which is not $(I,\Omega^I_{max})$-elementary.
\end{Lemma}
\begin{proof}
Set $\tau:=[\sigma,T_{-1,2}(1)]$. Then 
\begin{align*}
\tau
=&[\sigma,T_{-1,2}(1)]\\
=&(e+\sigma_{*,-1}\sigma'_{2*}-\sigma_{*,-2}\lambda\sigma'_{1*})T_{-1,2}(-1)\\
=&\left(e+\begin{pmatrix}\sigma_{1,-1}\\\sigma_{2,-1}\\0\\\sigma_{0,-1}\\ \sigma_{-3,-1}\\ \sigma_{-2,-1} \\ \sigma_{-1,-1}\end{pmatrix}
\bar\lambda\begin{pmatrix}\bar\sigma_{-1,-2}\lambda&\bar\sigma_{-2,-2}\lambda&\bar\sigma_{-3,-2}\lambda&\bar\sigma_{0,-2}\mu &0&\bar\sigma_{2,-2}&\bar\sigma_{1,-2}\end{pmatrix}\right.
\end{align*}
\begin{align*}
&\left.-\begin{pmatrix}\sigma_{1,-2}\\\sigma_{2,-2}\\0\\\sigma_{0,-2}\\ \sigma_{-3,-2}\\ \sigma_{-2,-2} \\ \sigma_{-1,-2}\end{pmatrix}\lambda\bar\lambda
\begin{pmatrix}\bar\sigma_{-1,-1}\lambda&\bar\sigma_{-2,-1}\lambda&\bar\sigma_{-3,-1}\lambda&\bar\sigma_{0,-1}\mu &0&\bar\sigma_{2,-1}&\bar\sigma_{1,-1}\end{pmatrix}\right)T_{-1,2}(-1).
\end{align*}
Assume that $\tau\in \U_{7}((R,\Delta),(I,\Omega_{max}^I))$. Then $\sigma_{*,-1}\bar\lambda\bar\sigma_{-j,-2}\lambda^{(\epsilon(j)+1)/2}-\sigma_{*,-2}\bar\sigma_{-j,-1}$ $\lambda^{(\epsilon(j)+1)/2}=\tau_{*j}-e_j\equiv 0\bmod I,\tilde I ~\forall j\in\{3,-2,-1\}$ and $\sigma_{*,-1}\bar\lambda\bar\sigma_{0,-2}\mu-\sigma_{*,-2}\bar\sigma_{0,-1}\mu=\tau_{*0}-e_0\equiv 0\bmod I_0,\tilde I_0 $. It follows that $\sigma_{1,-2},\sigma_{2,-2},\sigma_{-3,-2},\sigma_{1,-1},\sigma_{2,-1},\sigma_{-3,-1}\in I$ and $\sigma_{0,-2},\sigma_{0,-1}\in \tilde I$ (multiply $\sigma'_{-1,*}$ resp. $\sigma'_{-2,*}$ from the left). Since that is a contradiction, $\tau\not\in \U_{7}((R,\Delta),(I,\Omega_{max}^I))$. Clearly $\tau_{*,-3}=e_{-3}$ and thus one can apply Proposition \ref{n4} to $\tau$.
\end{proof}
\begin{Lemma}\label{o3}
Let $H$ be an E-normal subgroup of $\U_{7}(R,\Delta)$ such that $\EU_{7}(I,\Omega^I_{min})\subseteq H$. If $H$ contains a matrix $\sigma$ of the form 
\[\sigma=\left(\begin{array}{ccc|c|ccc} 1&0&-&+&+&0&+\\0&1&-&+&+&0&+\\+&+&+&+&+&+&+\\\hline +&+&-&+&+&+&+\\\hline +&+&-&-&+&-&-\\0&0&+&+&+&1&+\\0&0&+&+&+&0&+\end{array}\right),\]
then $H$ contains an elementary matrix which is not $(I,\Omega^I_{max})$-elementary.
\end{Lemma}
\begin{proof}$~$\\
\underline{case 1} Assume that $\sigma_{13}\not\in I$ or $\sigma_{23}\not\in I$.\\
Set $\tau:=[\sigma^{-1},T_{-2,1}(1)]$. Then 
\begin{align*}
\tau
=&[\sigma^{-1},T_{-2,1}(1)]\\
=&(e+\sigma'_{*,-2}\sigma_{1*}-\sigma'_{*,-1}\lambda\sigma_{2*})T_{-2,1}(-1)\\
=&\left(e+\begin{pmatrix}\bar\lambda\bar\sigma_{2,-1}\\0\\\bar\lambda\bar\sigma_{2,-3}\\\sigma'_{0,-2}\\\bar\sigma_{23}\\\bar\sigma_{22}\\\bar\sigma_{21}\end{pmatrix}
\begin{pmatrix}\sigma_{11}&\sigma_{12}&\sigma_{13}&\sigma_{10}&\sigma_{1,-3}&0&\sigma_{1,-1}\end{pmatrix}\right.\\
&\left.-\begin{pmatrix}\bar\lambda\bar\sigma_{1,-1}\\0\\\bar\lambda\bar\sigma_{1,-3}\\\sigma'_{0,-1}\\\bar\sigma_{13}\\\bar\sigma_{12}\\\bar\sigma_{11}\end{pmatrix}\lambda
\begin{pmatrix}\sigma_{21}&\sigma_{22}&\sigma_{23}&\sigma_{20}&\sigma_{2,-3}&0&\sigma_{2,-1}\end{pmatrix}\right)T_{-2,1}(-1).
\end{align*}
Assume that $\tau\in \U_{7}((R,\Delta),(I,\Omega_{max}^I))$. Then $\tau_{*3}-e_3=\sigma'_{*,-2}\sigma_{13}-\sigma'_{*,-1}\lambda\sigma_{23}\equiv 0 \bmod I,\tilde I $. 
It follows that $\sigma_{13},\sigma_{23}\in I$ which is a contradiction. Hence $\tau\not\in \U_{7}((R,\Delta),(I,\Omega_{max}^I))$. Clearly $\tau_{*,-2}=e_{-2}$ and thus one can apply Proposition \ref{n4} to $\tau$.\\
\\
\underline{case 2} Assume that $\sigma_{13},\sigma_{23}\in I$.\\
It follows that $\sigma_{-3,-2},\sigma_{-3,-1}\in I$ (consider $b(\sigma_{*3},\sigma_{*,-2})$ and $b(\sigma_{*3},\sigma_{*,-1})$). Set $\xi:=T_{23}(-\sigma_{23})T_{2,-1}(-\sigma_{2,-1})\in \EU_{7}(I,\Omega_{min}^I)\subseteq H$ and $\tau:=\sigma\xi_3$. Then $\tau$ has the form
\[\tau=\left(\begin{array}{ccc|c|ccc} 1&0&+&+&+&+&+\\0&1&0&+&+&+&0\\+&+&+&+&+&+&+\\\hline +&+&-&+&+&+&+\\\hline +&+&-&-&+&+&+\\0&0&+&+&+&+&+\\0&0&+&+&+&+&+\end{array}\right)\]
Set $\omega :=[\tau,T_{31}(1)]$. Then 
\begin{align*}
\omega
=&[\tau,T_{31}(1)]\\
=&(e+\tau_{*3}\tau'_{1*}-\tau_{*,-1}\tau'_{-3,*})T_{31}(-1)\\
=&\left(e+\begin{pmatrix}\tau_{13}\\0\\\tau_{33}\\\tau_{03}\\\tau_{-3,3}\\\tau_{-2,3}\\\tau_{-1,3}\end{pmatrix}\bar\lambda
\begin{pmatrix}\bar\tau_{-1,-1}\lambda&\bar\tau_{-2,-1}\lambda&\bar\tau_{-3,-1}\lambda&\bar\tau_{0,-1}\mu &\bar\tau_{3,-1}&0&\bar\tau_{1,-1}\end{pmatrix}\right.\\
&\left(-\begin{pmatrix}\tau_{1,-1}\\0\\\tau_{3,-1}\\\tau_{0,-1}\\\tau_{-3,-1}\\\tau_{-2,-1}\\\tau_{-1,-1}\end{pmatrix}
\begin{pmatrix}\bar\tau_{-1,3}\lambda&\bar\tau_{-2,3}\lambda&\bar\tau_{-3,3}\lambda&\bar \tau_{03}\mu &\bar\tau_{33}&0&\bar\tau_{13}\end{pmatrix}\right)T_{31}(-1).
\end{align*}
Assume that $\omega \in \U_{7}((R,\Delta),(I,\Omega_{max}^I))$. Then 
\[\tau_{-3,3}\bar\lambda\bar\tau_{-1,-1}\lambda-\tau_{-3,-1}\bar\tau_{-1,3}\lambda-(\tau_{-3,3}\bar\lambda\bar\tau_{-3,-1}\lambda-\tau_{-3,-1}\bar\tau_{-3,3}\lambda)=\omega _{-3,1}\in I.\]
It follows that $\tau_{-3,3}\in I$ since $\tau_{-1,-1}\equiv 1 \bmod I $ and $\tau_{-3,-1}\in I$. Further 
\[\tau_{03}\bar\lambda\bar\tau_{-1,-1}\lambda-\tau_{0,-1}\bar\tau_{-1,3}\lambda-(\tau_{03}\bar\lambda\bar\tau_{-3,-1}\lambda-\tau_{0,-1}\bar\tau_{-3,3}\lambda)=\omega_{01}\in \tilde I.\]
It follows that $\tau_{03}\in \tilde I$ since $\tau_{-1,-1}\equiv 1 \bmod I $ and $\tau_{-1,3},\tau_{-3,-1},\tau_{-3,3}\in I$. Hence $\tau\in \U_{7}((R,\Delta),(I,\Omega_{max}^I))$ which is a contradiction. Therefore $\omega \not\in \U_{7}((R,\Delta),(I,$ $\Omega_{max}^I))$. Clearly $\omega_{*,-2}=e_{-2}$ and thus one can apply Proposition \ref{n4} to $\omega$.
\end{proof}
\begin{Lemma}\label{o4}
Let $H$ be an E-normal subgroup of $\U_{7}(R,\Delta)$ such that $\EU_{7}(I,\Omega^I_{min})\subseteq H$. If $H$ contains a matrix $\sigma$ of the form 
\[\sigma=\left(\begin{array}{ccc|c|ccc} -&*&*&*&+&+&+\\ -&*&*&*&+&+&+\\ 0&0&+&+&+&0&0\\\hline +&+&*&*&+&+&+\\\hline +&+&*&*&+&+&+\\ +&+&*&*&+&*&*\\ +&+&-&*&+&-&-\end{array}\right),\]
then $H$ contains an elementary matrix which is not $(I,\Omega^I_{max})$-elementary.
\end{Lemma}
\begin{proof}
Set $\tau:=[\sigma^{-1},T_{13}(1)]$. Then 
\begin{align*}
\tau
=&[\sigma^{-1},T_{13}(1)]\\
=&(e+\sigma'_{*1}\sigma_{3*}-\sigma'_{*,-3}\sigma_{-1,*})T_{13}(-1)\\
=&\left(e+\begin{pmatrix}\bar\lambda\bar\sigma_{-1,-1}\lambda\\\bar\lambda\bar\sigma_{-1,-2}\lambda\\ \bar\lambda\bar\sigma_{-1,-3}\lambda\\ \sigma'_{01}\\\bar\sigma_{-1,3}\lambda\\ \bar\sigma_{-1,2}\lambda \\ \bar\sigma_{-1,1}\lambda\end{pmatrix}
\begin{pmatrix}0&0&\sigma_{33}&\sigma_{30}&\sigma_{3,-3}&0&0\end{pmatrix}\right.\\
&\left.-\begin{pmatrix}0\\0\\ \bar\lambda\bar\sigma_{3,-3}\\\sigma'_{0,-3}\\ \bar\sigma_{33}\\0\\0\end{pmatrix}
\begin{pmatrix}\sigma_{-1,1}&\sigma_{-1,2}&\sigma_{-1,3}&\sigma_{-1,0}&\sigma_{-1,-3}&\sigma_{-1,-2}&\sigma_{-1,-1}\end{pmatrix}\right)T_{13}(-1).
\end{align*}
Assume that $\tau\in \U_{7}((R,\Delta),(I,\Omega_{max}^I))$. Then $\bar\lambda\bar\sigma_{-1,-1}\lambda\sigma_{33}-1=\tau_{13}\in I$ and $\bar\lambda\bar\sigma_{-1,-2}\lambda\sigma_{33}=\tau_{23}\in I$. Since $\sigma_{33}\equiv 1 \bmod I $, it follows that $\sigma_{-1,-1}-1, \sigma_{-1,-2}\in I$. Since by assumption $\tau\in \U_{7}((R,\Delta),(I,\Omega_{max}^I))$,
\begin{align*}
\tau_{*3}-e_3=\sigma'_{*1}\sigma_{33}-\sigma'_{*,-3}\sigma_{-1,3}-\begin{pmatrix}1&0&-\bar\lambda\bar\sigma_{3,-3}\sigma_{-1,1}&-\sigma'_{0,-3}\sigma_{-1,1}&-\bar\sigma_{33}\sigma_{-1,1} &0&0\end{pmatrix}^t
\end{align*}
is congruent to $0$ modulo $I,\tilde I$. By multiplying $\sigma_{-3,*}$ from the left we get that $\sigma_{-1,3}\in I$ since $\sigma_{-3,1}, \sigma_{-1,1}\in I$. Further $\bar\sigma_{-1,3}\lambda\sigma_{30}-\bar\sigma_{33}\sigma_{-1,0}=\tau_{-3,0}\in I_0$. Since $\sigma_{-1,3}\in I$ and $\sigma_{33}\equiv 1\bmod I $, it follows that $\sigma_{-1,0}\in I_0$ (note that $I\subseteq I_0$). Hence $\sigma_{-1,*}\equiv e_{-1}^t\bmod I,I_0 $ and therefore $\sigma_{*1}\equiv e_1\bmod I,\tilde I $ (follows from Lemma \ref{20}). This implies $\sigma_{11}-1,\sigma_{21}\in I$. Since that is a contradiction, $\tau\not\in \U_{7}((R,\Delta),(I,\Omega_{max}^I))$. One checks easily that $\tau$ has the form 
\[\tau=\left(\begin{array}{ccc|c|ccc} 1&0&-&+&+&0&+\\0&1&-&+&+&0&+\\+&+&+&+&+&+&+\\\hline +&+&-&+&+&+&+\\\hline +&+&-&-&+&-&-\\0&0&+&+&+&1&+\\0&0&+&+&+&0&+\end{array}\right)\]
and thus one can apply the previous lemma to $\tau$.
\end{proof}
\begin{Lemma}\label{o5}
Let $H$ be an E-normal subgroup of $\U_{7}(R,\Delta)$ such that $\EU_{7}(I,\Omega^I_{min})\subseteq H$. If $H$ contains a matrix $\sigma$ of the form 
\[\sigma=\left(\begin{array}{ccc|c|ccc} +&-&-&*&+&+&+\\ +&-&-&*&+&+&+\\ 0&0&+&+&+&0&0\\\hline +&+&*&*&+&+&+\\\hline +&+&*&*&+&+&+\\ +&+&-&*&+&-&-\\ +&+&+&*&+&+&+\end{array}\right),\]
then $H$ contains an elementary matrix which is not $(I,\Omega^I_{max})$-elementary.
\end{Lemma}
\begin{proof}
Set $\tau:=[\sigma,T_{21}(1)]$. Then 
\begin{align*}
\tau
=&[\sigma,T_{21}(1)]\\
=&(e+\sigma_{*2}\sigma'_{1*}-\sigma_{*,-1}\sigma'_{-2,*})T_{21}(-1)\\
=&\left(e+\begin{pmatrix}\sigma_{12}\\\sigma_{22}\\0\\\sigma_{02}\\ \sigma_{-3,2}\\ \sigma_{-2,2} \\ \sigma_{-1,2}\end{pmatrix}\bar\lambda
\begin{pmatrix}\bar\sigma_{-1,-1}\lambda&\bar\sigma_{-2,-1}\lambda&\bar\sigma_{-3,-1}\lambda&\bar\sigma_{0,-1}\mu &0&\bar\sigma_{2,-1}&\bar\sigma_{1,-1}\end{pmatrix}\right.\\
&\left.-\begin{pmatrix}\sigma_{1,-1}\\\sigma_{2,-1}\\0\\ \sigma_{0,-1}\\\sigma_{-3,-1}\\ \sigma_{-2,-1} \\ \sigma_{-1,-1}\end{pmatrix}
\begin{pmatrix}\bar\sigma_{-1,2}\lambda&\bar\sigma_{-2,2}\lambda&\bar\sigma_{-3,2}\lambda&\bar\sigma_{02}\mu &0&\bar\sigma_{22}&\bar\sigma_{12}\end{pmatrix}\right)T_{21}(-1).
\end{align*}
Assume that $\tau\in \U_{7}((R,\Delta),(I,\Omega_{max}^I))$. Then 
$\sigma_{*2}\bar\lambda\bar\sigma_{1,-1}-\sigma_{*,-1}\bar\sigma_{12}=\tau_{*,-1}-e_{-1}\equiv 0\bmod I,\tilde I $. It follows that $\sigma_{12}\in I$. Further $\sigma_{*2}\bar\lambda\bar\sigma_{2,-1}-\sigma_{*,-1}\bar\sigma_{22}+e_{-1}\equiv\tau_{*,-2}-e_{-2}\equiv 0\bmod I,\tilde I $. By multiplying $\sigma'_{-1,*}$ from the left we get that $-\bar\sigma_{22}+\bar\sigma_{11}\in I$ which implies $\sigma_{22}\equiv 1 \bmod I $ since $\sigma_{11}\equiv 1\bmod I $. Hence $\sigma_{*2}\equiv e_2\bmod I,\tilde I $ and therefore $\sigma_{-2,*}\equiv e_{-2}^t\bmod I,I_0$. This implies $\sigma_{*,-2}\equiv e_{-2}\bmod I,\tilde I $ and $\sigma_{*,-1}\equiv e_{-1}\bmod I,\tilde I $ and therefore $\sigma_{2*}\equiv e_{2}^t\bmod I,I_0 $ and $\sigma_{1*}\equiv e_{1}^t\bmod I,I_0 $. 
Thus $\sigma_{-2,3},\sigma_{-2,-2}-1,\sigma_{-2,-1},\sigma_{13},\sigma_{23}\in I$. Since that is a contradiction, $\tau\not\in \U_{7}((R,\Delta),(I,$ $\Omega_{max}^I))$. Clearly $\tau_{*,-3}=e_{-3}$ and thus one can apply Proposition \ref{n4} to $\tau$.
\end{proof}
\begin{Lemma}\label{o6}
Let $H$ be an E-normal subgroup of $\U_{7}(R,\Delta)$ such that $\EU_{7}(I,\Omega^I_{min})\subseteq H$. If $H$ contains a matrix $\sigma$ of the form 
\[\sigma=\left(\begin{array}{ccc|c|ccc} +&+&+&*&+&+&+\\ +&+&+&*&+&+&+\\ 0&0&+&+&+&0&0\\\hline +&+&-&*&+&+&+\\\hline +&+&-&*&+&+&+\\ +&+&+&*&+&+&+\\ +&+&+&*&+&+&+\end{array}\right)\]
such that $\sigma_{33}\equiv 1\bmod \mathrm{rad}(R)$ where $\mathrm{rad}(R)$ denotes the Jacobson radical of $R$, then $H$ contains an elementary matrix which is not $(I,\Omega^I_{max})$-elementary.
\end{Lemma}
\begin{proof}
By \cite[Chapter III, Lemma 4.3]{bak_2} there are $y_1, y_2\in I$ such that $\sigma_{-1,3}+y_2(y_1\sigma_{-1,3}+\sigma_{-2,3})\in \mathrm{rad} (R)$. Set $\xi:=T_{-1,-2}(y_2)T_{-2,-1}(y_1)\in \EU_{7}(I,\Omega^I_{min})\subseteq H$ and $\tau:=\xi\sigma$. Then $\tau\equiv\sigma \bmod I,\tilde I,I_0,\tilde I_0 $, $\tau_{33}=\sigma_{33}\equiv 1 \bmod \mathrm{rad} (R)$ and $\tau_{-1,3}=\sigma_{-1,3}+y_2(y_1\sigma_{-1,3}+\sigma_{-2,3})\in \mathrm{rad} (R)$. Set $\omega:={}^{T_{3,-1}(-1)}\tau$. Then $\omega_{3*}\equiv e_3^t\bmod I,I_0 $, $\omega_{33}\equiv 1 \bmod \mathrm{rad}(R)$,  $\omega_{0,-1}=\tau_{0,-1}+\tau_{03}\equiv \tau_{03}\equiv\sigma_{03}\bmod \tilde I$ and $\omega_{-3,-1}=\tau_{-3,-1}+\tau_{-3,3}\equiv \tau_{-3,3}\equiv\sigma_{-3,3}\bmod I $. Hence $\omega_{0,-1}\not\in \tilde I\lor\omega_{-3,-1}\not\in I$. By Nakayama's lemma $\omega_{33}$ is invertible. Set $\zeta:=T_{3,-1}(-(\omega_{33})^{-1}\omega_{3,-1})T_{3,-2}(-(\omega_{33})^{-1}\omega_{3,-2})\in \EU_{7}(I,\Omega^I_{min})\subseteq H$ and $\psi:=\omega\zeta$. Then $\psi$ has the form
\[\psi=\left(\begin{array}{ccc|c|ccc} *&*&*&*&*&*&*\\ *&*&*&*&*&*&*\\ *&*&*&*&*&0&0\\\hline *&*&*&*&*&*&-\\\hline *&*&*&*&*&*&-\\ *&*&*&*&*&*&*\\ *&*&*&*&*&*&*\end{array}\right).\]
Thus one can apply Lemma \ref{o2} to $\psi$.
\end{proof}
\begin{Lemma}\label{o7}
Let $H$ be an E-normal subgroup of $\U_{7}(R,\Delta)$ such that $\EU_{7}(I,\Omega^I_{min})\subseteq H$. If $H$ contains a matrix $\sigma$ of the form 
\[\sigma=\left(\begin{array}{ccc|c|ccc} +&+&+&*&+&+&+\\ +&+&+&*&+&+&+\\ 0&0&+&+&+&0&0\\\hline +&+&+&-&+&+&+\\\hline +&+&+&*&+&+&+\\ +&+&+&*&+&+&+\\ +&+&+&*&+&+&+\end{array}\right)\]
such that $\sigma_{33}$ is invertible, then $H$ contains an elementary matrix which is not $(I,\Omega^I_{max})$-elementary.
\end{Lemma}
\begin{proof}
Since $\sigma_{00}\not\equiv 1\bmod \tilde I_0$, there is an $a\in J(\Delta)$ such that $(\sigma_{00}-1) a\not\in \tilde I$. Choose a $b\in R$ such that $(a,b)\in \Delta$ and set $\tau:={}^{T_{1}(a,\underline{b})}\sigma$. Then $\tau_{3,-2}=0$, $\tau_{33}=\sigma_{33}$, $\tau_{3,-1}\in I$ and $\tau_{0,-1}\not\in \tilde I$. Set $\xi:=T_{3,-1}(-(\tau_{33})^{-1}\tau_{3,-1})\in \EU_{7}(I,\Omega^I_{min})$ and $\omega:=\tau\xi$. Then $\omega$ has the form
\[\omega=\left(\begin{array}{ccc|c|ccc} *&*&*&*&*&*&*\\ *&*&*&*&*&*&*\\ *&*&*&*&*&0&0\\\hline *&*&*&*&*&*&-\\\hline *&*&*&*&*&*&*\\ *&*&*&*&*&*&*\\ *&*&*&*&*&*&*\end{array}\right).\]
Thus one can apply Lemma \ref{o2} to $\omega$.\\
\\
\end{proof}
\begin{Corollary}\label{o8}
Let $H$ be an E-normal subgroup of $\U_{7}(R,\Delta)$ such that $\EU_{7}(I,\Omega^I_{min})\subseteq H$. If $H$ contains a matrix $\sigma\not\in \U_{7}((R,\Delta),(I,\Omega_{max}^I))$ of the form 
\[\sigma=\left(\begin{array}{ccc|c|ccc} *&*&*&*&+&*&*\\ *&*&*&*&+&*&*\\ 0&0&+&+&+&0&0\\\hline *&*&*&*&+&*&*\\\hline *&*&*&*&+&*&*\\ *&*&*&*&+&*&*\\ *&*&*&*&+&*&*\end{array}\right)\]
such that $\sigma_{33}\equiv 1\bmod \mathrm{rad}(R)$, then $H$ contains an elementary matrix which is not $(I,\Omega^I_{max})$-elementary.
\end{Corollary}
\begin{proof}
Follows from the Lemmas \ref{o1}, \ref{o2}, \ref{o4}, \ref{o5}, \ref{o6} and \ref{o7}.
\end{proof}
\begin{Lemma}\label{o9}
Let $H$ be an E-normal subgroup of $\U_{7}(R,\Delta)$ such that $\EU_{7}(I,\Omega^I_{min})\subseteq H$. If $H$ contains a matrix $\sigma$ of the form 
\[\sigma=\left(\begin{array}{ccc|c|ccc} *&*&*&*&*&*&*\\ *&*&*&*&*&*&*\\ 0&0&1&0&-&-&*\\\hline +&+&*&*&*&*&*\\\hline +&+&*&*&*&*&*\\ +&+&*&*&*&*&*\\ +&+&*&*&*&*&*\end{array}\right),\]
then $H$ contains an elementary matrix which is not $(I,\Omega^I_{max})$-elementary.
\end{Lemma}
\begin{proof}
By \cite[chapter III, Lemma 4.3]{bak_2} there are $x_1, x_2\in I$ such that $\sigma_{-1,1}+x_2(x_1\sigma_{-1,1}+\sigma_{-2,1})\in \mathrm{rad} (R)$. Set $\xi:=T_{-1,-2}(x_2)T_{-2,-1}(x_1)\in \EU_{7}(I,\Omega^I_{min})\subseteq H$. Then $\tau:=\xi\sigma$ has the form
\[\tau=\left(\begin{array}{ccc|c|ccc} *&*&*&*&*&*&*\\ *&*&*&*&*&*&*\\ 0&0&1&0&-&-&*\\\hline +&+&*&*&*&*&*\\\hline +&+&*&*&*&*&*\\ +&+&*&*&*&*&*\\ +&+&*&*&*&*&*\end{array}\right)\]
and $\tau_{-1,1}=\sigma_{-1,1}+x_2(x_1\sigma_{-1,1}+\sigma_{-2,1})\in \mathrm{rad} (R)$. Set $\omega:=[\tau^{-1},T_{13}(1)]$. Then
\begin{align*}
\omega
=&[\tau^{-1},T_{13}(1)]\\
=&(e+\tau'_{*1}\tau_{3*}-\tau'_{*,-3}\tau_{-1,*})T_{13}(-1)\\
=&\left(e+\begin{pmatrix}\bar\lambda\bar\tau_{-1,-1}\lambda\\\bar\lambda\bar\tau_{-1,-2}\lambda\\\bar\lambda\tau_{-1,-3}\lambda\\\tau'_{01}\\\bar\tau_{-1,3}\lambda\\\bar\tau_{-1,2}\lambda\\ \bar\tau_{-1,1}\lambda\end{pmatrix}
\begin{pmatrix}0&0&1&0&\tau_{3,-3}&\tau_{3,-2}&\tau_{3,-1}\end{pmatrix}\right.\\
&\left.-\begin{pmatrix}\bar\lambda\bar\tau_{3,-1}\\\bar\lambda\bar\tau_{3,-2}\\\bar\lambda\bar\tau_{3,-3}\\\tau'_{0,-3}\\1\\0\\0\end{pmatrix}
\begin{pmatrix}\tau_{-1,1}&\tau_{-1,2}&\tau_{-1,3}&\tau_{-1,0}&\tau_{-1,-3}&\tau_{-1,-2}&\tau_{-1,-1}\end{pmatrix}\right)T_{13}(-1).
\end{align*}
Assume that $\omega\in \U_{7}((R,\Delta),(I,\Omega_{max}^I))$. Then $\tau'_{*1}\tau_{3j}-\tau'_{*,-3}\tau_{-1,j}=\omega_{*j}-e_j\equiv 0\bmod I,\tilde I ~\forall j\in\{-3,-2\}$.
By multiplying $\tau_{1*}$ from the left we get that $\tau_{3,-3},\tau_{3,-2}\in I$. Since that is a contradiction, $\omega\not\in \U_{7}((R,\Delta),(I,\Omega_{max}^I))$. 
Obviously $\omega_{*1}\equiv e_1\bmod I,\tilde I $, $\omega_{-1,*}\equiv e^t_{-1}\bmod I,I_0 $ and $\omega_{-1,-1}\equiv 1 \bmod \mathrm{rad}(R)$ (note that $\mathrm{rad}(R)$ is involution invariant since $\bar~$ defines a bijection between maximal left and maximal right ideals of $R$). Set $\psi:={}^{P_{3,-1}}\omega$. Then $\psi\not\in \U_{7}((R,\Delta),(I,\Omega_{max}^I))$. Further $\psi_{*,-3}\equiv e_{-3}\bmod I,\tilde I $, $\psi_{3*}\equiv e^t_3\bmod I,I_0 $ and $\psi_{33}\equiv 1 \bmod \mathrm{rad}(R)$. By Nakayama's lemma $\psi_{33}$ is invertible. Set $\zeta:=T_{32}(-(\psi_{33})^{-1}\psi_{32})T_{31}(-(\psi_{33})^{-1}\psi_{31})$ $T_{3,-1}(-(\psi_{33})^{-1}\psi_{3,-1})T_{3,-2}(-(\psi_{33})^{-1}\psi_{3,-2})\subseteq \EU_{7}(I,\Omega^I_{min})\subseteq H$ and $\eta:=\psi\zeta$. Then $\eta$ has the form
\[\eta=\left(\begin{array}{ccc|c|ccc} *&*&*&*&+&*&*\\ *&*&*&*&+&*&*\\ 0&0&+&+&+&0&0\\\hline *&*&*&*&+&*&*\\\hline *&*&*&*&+&*&*\\ *&*&*&*&+&*&*\\ *&*&*&*&+&*&*\end{array}\right)\]
and further $\eta_{33}\equiv 1 \bmod \mathrm{rad}(R)$. Since $\psi\not\in \U_{7}((R,\Delta),(I,\Omega_{max}^I))$, we have $\eta\not\in \U_{7}((R,\Delta),(I,\Omega_{max}^I))$. Thus we can apply Corollary \ref{o8} to $\eta$.
\end{proof}
\begin{Proposition}\label{p2}
Suppose that $R$ is semilocal. Let $H$ be an E-normal subgroup of $\U_{7}(R,\Delta)$ such that $\EU_{7}(I,\Omega^I_{min})\subseteq H$. If $H\not\subseteq \CU_{7}((R,\Delta),(I,\Omega_{max}^I))$, then $H$ contains an elementary matrix which is not $(I,\Omega^I_{max})$-elementary. 
\end{Proposition}
\begin{proof}
By Proposition \ref{123}, $H$ contains a matrix $\sigma\not\in \U_{7}((R,\Delta),(I,\Omega^{I}_{max}))$ of the form
\[\sigma=\left(\begin{array}{ccc|c|ccc} *&*&*&*&*&*&*\\ *&*&*&*&*&*&*\\0&0&1&0&*&*&*\\\hline *&*&*&*&*&*&*\\\hline *&*&*&*&*&*&*\\ *&*&*&*&*&*&*\\ *&*&*&*&*&*&*\end{array}\right).\]
By Lemma \ref{o1} and the previous lemma we may assume that 
\[\sigma=\left(\begin{array}{ccc|c|ccc} *&*&*&*&*&*&*\\ *&*&*&*&*&*&*\\0&0&1&0&+&+&*\\\hline +&+&*&*&*&*&*\\\hline +&+&*&*&*&*&*\\ +&+&*&*&*&*&*\\ +&+&*&*&*&*&*\end{array}\right).\]
\underline{case 1} Assume that $\sigma_{3,-1}\not\in I$.\\
Set $\tau:={}^{T_{21}(1)}\sigma$. Clearly $\tau_{01},\tau_{02}\in \tilde I$ and $\tau_{-3,1},\tau_{-2,1},\tau_{-1,1},\tau_{-3,2},\tau_{-2,2},$ $\tau_{-1,2}\in I$. Further \[\tau_{3*}=\begin{pmatrix}0&0&1&0&\sigma_{3,-3}&\sigma_{3,-2}+\sigma_{3,-1}&\sigma_{3,-1}\end{pmatrix}.\]
Since $\sigma_{3,-2}\in I$ and $\sigma_{3,-1}\not\in I$, we have $\tau_{3,-2}=\sigma_{3,-2}+\sigma_{3,-1}\not\in I$. Thus one can apply the previous lemma to $\tau$.\\
\\
\underline{case 2} Assume that $\sigma_{3,-1}\in I$.\\
Then $\sigma_{3*}\equiv e_3^t\bmod I,I_0 $ and hence $\sigma_{*,-3}\equiv e_{-3}\bmod I,\tilde I $. Set $\xi:=T_{3,-1}(-\sigma_{3,-1})T_{3,-2}(-\sigma_{3,-2})\subseteq \EU_{7}(I,\Omega^I_{min})\subseteq H$ and $\tau:=\sigma\xi$. Then $\tau$ has the form
\[\tau=\left(\begin{array}{ccc|c|ccc} *&*&*&*&+&*&*\\ *&*&*&*&+&*&*\\ 0&0&1&0&+&0&0\\\hline *&*&*&*&+&*&*\\\hline *&*&*&*&+&*&*\\ *&*&*&*&+&*&*\\ *&*&*&*&+&*&*\end{array}\right).\]
Since $\sigma\not\in \U_{7}((R,\Delta),(I,\Omega_{max}^I))$, we have $\tau\not\in \U_{7}((R,\Delta),(I,\Omega_{max}^I))$. Thus we can apply Corollary \ref{o8} to $\tau$.
\end{proof}
\subsection{General case}
In this subsection $(I,\Omega)$ denotes an odd form ideal of $(R,\Delta)$.
\begin{Lemma}\label{71}
Let $\sigma\in \U_{2n+1}(R,\Delta)$. Further let $i,j\in \Theta_{hb}$ such that $i\neq \pm j$, $x\in R$ and $(y,z)\in \Delta^{-\epsilon(i)}$. Set $\tau:=[\sigma,T_{ij}(x)]$ and $\rho:=[\sigma,T_{i}(y,z)]$. Further let $J(\sigma)$ denote the ideal generated by $\{\sigma_{kl},\sigma'_{kl}|k,l\in \Theta_{hb}, k\neq l\}\cup\{\bar a\mu\sigma_{0l},\bar a\mu\sigma'_{0l}|a\in J(\Delta), l\in \Theta_{hb}\}$ and $J'(\sigma)$ the left ideal generated by $\{\sigma_{k0},\sigma'_{k0}|k\in\Theta_{hb}\}$. Then 
\[q(\tau_{*k})=(\delta_{0k},0)\+ q(\sigma_{*i})\circ x\sigma'_{jk}\+ q(\sigma_{*,-j})\circ\tilde x\sigma'_{-i,k}\+(0,y_k-\bar y_k\lambda),\]
if $k\neq j,-i$,
\[q(\tau_{*j})=q(\sigma_{*i})\circ x\sigma'_{jj}\+ q(\sigma_{*,-j})\circ\tilde x \sigma'_{-i,j}\+ q(\tau_{*i})\circ(-x)\+(0,y_j-\bar y_j\lambda)\]
and
\[q(\tau_{*,-i})=q(\sigma_{*i})\circ x\sigma'_{j,-i}\+ q(\sigma_{*,-j})\circ\tilde x \sigma'_{-i,-i}\+ q(\tau_{*,-j})\circ(-\tilde x)\+ (0,y_{-i}-\bar y_{-i}\lambda)\]
where $\tilde x=-\lambda^{(\epsilon(j)-1)/2}\bar x\lambda^{(1-\epsilon(i)/2}$, $y_k\in J(\sigma)~\forall k\in\Theta_{hb}$ and $y_0\in J'(\sigma)$.\\
Further 
\begin{align*}
q(\rho_{*k})=&(q(\sigma_{*0})\minus(1,0))\circ y\sigma'_{-i,k}\+ q(\sigma_{*i})\circ \hat y\sigma'_{0k}\+ q(\sigma_{*i})\circ z\sigma'_{-i,k}\\
&\+ a_k\circ \sigma'_{-i,k}\+(0,z_k-\bar z_k\lambda),
\end{align*}
if $k\neq 0,-i$,
\begin{align*}q(\rho_{*0})=&(1,0)\+ (q(\sigma_{*0})\minus(1,0))\circ y\sigma'_{-i,0}\+ q(\sigma_{*i})\circ \hat y\sigma'_{00}\+ q(\sigma_{*i})\circ z \sigma'_{-i,0}\\
&\+ q(\rho_{*i})\circ(-\hat y)\+ a_0\circ \sigma'_{-i,0}\+(0,z_0-\bar z_0\lambda)
\end{align*}
and
\begin{align*}
q(\rho_{*,-i})=&(q(\sigma_{*0})\minus(1,0))\circ y\sigma'_{-i,-i}\+ q(\sigma_{*i})\circ\hat y \sigma'_{0,-i}\+ q(\sigma_{*i})\circ z\sigma'_{-i,-i}\\
&\+(q(\sigma_{*0})\minus(1,0))\circ(-y\sigma'_{-i,0}y)\+ q(\sigma_{*i})\circ(-\hat y\sigma'_{00}y)\+ q(\sigma_{*i})\circ(-z\sigma'_{-i,0}y)\\&\+ q(\rho_{*i})\circ \hat z\+ a_{-k}\circ (\sigma'_{-i,-i}-1)\+ b\+ c\circ (-\sigma'_{-i,0}y)\+ d \+(0,z_{-i}-\bar z_{-i}\lambda)\end{align*}
where $\hat y=-\lambda^{-(1+\epsilon(i))/2}\bar y\mu $, $\hat z=\lambda^{-(1+\epsilon(i))/2}\bar z\lambda^{(1-\epsilon(i))/2}$,
\begin{align*}
a_k=\begin{cases} (y,\lambda^{(\epsilon(i)+1)/2}z), &\mbox{if } k\in\Theta\setminus\Theta_-, \\ 
(y,\bar z\lambda^{(1-\epsilon(i))/2}),& \mbox{if } k\in\Theta_-, \end{cases}
\end{align*}
\begin{align*}
b=\begin{cases} (0,\bar z(\sigma'_{-1,-1}-1)-\overline{\bar z(\sigma'_{-1,-1}-1)}\lambda), &\mbox{if } i\in\Theta_+, \\ 
(0,z(\sigma'_{-1,-1}-1)-\overline{z(\sigma'_{-1,-1}-1)}\lambda),& \mbox{if } i\in\Theta_-, \end{cases}
\end{align*}
\begin{align*}
c=\begin{cases} (y,\lambda z), &\mbox{if } i\in\Theta_+, \\ 
(y,z),& \mbox{if } i\in\Theta_-, \end{cases}
\end{align*}
\begin{align*}
d=\begin{cases} (0,\bar\sigma'_{-i,-i}\bar y \bar \sigma_{00}\mu y-\overline{\bar\sigma'_{-i,-i}\bar y \bar \sigma_{00}\mu y}\lambda), &\mbox{if } i\in\Theta_+, \\ 
0,& \mbox{if } i\in\Theta_-, \end{cases}
\end{align*}
$z_k\in J(\sigma)~\forall k\in \Theta_{hb}$ and $z_0\in J(\sigma)+J'(\sigma)$.
\end{Lemma}
\begin{proof}
Straightforward computation.
\end{proof}
\begin{Lemma}\label{q1}
Let $H$ be an E-normal subgroup of $\U_{2n+1}(R,\Delta)$ such that $\EU_{2n+1}(I,\Omega)\subseteq H$. If $H$ contains a matrix $\sigma\in \U_{2n+1}((R,\Delta),(I,\Omega^{I}_{max}))\setminus \U_{2n+1}((R,\Delta),(I,\Omega))$ of the form
\[\sigma=\begin{pmatrix} 1&0&0\\0&A&0\\0&0&1\end{pmatrix}.\]
where $A\in \M_{2n-1}(R)$, then $H$ contains an elementary matrix which is not $(I,\Omega)$-elementary.
\end{Lemma}
\begin{proof}
By Lemma \ref{39}, there is a $j\in \{2,\dots,-2\}\setminus\{0\}$ such that $q(\sigma_{*j})\not\in\Omega$ or there is a $x\in J(\Delta)$ such that $(q(\sigma_{*0})\minus(1,0))\circ x\not\in \Omega$.\\
\\
\underline{case 1} Assume that there is a $j\in \{2,\dots,-2\}\setminus\{0\}$ such that $q(\sigma_{*j})\not\in\Omega$.\\
We only consider the subcase that $\epsilon(j)=1$ and leave the subcase $\epsilon(j)=-1$ to the reader. Set $\tau:=[\sigma,T_{j1}(1)]$. It is easy to show that 
\[\tau=T_{-1}(q(\sigma_{*j})\+(0,-\sigma_{-j,j}+\bar\sigma_{-j,j}\lambda))\prod\limits_{\substack{i=2\\i\neq 0}}^{-2}T_{i1}(\sigma_{ij}-\delta_{ij}).\]
Since $\prod\limits_{\substack{i=2\\i\neq 0}}^{-2}T_{i1}(\sigma_{ij}-\delta_{ij})\in \EU_{2n+1}(I,\Omega)\subseteq H$, it follows that $T_{-1}(q(\sigma_{*j})\+(0,-\sigma_{-j,j}+\bar\sigma_{-j,j}\lambda))\in H$. Clearly $T_{-1}(q(\sigma_{*j})\+(0,-\sigma_{-j,j}+\bar\sigma_{-j,j}\lambda))$ is not $(I,\Omega)$-elementary since $q(\sigma_{*j})\not\in\Omega$ and $(0,-\sigma_{-j,j}+\bar\sigma_{-j,j}\lambda))\in\Omega_{min}^I\subseteq \Omega$. \\
\\
\underline{case 2} Assume there is a $x\in J(\Delta)$ such that $(q(\sigma_{*0})\minus(1,0))\circ x\not\in \Omega$.\\
Choose a $y\in R$ such that $(x,y)\in \Delta$ and set $\tau:=[\sigma,T_{-1}(x,y)]$. It is easy to show that 
\[\tau=T_{-1}((q(\sigma_{*0})\minus(1,0))\circ x)\prod\limits_{\substack{i=2\\i\neq 0}}^{-2}T_{i1}(\sigma_{i0}x).\]
Since $\prod\limits_{\substack{i=2\\i\neq 0}}^{-2}T_{i1}(\sigma_{i0}x)\in \EU_{2n+1}(I,\Omega)\subseteq H$, it follows that $T_{-1}((q(\sigma_{*0})\minus(1,0))\circ x)\in H$. Clearly $T_{-1}((q(\sigma_{*0})\minus(1,0))\circ x)\in H$ is not $(I,\Omega)$-elementary since $(q(\sigma_{*0})\minus(1,0))\circ x\not\in \Omega$.
\end{proof} 
\begin{Lemma}\label{q2}
Let $H$ be an E-normal subgroup of $\U_{2n+1}(R,\Delta)$ such that $\EU_{2n+1}(I,\Omega)\subseteq H$. If $H$ contains a matrix $\sigma\in \U_{2n+1}((R,\Delta),(I,\Omega^{I}_{max}))\setminus \U_{2n+1}((R,\Delta),(I,\Omega))$ such that $\sigma_{11}=1$, then $H$ contains an elementary matrix which is not $(I,\Omega)$-elementary.
\end{Lemma}
\begin{proof}$~$\\
\underline{case 1} Assume that $q(\sigma_{*1}),q(\sigma_{*,-1})\in\Omega$.\\
Set 
\[\xi:=T_{1}(\minus_{-1} q^{-1}(\sigma_{*,-1}))\prod\limits_{\substack{i=2\\i\neq 0}}^{-2}T_{i,-1}(-\sigma_{i,-1})T_{-1}(\minus q(\sigma_{*1}))\prod\limits_{\substack{i=-2\\i\neq 0}}^2T_{i1}(-\sigma_{i1})\in \EU_{2n+1}(I,\Omega)\subseteq H\]
where $q^{-1}($ $\sigma_{*,-1})=(q_1(\sigma_{*,-1}),\underline{q_2(\sigma_{*,-1})})\in \Omega^{-1}$. Further set $\tau:=\xi\sigma$. Then there is an $A\in \M_{2n-1}(R)$ such that 
\[\tau=\begin{pmatrix} 1&0&0\\0&A&0\\0&0&1\end{pmatrix}.\]
Clearly $\tau\in \U_{2n+1}((R,\Delta),(I,\Omega^{I}_{max}))\setminus \U_{2n+1}((R,\Delta),(I,\Omega))$ and thus we can apply the previous lemma to $\tau$. \\
\\
\underline{case 2} Assume that $q(\sigma_{*1})\not\in\Omega$ or $q(\sigma_{*,-1})\not\in\Omega$.\\
We consider only the subcase $q(\sigma_{*1})\not\in\Omega$ and leave the subcase $q(\sigma_{*,-1})\not\in\Omega$ to the reader. Set $\xi:=T_{13}(-\sigma_{13})T_{1,-2}(-\sigma_{1,-2})\in \EU_{2n+1}(I,\Omega)\subseteq H$ and $\tau:=\sigma\xi$. Then 
clearly $\tau_{11}=1$, $\tau_{13}=\tau_{1,-2}=0$, $\tau\in \U_{2n+1}((R,\Delta),(I,\Omega^{I}_{max}))$ and $q(\tau_{*1})\not\in \Omega$.\\
\\
\underline{case 2.1} Assume that $q(\tau_{*,-2})\in\Omega$.\\ 
Set $\zeta:=\prod\limits_{\substack{i=-2\\i\neq 0}}^2T_{i1}(-\tau_{i1})\in \EU_{2n+1}(I,\Omega)\subseteq H$ and $\omega:=\zeta\tau$. Further set $\chi:=T_{-1}(\minus q(\tau_{*1}))\in \EU_{2n+1}((R,\Delta),$ $(I,\Omega_{max}^I))$. Then $\psi:=\chi\omega$ has the form \[\psi=\begin{pmatrix}1&*&*\\0&A&*\\0&0&1\end{pmatrix}\]
where $A\in \M_{2n-1}(R)$. One checks easily that $\psi\in \U_{2n+1}((R,\Delta),(I,\Omega^{I}_{max}))$ and $q(\psi_{*,-2})\in\Omega$. Set $g:=T_{12}(-1)$. Using the equality $[\alpha\beta,\gamma]=$ $^{\alpha}[\beta,\gamma][\alpha,\gamma]$ one gets that $[\omega,g]=[\chi^{-1}\psi,g]={}^{\chi^{-1}}[\psi,g][\chi^{-1},g]$. It is easy to show that $[\psi,g]\in \EU_{2n+1}(I,\Omega)$ and hence $^{\chi^{-1}}[\psi,g]\in \EU_{2n+1}((R,\Delta),(I,\Omega))\subseteq H$. On the other hand $[\chi^{-1},g]=T_{-1,2}(-q_2(\tau_{*1}))T_{-2}(\minus q(\tau_{*1}))$ by (S1), (SE1) and (SE2). Since $T_{-1,2}(-q_2(\tau_{*1}))\in H$, it follows that $T_{-2}(\minus q(\tau_{*1}))\in H$. Clearly $T_{-2}(\minus q(\tau_{*1}))$ is not $(I,\Omega)$-elementary since $q(\tau_{*1})\not\in\Omega$.\\
\\
\underline{case 2.2} Assume that $q(\tau_{*,-2})\not\in\Omega$.\\
Set $\omega:=[\tau,T_{-2,-3}(1)]$. Then clearly $\omega_{11}=1$ and $\omega\in \U_{2n+1}((R,\Delta),(I,\Omega_{max}^I))$. Further, by Lemma \ref{71}, $q(\omega_{*,-3})\equiv q(\tau_{*,-2})\bmod \Omega$ (which implies $\omega\not\in \U_{2n+1}((R,\Delta),(I,\Omega))$ since $q(\tau_{*,-2})\not\in\Omega$) and $q(\omega_{*1}),$ $q(\omega_{*,-1})\in\Omega$. Thus one can apply case 1 to $\omega$.
\end{proof}
\begin{Theorem}[Extraction Theorem]\label{extraction}
Suppose $R$ is semilocal. Let $H$ be an E-normal subgroup of $\U_{2n+1}(R,\Delta)$ such that $\EU_{2n+1}(I,$ $\Omega)\subseteq H$. If $H\not\subseteq \CU_{2n+1}((R,\Delta),(I,\Omega))$, then $H$ contains an elementary matrix which is not $(I,\Omega)$-elementary.
\end{Theorem}
\begin{proof}
By Proposition \ref{p1} and Proposition \ref{p2} we may assume that $H\subseteq \CU_{2n+1}((R,\Delta),(I,\Omega_{max}^I))$. Choose an $h\in H\setminus \CU_{2n+1}((R,\Delta),(I,\Omega))$. Then, by the definition of $\CU_{2n+1}((R,\Delta),(I,\Omega))$, there is a $g\in \EU_{2n+1}(R,\Delta)$ such that $\sigma:=[h,g]\not\in \U_{2n+1}((R,\Delta),(I,\Omega))$. Since $h\in H\subseteq \CU_{2n+1}((R,\Delta),(I,\Omega^{I}_{max}))$, we have $\sigma\in \U_{2n+1}((R,\Delta),(I,\Omega^{I}_{max}))$. By Lemma \ref{69}, there is an $f\in \TEU_{2n+1}(R,\Delta)$ such that $(f\sigma)_{11}$ is left invertible. Set $\tau:={}^{f}\sigma$. Clearly $\tau_{11}=(f\sigma)_{11}$ and hence $\tau_{11}$ is left invertible. Further $\tau\in \U_{2n+1}((R,\Delta),(I,\Omega^{I}_{max}))\setminus \U_{2n+1}((R,\Delta),(I,\Omega))$ (note that $\U_{2n+1}((R,\Delta),(I,\Omega))$ is E-normal by Corollary \ref{10001}). Let $x$ be an left inverse of $\tau_{11}$ and set $\xi:=T_{21}((1-\tau_{11}-\tau_{21})x)\in \EU_{2n+1}(I,\Omega)\subseteq H$, $\zeta:=T_{12}(1)$ and $\omega:={}^{\zeta}(\xi\tau)$. One checks easily that $\omega_{11}=1$ and $\omega\in \U_{2n+1}((R,\Delta),(I,\Omega^{I}_{max}))\setminus \U_{2n+1}((R,\Delta),(I,\Omega))$. Thus we can apply the previous lemma to $\omega$.
\end{proof}
\section{Sandwich classification of E-normal subgroups}\label{sec5}
In this section $(R,\Delta)$ denotes a Hermitian form ring and $n\geq 3$ a natural number. We will show that if $R$ semilocal or quasifinite, then a subgroup $H$ of $\U_{2n+1}(R,\Delta)$ is E-normal if and only if there is an odd form ideal $(I,\Omega)$ of $(R,\Delta)$ such that 
\[\EU_{2n+1}((R,\Delta),(I,\Omega))\subseteq H \subseteq \CU_{2n+1}((R,\Delta),(I,\Omega)).\]
Further $(I,\Omega)$ is uniquely determined, namely it is the level of $H$. 

If $R$ is semilocal, then we use the Extraction Theorem \ref{extraction} of the last section in order to prove this result. If $R$ is Noetherian and there is a subring $C$ of $\Center(R)$ with certain properties (i.a. $(C\setminus \m)^{-1}R$ is semilocal for any maximal ideal $\m$ of $C$, cf. Subsection \ref{sec5.2}), then we use localisation to deduce the result from the semilocal case. If $R$ is quasifinite, then we use the fact that $\EU_{2n+1}$ and $\U_{2n+1}$ commute with direct limits to deduce the result from the Noetherian case.
\subsection{Semilocal case}\label{sec5.1}
In this subsection we assume that $R$ is semilocal.
\begin{Theorem}\label{74}
Let $H$ be a subgroup of $\U_{2n+1}(R,\Delta)$. Then $H$ is E-normal if and only if there is an odd form ideal $(I,\Omega)$ of $(R,\Delta)$ such that 
\[\EU_{2n+1}((R,\Delta),(I,\Omega))\subseteq H \subseteq \CU_{2n+1}((R,\Delta),(I,\Omega)).\]
Further $(I,\Omega)$ is uniquely determined, namely it is the level of $H$.
\end{Theorem}
\begin{proof}
~\\
$\Rightarrow$:\\
Suppose that $H$ is E-normal. Let $(I,\Omega)$ be the level of $H$. Then $\EU_{2n+1}((R,\Delta),(I,\Omega))\subseteq H$. Suppose that $H\not\subseteq \CU_{2n+1}((R,\Delta),(I,\Omega))$. Then, by Theorem \ref{extraction}, $H$ contains an elementary matrix which is not $(I,\Omega)$-elementary. But this contradicts the assumption that the level of $H$ is $(I,\Omega)$. Hence $H\subseteq \CU_{2n+1}((R,\Delta),(I,\Omega))$. Let $(I',\Omega')$ be an odd form ideal such that 
\[\EU_{2n+1}((R,\Delta),(I',\Omega'))\subseteq H \subseteq \CU_{2n+1}((R,\Delta),(I',\Omega')).\]
Then, by the standard commutator formulas,
\begin{align*}
&\EU_{2n+1}((R,\Delta),(I,\Omega))\\
=&[\EU_{2n+1}((R,\Delta),(I,\Omega)),\EU_{2n+1}(R,\Delta)]\\
\subseteq&[\CU_{2n+1}((R,\Delta),(I',\Omega')),\EU_{2n+1}(R,\Delta)]\\
=&\EU_{2n+1}((R,\Delta),(I',\Omega')).
\end{align*}
It is easy to deduce that $I\subseteq I'$ and  $\Omega\subseteq\Omega'$. After reversing the roles of $I$ and $I'$ and of $\Omega$ and $\Omega'$, we obtain by the same argument that $I'\subseteq I$ and $\Omega'\subseteq \Omega$. Thus $I=I'$ and $\Omega=\Omega'$.\\
$\Leftarrow$:\\
Suppose that $\EU_{2n+1}((R,\Delta),(I,\Omega))\subseteq H \subseteq \CU_{2n+1}((R,\Delta),(I,\Omega))$ for some odd form ideal $(I,\Omega)$. Then
\begin{align*}
&[H, \EU_{2n+1}(R,\Delta)]\\
\subseteq& [\CU_{2n+1}((R,\Delta),(I,\Omega)),\EU_{2n+1}(R,\Delta)]\\
=& \EU_{2n+1}((R,\Delta),(I,\Omega))\\
\subseteq& H
\end{align*}
and hence $H$ is E-normal.
\end{proof}
\subsection{Noetherian case}\label{sec5.2}
In this subsection we assume that $R$ is Noetherian. Further we assume that there is a subring $C$ of $\Center(R)$ such that
\begin{enumerate}[(1)]
\item $\bar c=c$ for any $c\in C$,
\item if $(I,\Omega)$ is an odd form ideal of $(R,\Delta)$, $(0,x)\in \Omega$ and $c\in C$, then $(0,cx)\in\Omega$,
\item $R_\m$ is semilocal for any maximal ideal $\m$ of $C$ and
\item if $(I,\Omega)$ is an odd form ideal of $(R,\Delta)$, then $\Omega^{I}_{max}/\Omega$ is a Noetherian $C$-module.
\end{enumerate} 
In (3), $R_\m$ denotes the ring $S_\m^{-1}R$ where $S_\mathfrak{m}=C\setminus \mathfrak{m}$. If $\m$ is a maximal ideal of $C$, set $\overline{(\frac{x}{s})}:=\frac{\overline x}{s}$, $\lambda_\m:=\frac{\lambda}{1}$, $\mu_\m:=\frac{\mu}{1}$ and $\Delta_\m:=\{(\frac{x}{s},\frac{y}{s^2})\mid (x,y)\in\Delta,s\in S_\m\}$. Then $((R_\m,~\bar{}~,\lambda_\m,\mu_\m),\Delta_\m)$ is a Hermitian form ring. If $(I,\Omega)$ is an odd form ideal of $(R,\Delta)$, set $I_\m:=\{\frac{x}{s}\mid x\in I,s\in S_\mathfrak{m}\}$ and $\Omega_\m:=\{(\frac{x}{s},\frac{y}{s^2})\mid (x,y)\in\Omega,s\in S_\m\}$. Then $(I_\m,\Omega_\m)$ is an odd form ideal of $(R_\m,\Delta_\m)$. Denote the localisation homomorphism $R\rightarrow R_\mathfrak{m}$ by $f_\mathfrak{m}$. Clearly $f_\mathfrak{m}$ induces a group homomorphism $F_\mathfrak{m}:\U_{2n+1}(R,\Delta)\rightarrow \U_{2n+1}(R_\mathfrak{m},\Delta_\mathfrak{m})$. It follows from Lemma \ref{10000} that the restriction of $F_\m$ to $\NU_{2n+1}((R,\Delta),(I,\Omega))$ (which we also denote by $F_\m$) is a group homomorphism $\NU_{2n+1}((R,\Delta),(I,\Omega))\rightarrow \NU_{2n+1}((R_\m,\Delta_\m),(I_\m,\Omega_\m))$. Let 
\begin{align*}
\phi_\m:&\NU_{2n+1}((R,\Delta),(I,\Omega))/\U_{2n+1}((R,\Delta),(I,\Omega))\\
\rightarrow~& \NU_{2n+1}((R_\m,\Delta_\m),(I_\m,\Omega_\m))/\U_{2n+1}((R_\m,\Delta_\m),(I_\m,\Omega_\m))
\end{align*}
be the group homomorphism induced by $F_\m$. Clearly we have a commutative diagram
\xymatrixcolsep{4pc}
\xymatrixrowsep{4pc}
\[\xymatrix{
\NU_{2n+1}((R,\Delta),(I,\Omega)) \ar[d]^{F_\m} \ar[r]^-{\pi} &\NU_{2n+1}((R,\Delta),(I,\Omega))/\U_{2n+1}((R,\Delta),(I,\Omega))\ar[d]^{\phi_\m}\\
\NU_{2n+1}((R_\m,\Delta_\m),(I_\m,\Omega_\m))\ar[r]^-{\pi_\m} &\NU_{2n+1}((R_\m,\Delta_\m),(I_\m,\Omega_\m))/\U_{2n+1}((R_\m,\Delta_\m),(I_\m,\Omega_\m)),}\]
where $\pi$ and $\pi_\m$ denote the canonical group homomorphisms.
\begin{Lemma}\label{75}
Let $(I,\Omega)$ be an odd form ideal of $(R,\Delta)$ and $h\in \U_{2n+1}(R,\Delta)\setminus \CU_{2n+1}((R,\Delta),(I,\Omega))$. Then there is a maximal ideal $\m$ of $C$ such that $F_\m(h)\in \U_{2n+1}(R_\m,\Delta_\m)\setminus \CU_{2n+1}((R_\m,\Delta_\m),(I_\m,\Omega_\m))$.
\end{Lemma}
\begin{proof}
Since $h\in \U_{2n+1}(R,\Delta)\setminus \CU_{2n+1}((R,\Delta),(I,\Omega))$, there is a $g\in \EU_{2n+1}(R,\Delta)$ such that $\sigma:=[h,g]\not\in \U_{2n+1}((R,\Delta),(I,\Omega))$. We will show that there is a maximal ideal $\m$ of $C$ such that $F_\m(\sigma)=[F_\m(h),F_\m(g)]\not\in \U_{2n+1}((R_\m,\Delta_\m),(I_\m,\Omega_\m))$ which implies that $F_\m(h)\not\in \CU_{2n+1}((R_\m,\Delta_\m),(I_\m,\Omega_\m))$.\\
\\
\underline{case 1} Assume that $\sigma\not\equiv e\bmod I,\tilde I,I_0,\tilde I_0 $.\\
Then either $\sigma_{hb}\not\equiv e_{hb}\bmod I $ or $\sigma_{0*}\not\equiv e_0^t\bmod  \tilde I,\tilde I_0$ (compare Remark \ref{40}).\\
\\
\underline{case 1.1} Assume that $\sigma_{hb}\not\equiv e_{hb}\bmod I $.\\
Then there are $i,j\in \Theta_{hb}$ such that $\sigma_{ij}\not\equiv\delta_{ij}\bmod I$. Set $Y:=\{c\in C\mid c(\sigma_{ij}-\delta_{ij})\in I\}$. Since $\sigma_{ij}-\delta_{ij}\not\in I$, $Y$ is a proper ideal of $C$. Hence it is contained in a maximal ideal $\m$ of $C$. Clearly $Y\subseteq \m$ and hence $S_\m\cap Y=\emptyset$. Assume $(F_\m(\sigma))_{ij}-\delta_{ij}\in I_{\m}$. One checks easily that then there is an $s\in S_\m$ such that $s\in Y$. But this contradicts $S_\m\cap Y=\emptyset$. Hence $(F_\m(\sigma))_{ij}-\delta_{ij}\not\in I_\m$ and therefore $F_\m(\sigma)\not\in \U_{2n+1}((R_{\m},\Delta_{\m}),(I_{\m},\Omega_{\m}))$.\\
\\
\underline{case 1.2} Assume that $\sigma_{0*}\not\equiv e_0^t\bmod  \tilde I,\tilde I_0$.\\
From the assumption it follows that $(F_\m(\sigma))_{0*}\not\equiv e_0^t\bmod  \widetilde{I_\m},(\widetilde{I_\m})_0$ (see case 1.1). Hence $F_\m(\sigma)\not\in \U_{2n+1}((R_{\m},$ $\Delta_{\m}),(I_{\m},\Omega_{\m}))$.\\
\\
\underline{case 2} Assume that $\sigma\equiv e\bmod I,\tilde I,I_0,\tilde I_0 $.\\ 
It follows that either $\exists j\in\Theta_{hb}:q(\sigma_{*j})\not\in\Omega$ or $\exists x\in J(\Delta):(q(\sigma_{*0})\minus(1,0))\circ x\not\in\Omega$ since $\sigma\not\in \U_{2n+1}((R,\Delta),(I,\Omega))$.\\
\\
\underline{case 2.1} Assume that $\exists j\in\Theta_{hb}:q(\sigma_{*j})\not\in\Omega$.\\
Set $Y:=\{c\in C\mid q(\sigma_{*j})\circ c\in \Omega\}$. Since $q(\sigma_{*j})\not\in \Omega$, $Y$ is a proper ideal of $C$. Hence it is contained in a maximal ideal $\m$ of $C$. Clearly $Y\subseteq \m$ and hence $S_\m\cap Y=\emptyset$. Assume $q((F_\m(\sigma))_{*j})\in\Omega_\m$. One checks easily that then there is an $s\in S_\m$ such that $s\in Y$.
But this contradicts $S_\m\cap Y=\emptyset$. Hence $q((F_\m(\sigma))_{*j})\not\in \Omega_{m}$ and therefore $F_\m(\sigma)\not\in \U_{2n+1}((R_{\m},\Delta_{\m}),(I_{\m},\Omega_{\m}))$. Thus $\phi_\m(h)$ does not commute with $\phi_\m(g)$.\\
\\
\underline{case 2.2} Assume that $\exists x\in J(\Delta):(q(\sigma_{*0})\minus(1,0))\circ x\not\in\Omega$.\\
From the assumption it follows that $(q((F_\m(\sigma))_{*0})\minus(1,0))\circ f_\m(x)\not\in\Omega_\m$ (see case 2.1). Hence $F_\m(\sigma)\not\in \U_{2n+1}((R_{\m},\Delta_{\m}),(I_{\m},\Omega_{\m}))$.
\end{proof}

In the next two lemmas we use the following notation. If $\Omega$ is a relative odd form parameter for an involution invariant ideal $I$ and $c\in C$, set $c\Omega:=\Omega\circ c\+ (0,c\Gamma)$ where $\Gamma=\Gamma(\Omega)$. One checks easily that $(cI,c\Omega)$ is an odd form ideal of $(R,\Delta)$.
\Lemma\label{56}{Let $(I,\Omega)$ be an odd form ideal of $(R,\Delta)$ and $\m$ a maximal ideal of $C$. Then there is an $s_0\in S_\m$ such that $\phi_\m$ is injective on $\pi(\U_{2n+1}((R,\Delta),(s_0R,s_0\Delta))\cap \NU_{2n+1}((R,\Delta),(I,\Omega)))$.
}
\begin{proof}
First we show that there is an $s_1\in S_\m$ with the properties 
\begin{enumerate}[(1)]
\item if $x \in s_1R$ and $\exists t\in S_\m: tx\in I$, then $x\in I$ and
\item if $a \in s_1\Omega_{max}^{I}$ and $\exists t\in S_\m: a\circ t\in \Omega$, then $a\in \Omega$.
\end{enumerate}
Then we will show that $\phi_\m$ is injective on $K':=\pi(K)$ where 
\[K:=\U_{2n+1}((R,\Delta),(s_1^3R,s_1^3\Delta))\cap \NU_{2n+1}((R,\Delta),(I,\Omega)).\]
For any $s\in S_\m$ set $Y(s):=\{x\in R/I\mid x\hat s=0\}$ and $Z(s):=\{a\in \Omega_{max}^I/\Omega\mid as=0\}$ where $\hat s$ is the image of $s$ in $R/I$. Then for any $s\in S_\m$, $Y(s)$ is an ideal of $R/I$ and $Z(s)$ an $R$-submodule of $\Omega^{I}_{max}/\Omega$. Since $R$ is Noetherian, $R/I$ is Noetherian and hence the set $A:=\{Y(s)\mid s\in S_\m\}$ has a maximal element $Y(s_3)$. Since $A$ is directed (namely $Y(s),Y(s')\subseteq Y(ss')$), $Y(s_3)$ is the greatest element of $A$. Clearly all elements $x\in s_3R$ have the property that $tx\in I$ for some $t\in S_\m$ implies $x\in I$. Since $\Omega^{I}_{max}/\Omega$ is a Noetherian $R$-module, the set $B:=\{Z(s)\mid s\in S_\m\}$ has a maximal element $Z(s_2)$. Since $B$ is directed (namely $Z(s),Z(s')\subseteq Z(ss')$), $Z(s_2)$ is the greatest element of $B$. Clearly all elements $a\in s_2^2\Omega_{max}^{I}$ have the property that $a\circ t\in\Omega$ for some $t\in S_\m$ implies $a\in \Omega$. Set $s_1:=s_3s_2^2$. Then clearly $s_1$ has the properties (1) and (2) above. We show next that $\phi_\m$ is injective on $K'$.\\
Let $g'_1,g'_2\in K'$ such that $\phi_\m(g'_1)=\phi_\m(g_2')$. Since $g'_1,g'_2$, there are $g_1,g_2\in K$ such that $\pi(g_1)=g'_1$ and $\pi(g_2)=g'_2$. Set $h:=g_1^{-1}g_2\in K$. Clearly $\phi_\m(g'_1)=\phi_\m(g_2')$ is equivalent to $h_\m:=F_\m(h)\in \U_{2n+1}((R_\m,\Delta_\m),(I_\m,\Omega_\m))$, i.e.
\begin{enumerate}[(a)]
\item 
\begin{align*}
(h_\m)_{hb}\equiv e_{hb}  (mod~I_\m) \text{ and }
\end{align*}
\item 
\begin{align*}
q((h_\m)_{*j})\in\Omega_\m~\forall j\in\Theta_{hb}\text{ and }(q((h_\m)_{*0})\ominus(1,0))\circ x\in\Omega_\m~\forall x\in J(\Delta_\m).
\end{align*}
\end{enumerate}
We want to show that $g'_1=g'_2$ which is equivalent to $h\in \U_{2n+1}((R,\Delta),(I,$ $\Omega))$, i.e.
\begin{enumerate}[(a')]
\item 
\begin{align*}
h_{hb}\equiv e_{hb}  (mod~I) \text{ and }
\end{align*}
\item 
\begin{align*}
q(h_{*j})\in\Omega~\forall j\in\Theta_{hb}\text{ and }(q(h_{*0})\ominus(1,0))\circ x\in \Omega~\forall x\in J(\Delta).
\end{align*}
\end{enumerate}
First we show that (a'). Let $i,j\in\Theta_{hb}$ such that $i\neq j$. Since (a) holds, $(h_\m)_{ij}\in I_\m$. It is easy to deduce that $th_{ij}\in I$ for some $t\in S_\m$. Since $h\in \U_{2n+1}((R,\Delta),(s_1^3R,s_1^3\Delta))$, we have $h_{ij}\in s_1^3R\subseteq s_1R$. Since $s_1$ has property (1), it follows that $h_{ij}\in I$. Analogously one can show that $h_{ii}-1\in I~\forall i\in\Theta_{hb}$ and hence (a') holds. Further one can show that $h_{0*}\equiv e_0^t(mod~\tilde I,\tilde I_0)$ which implies $h\equiv e(mod~I,\tilde I,I_0,\tilde I_0)$. Next we show (b').\\
Let $j\in\Theta_{hb}$. Since (b) holds, $q((h_m)_{*j})\in\Omega_\m$. It is easy to deduce that $q(h_{*j})\circ t\in\Omega$ for some $t\in S_\m$. Since $h\in \U_{2n+1}((R,\Delta),(s_1^3R,s_1^3\Delta))$, we have $q(h_{*j})\in s_1^3\Delta$. On the other hand $q(h_{*j})\in \tilde I\times I$ since $h\equiv e  (mod~I,\tilde I,I_0,\tilde I_0)$. It is easy to deduce that $q(h_{*j})\in s_1\Omega_{max}^I$. Since $s_1$ has property (2), it follows that $q(h_{*j})\in \Omega$. Analogously one can show that $(q(h_{*0})\minus(1,0))\circ x\in\Omega~\forall x\in J(\Delta)$. Hence (b') holds. It follows that $g'_1=g'_2$ and thus $\phi_\m$ is injective on $K'$.
\end{proof}

\textnormal{In the next lemma we use the following notation. Suppose $\m$ is a maximal ideal of $C$, $s\in S_\m$ and $t\in C$. We denote the subset $\{\frac{x}{s}\mid x\in tR\}$ of $(tR)_\m$ by $(1/s)tR$. Further, for any $\epsilon\in \{\pm 1\}$ we denote the subset $\{(\frac{x}{s^2},\frac{y}{s})\mid  (x,y)\in (t\Delta)^\epsilon\}$ of $(t\Delta_\m)^\epsilon$ by $(1/s)t\Delta^\epsilon$ (instead of $(1/s)t\Delta^1$ we sometimes write $(1/s)t\Delta$). For any $N\in\mathbb{N}$ we denote by $\E^N((1/s)tR,(1/s)t\Delta)$ the subset of $\EU_{2n+1}((tR)_\m,(t\Delta)_\m)$ consisting of all products of $N$ elementary matrices of the form $T_{ij}(x)$ where $x\in (1/s)tR$ or $T_i(a)$ where $a\in (1/s)t\Delta^{-\epsilon(i)}$. Instead of $\E^N((1/1)tR,(1/1)t\Delta)$ we sometimes write $\E^N(tR,t\Delta)$. If $k\in\mathbb{N}$ and $\M_1,\dots,\M_k$ are subsets of $\U_{2n+1}(R_\m,\Delta_\m)$, then $^{\M_k}(\dots^{\M_2}(^{\M_1}\E^N((1/s)tR,(1/s)t\Delta))\dots)$ denotes the set of all product of $N$ matrices of the form $^{\sigma_1}(\dots^{\sigma_{k-2}}(^{\sigma_{k}}\mu )\dots)$ where $\sigma_i\in \M_i~\forall i\in \{1,\dots,k\}$ and $\mu \in \E^1((1/s)tR,(1/s)t\Delta)$.}
\begin{Lemma}\label{57}
Let $\m$ be a maximal ideal of $C$, $s\in S_\m$ and $t\in C$. Given $K,L,m\in\mathbb{N}$ there are $k,M\in\mathbb{N}$, e.g. $k=(m+2)4^K+2\cdot 4^{K-1}+\dots+2\cdot 4$ and $M=L\cdot 22^{K}$, such that 
\[^{\E^K((1/s)tR,(1/s)t\Delta)}\E^L(s^{k}tR,s^kt\Delta)\subseteq \E^M(s^{m}tR,s^mt\Delta).\]
\end{Lemma}
\begin{proof}
In order to prove the lemma it clearly suffices to prove the case that $L=1$. We will do this by induction on $K$.\\
\\
\underline{$K=1$}\\
Let $m\in\mathbb{N}$. We have to show that
\[^{\E^1((1/s)tR,(1/s)t\Delta)}\E^1(s^{(m+2)4}tR,s^{(m+2)4}t\Delta)\subseteq \E^{22}(s^{m}tR,s^mt\Delta).\]
Let $\sigma\in \E^1((1/s)tR,(1/s)t\Delta)$ and $\tau\in \E^1(s^{(m+2)4}tR,s^{(m+2)4}t\Delta)$. We have to show that $\rho:=$ $^{\sigma}\tau\in \E^{22}(s^{m}tR,s^mt\Delta)$.\\
\\
\underline{Part I}\\
Assume that $\sigma$ and $\tau$ are short root matrices. Then $\rho\in \E^{14}(s^{m}tR,s^mt\Delta)$ by \cite[proof of Lemma 4.1, Case I]{hazrat}.\\
\\
\underline{Part II}\\
Assume that $\sigma$ is an extra short root and $\tau$ a short root matrix. Hence there are $h,i,j\in\Theta_{hb}$ where $i\neq\pm j$, $(y,z)\in (1/s)(t\Delta)^{-\epsilon(h)}$ and $x\in (1/1)s^{(m+2)4}tR$ such that $\sigma=T_h(y,z)$ and $\tau=T_{ij}(x)$.\\ 
\underline{case 1} Assume that $h\neq j,-i$.\\
Then $\sigma$ and $\tau$ commute by (SE1) in Lemma \ref{23}. Hence $\rho=\tau\in \E^{1}(s^{m}tR,s^mt\Delta)$.\\
\underline{case 2} Assume that $h=j$.\\
Then 
\begin{align*}
\rho&=T_j(y,z)T_{ij}(x)T_j(\minus_{-\epsilon(j)}(y,z))\\
&=T_{ij}(x)\underbrace{T_{ij}(-x)T_j(y,z)T_{ij}(x)T_j(\minus_{-\epsilon(j)}(y,z))}_{(SE2)}\\
&=T_{ij}(x)T_{j,-i}(-z\lambda_\m^{(\epsilon(j)-1)/2}\bar x\lambda_\m^{(1-\epsilon(i))/2})\\
&~\cdot T_i(-y\lambda_\m^{(\epsilon(j)-1)/2}\bar x\lambda_\m^{(1-\epsilon(i))/2},xz\lambda_\m^{(\epsilon(j)-1)/2}\bar x\lambda_\m^{(1-\epsilon(i))/2}))\\
&\in \E^{3}(s^{m}tR,s^mt\Delta).
\end{align*}
\\
\underline{case 3} Assume that $h=-i$.\\
This case can be reduced to case 2 using (S1).\\
\\
\underline{Part III}\\
Assume that $\sigma$ and $\tau$ are extra short root matrices. Hence there are $h,i\in\Theta_{hb}$, $(x_1,y_1)\in (1/s)(t\Delta)^{-\epsilon(h)}$ and $(x_2,y_2)\in (1/1)(s^{(m+2)4}t\Delta)^{-\epsilon(i)}$ such that $\sigma=T_h(x_1,y_1)$ and $\tau=T_{i}(x_2,y_2)$.\\
\underline{case 1} Assume that $h\neq \pm i$.\\
Then 
\begin{align*}
\rho&=T_h(x_1,y_1)T_{i}(x_2,y_2)T_h(\minus_{-\epsilon(h)}(x_1,y_1))\\
&=\underbrace{T_h(x_1,y_1)T_{i}(x_2,y_2)T_h(\minus_{-\epsilon(h)}(x_1,y_1))T_i(\minus_{-\epsilon(i)}(x_2,y_2))}_{(E2)}T_{i}(x_2,y_2)\\
&=T_{h,-i}(-\lambda_\m^{-(1+\epsilon(h))/2}\bar x_1\mu_\m x_2)T_{i}(x_2,y_2)\\
&\in \E^{2}(s^mtR,s^mt\Delta).
\end{align*}
\\
\underline{case 2} Assume that $h=i$.\\
Then
\begin{align*}
\rho&~=\underbrace{T_i(x_1,y_1)T_{i}(x_2,y_2)T_i(\minus_{-\epsilon(i)}(x_1,y_1))}_{(E1)}\\
&~=T_{i}(x_2,y_2-\lambda_\m^{-(\epsilon(i)+1)/2}(\bar x_1\mu_\m x_2-\bar x_2\mu_\m x_1))\\
&~\in \E^{1}(s^mtR,s^mt\Delta).
\end{align*}
\\
\underline{case 3} Assume that $h=-i$.\\
One checks easily that there is an $(a,b)\in (1/1)(s^{(m+2)2}t\Delta)^{-\epsilon(i)}$ such that $(x_2,y_2)=(a,b)\circ c $ where $c:=\frac{s^{m+2}}{1}$. Choose a $p\in\Theta_{hb}$ such that $p\neq\pm h$ and $\epsilon(p)=\epsilon(-h)$. By (SE2),
\[\tau=T_{-h}(x_2,y_2)=T_{ph}(-cb)[T_{-h,p}(c),T_p(a,b)].\]
Hence
\begin{align*}
\rho=~&^{T_h(x_1,y_1)}T_{-h}(x_2,y_2)\\
=~&^{T_h(x_1,y_1)}(T_{ph}(-cb)[T_{-h,p}(c),T_p(a,b)])\\
=~&\underbrace{^{T_h(x_1,y_1)}T_{ph}(-cb)}_{\text{Part II, case 2}}[\underbrace{^{T_h(x_1,y_1)}T_{-h,p}(c)}_{\text{Part II, case 3}},\underbrace{^{T_h(x_1,y_1)}T_p(a,b)}_{\text{Part III, case 1}}]\\
=~&\underbrace{T_{ph}(-cb)T_{h,-p}(y_1\lambda_\m^{(\epsilon(h)-1)/2}c\bar b\lambda_\m^{(1-\epsilon(h))/2})}_{T_1}\times\\
&\times \underbrace{T_p(x_1\lambda_\m^{(\epsilon(h)-1)/2}c\bar b\lambda_\m^{(1-\epsilon(h))/2},cby_1\lambda_\m^{(\epsilon(h)-1)/2}c\bar b\lambda_\m^{(1-\epsilon(h))/2})}_{T_2}\times\\
&\times[\underbrace{T_{-p,h}(-c)T_{hp}(cy_1)T_{-p}(cx_1,c^2y_1)}_{(SE1)},T_{h,-p}(-\lambda_\m^{-(1+\epsilon(h))/2}\bar x_1\mu_\m a)T_p(a,b)]\\
=~&T_1T_2[T_{-p}(cx_1,c^2y_1)T_{-p,h}(-c)T_{hp}(cy_1),T_{h,-p}(-\lambda_\m^{-(1+\epsilon(h))/2}\bar x_1\mu_\m a)T_p(a,b)]\\
=~&T_1T_2\underbrace{T_{-p}(cx_1,c^2y_1)}_{T_3}T_{-p,h}(-c)T_{hp}(cy_1)T_{h,-p}(-\lambda_\m^{-(1+\epsilon(h))/2}\bar x_1\mu_\m a)T_p(a,b)\times\\
&\times T_{hp}(-cy_1)T_{-p,h}(c)\underbrace{T_{-p}(\minus_{\epsilon(p)}(cx_1,c^2y_1))T_p(\minus_{-\epsilon(p)}(a,b))}_{T_8}\times\\
&\underbrace{T_{h,-p}(\lambda_\m^{-(1+\epsilon(h))/2}\bar x_1\mu_\m a)}_{T_9}\\
=~&T_1T_2T_3\underbrace{^{T_{-p,h}(-c)}T_{hp}(cy_1)}_{T_4}\underbrace{^{T_{-p,h}(-c)}T_{h,-p}(-\lambda_\m^{-(1+\epsilon(h))/2}\bar x_1\mu_\m a)}_{T_5}\times\\&\times\underbrace{^{T_{-p,h}(-c)}T_p(a,b)}_{T_6}\underbrace{^{T_{-p,h}(-c)}T_{hp}(-cy_1)}_{T_7}T_8T_9.
\end{align*}
(S5) implies that \[T_4=T_{-p}(0,c^2(-y_1+\lambda_\m^{(-1+\epsilon(p))/2}\bar y_1\lambda_\m^{(1+\epsilon(p))/2}))T_{hp}(cy_1)\in \E^{2}(s^mtR,s^mt\Delta)\] and hence \[T_7=T_4^{-1}=T_{hp}(-cy_1)T_{-p}(0,-c^2(-y_1+\lambda_\m^{(-1+\epsilon(p))/2}\bar y_1\lambda_\m^{(1+\epsilon(p))/2}))\in \E^{2}(s^mtR,s^mt\Delta).\]
Obviously $-\lambda_\m^{-(1+\epsilon(h))/2}\bar x_1\mu_\m a\in (1/1)s^{2m}t^2R$. Hence there is an $r\in R$ such that $-\lambda_\m^{-(1+\epsilon(h))/2}\bar x_1\mu_\m a=f_\m(s^{2m}t^2r)$. Set $d_1:=f_\m(s^mt)$, $d_2:=f_\m(s^mtr)$ and choose a $k\in\Theta_{hb}$ such that $k\neq\pm h,\pm p$. By (S4),
\begin{align*}
T_5=~&^{T_{-p,h}(-c)}T_{h,-p}(-\lambda_\m^{-(1+\epsilon(h))/2}\bar x_1\mu_\m a)\\
=~&^{T_{-p,h}(-c)}[T_{hk}(d_1),T_{k,-p}(d_2)]\\
=~&[^{T_{-p,h}(-c)}T_{hk}(d_1),^{T_{-p,h}(-c)}T_{k,-p}(d_2)]\\
=~&[T_{-p,k}(-cd_1)T_{hk}(d_1),T_{kh}(cd_2)T_{k,-p}(d_2)]\\
\in~&\E^{8}(s^mtR,s^mt\Delta).
\end{align*}
Further 
\begin{align*}
T_6=~&^{T_{-p,h}(-c)}T_p(a,b)\\
=~&[\underbrace{T_{-p,h}(-c)}_{(S1)},T_p(a,b)]T_p(a,b)\\
=~&\underbrace{[T_{-h,p}(c),T_p(a,b)]}_{(SE2)}T_p(a,b)\\
=~&T_{ph}(cb)T_{-h}(ca,c^2b)T_p(a,b)\\
\in~&\E^{3}(s^mtR,s^mt\Delta).
\end{align*}
Thus \[\rho=\underbrace{T_1}_{2}\underbrace{T_2}_{1}\underbrace{T_3}_{1}\underbrace{T_4}_{2}\underbrace{T_5}_{8}\underbrace{T_6}_{3}\underbrace{T_7}_{2}\underbrace{T_8}_{2}\underbrace{T_9}_{1}\in \E^{22}(s^mtR,s^mt\Delta).\]\\
\\
\underline{Part IV}\\
Assume that $\sigma$ is a short root and $\tau$ an extra short root matrix. All the possibilities which may occur here reduce to one of the cases above.\\
\\
Thus $\rho\in \E^{22}(s^mtR,s^mt\Delta)$.\\
\\
\underline{$K\rightarrow K+1$}\\
Let $m\in\mathbb{N}$. We have to show that
\[^{\E^{K+1}((1/s)tR,(1/s)t\Delta)}\E^1(s^{k}tR,s^{k}t\Delta)\subseteq \E^{M}(s^{m}tR,s^mt\Delta).\]
where $k=(m+2)4^{K+1}+2\cdot 4^{K}+\dots+2\cdot 4$ and $M=22^{K+1}$. Set $m':=(m+2)4$ and $M':=22^K$. Clearly
\begin{align*}
&^{\E^{K+1}((1/s)tR,(1/s)t\Delta)}\E^1(s^{k}tR,s^{k}t\Delta)\\
\subseteq&^{\E^{1}((1/s)tR,(1/s)t\Delta)}(^{\E^{K}((1/s)tR,(1/s)t\Delta)}\E^1(s^{k}tR,s^{k}t\Delta))\\
\overset{\text{I.A.}}{\subseteq}&^{\E^{1}((1/s)tR,(1/s)t\Delta)}\E^{M'}(s^{m'}tR,s^{m'}t\Delta)\\
\overset{K=1}{\subseteq}&\E^{M}(s^{m}tR,s^mt\Delta).
\end{align*}
\end{proof}
\begin{Definition}\label{76}
Let $G$ denote a group and $A$ a set of subgroups of $G$ such that
\begin{enumerate}[(1)]
\item for any $U,V\in A$ there is a $W\in A$ such that $W\subseteq U\cap V$ and
\item for any $g\in G$ and $U\in A$ there is a $V\in A$ such that $^gV\subseteq U$.
\end{enumerate}
Then $A$ is called a {\it base of open subgroups of $1\in G$} or just a {\it base for $G$}.
\end{Definition}
\begin{Lemma}\label{77}
Let $(I,\Omega)$ be an odd form ideal of $(R,\Delta)$, $\m$ a maximal ideal of $C$ and $s_0\in S_\m$. Set $A:=\{\EU_{2n+1}(ss_0R,ss_0\Delta)\mid s\in S_\m\}$. Then $A$ is a base for  $\EU_{2n+1}(R,\Delta)$ and $A_\m:=\{F_\m(U)\mid U\in A\}$ is a base for $\EU_{2n+1}(R_\m,\Delta_\m)$.
\end{Lemma}
\begin{proof}
First we show that $A$ is a base for $\EU_{2n+1}(R,\Delta)$.
\begin{enumerate}[(1)]
\item Let $U=\EU_{2n+1}(ss_0R,ss_0\Delta),V=\EU_{2n+1}(ts_0R,ts_0\Delta)\in A$. Set $W:=\EU_{2n+1}(sts_0R,sts_0\Delta)\in A$. Then clearly $W\subseteq U\cap V$.
\item Let $g\in \EU_{2n+1}(R,\Delta)$ and $U=\EU_{2n+1}(ss_0R,ss_0\Delta)\in A$. There is a $K\in\mathbb{N}$ such that $g$ is the product of $K$ elementary matrices. Set $V:=\EU_{2n+1}((ss_0)^{3\cdot 4^K+}$ $^{2\cdot 4^{K-1}+\dots+2\cdot 4}R,(ss_0)^{3\cdot 4^K+2\cdot4^{K-1}+\dots+2\cdot4}\Delta)\in A$. Then $^gV\subseteq U$ by Lemma \ref{57}.
\end{enumerate}
Hence $A$ is a base for $\EU_{2n+1}(R,\Delta)$. We show now that $A_\m$ is a base for $\EU_{2n+1}(R_\m,\Delta_\m)$.
\begin{enumerate}[(1)]
\item Let $U=F_\m(\EU_{2n+1}(ss_0R,ss_0\Delta)),V=F_\m(\EU_{2n+1}(ts_0R,ts_0\Delta))\in A_\m$. Set $W:=F_\m(\EU_{2n+1}(sts_0R,$ $sts_0\Delta))\in A_\m$. Then clearly $W\subseteq U\cap V$.
\item Let $g\in \EU_{2n+1}(R_\m,\Delta_\m)$ and $U=F_\m(\EU_{2n+1}(ss_0R,ss_0\Delta))\in A_\m$. Clearly there are a $K\in\mathbb{N}$ and a $t\in S_\m$ such that $g\in \E^K((1/t)R,(1/t)\Delta)$. Set $V:=F_\m(\EU_{2n+1}((tss_0)^{3\cdot 4^K+2\cdot 4^{K-1}+\dots+2\cdot 4}R,$ $(tss_0)^{3\cdot 4^K+2\cdot 4^{K-1}+\dots+2\cdot 4}\Delta))\in A_\m$. Then $^gV$ $\subseteq U$ by Lemma \ref{57}.
\end{enumerate}
Hence $A_\m$ is a base for $\EU_{2n+1}(R_\m,\Delta_\m)$.
\end{proof}
\begin{Lemma}\label{78}
Let $(I,\Omega)$ be an odd form ideal of $(R,\Delta)$, $\m$ a maximal ideal of $C$ and $h\in \NU_{2n+1}(R_\m,\Delta_\m)\setminus \CU_{2n+1}((R_\m,\Delta_\m),(I_\m,\Omega_\m))$. Then given any $U\in A_\m$, there is a $k\in\mathbb{N}$ and elements $\xi\in \EU_{2n+1}((R_\m,\Delta_\m),$ $(I_\m,\Omega_\m))$, $g_0,\dots,g_k\in F_\m(\EU_{2n+1}(R,\Delta))$, $\epsilon_0,\dots,\epsilon_k\in \EU_{2n+1}(R_\m,\Delta_\m)$ and $l_1,\dots,l_k\in \{\pm 1\}$ such that $g_k$ is an elementary matrix in $\EU_{2n+1}(R_\m,\Delta_\m)$ which is not $(I_\m,\Omega_\m)$-elementary,
\[ ^{\epsilon_{k}}([^{\epsilon_{k-1}}(\dots^{\epsilon_2}([^{\epsilon_1}([^{\epsilon_0}h,g_0]^{l_1}),g_1]^{l_2})\dots),g_{k-1}]^{l_{k}})\xi=g_k\]
and
\[^{d_i}g_i\in U~\forall i\in \{0,\dots,k\}\]
where $d_i=(\epsilon_{i}\hdots\epsilon_0)^{-1}\hspace{0.1cm}~(0\leq i\leq k)$. 
\end{Lemma}
\begin{proof}
Section \ref{sec6} shows that there is a $k\in\mathbb{N}$ and elements $\xi\in EU_{2n+1}((R_\m,\Delta_\m),(I_\m,\Omega_\m))$, $g_0,\dots,g_k,\epsilon_0,\dots,$ $\epsilon_k\in \EU_{2n+1}(R_\m,\Delta_\m)$ and $l_1,\dots,l_k\in \{\pm 1\}$ such that $g_k$ is an elementary matrix in $\EU_{2n+1}(R_\m,\Delta_\m)$ which is not $(I_\m,\Omega_\m)$-elementary and
\[ ^{\epsilon_{k}}([^{\epsilon_{k-1}}(\dots^{\epsilon_2}([^{\epsilon_1}([^{\epsilon_0}h,g_0]^{l_1}),g_1]^{l_2})\dots),g_{k-1}]^{l_{k}})\xi=g_k.\]
Now replace each $g_i$ by an appropriate element of $U_i$ where $U_i\in A_\m$ is chosen such that $^{d_i}U_i\subseteq U$ (possible since $A_\m$ is a base for $\EU_{2n+1}(R_\m,\Delta_\m)$ by Lemma \ref{77}). 
\end{proof}
\begin{Lemma}\label{79}
Let $(I,\Omega)$ be an odd form ideal of $(R,\Delta)$, $m$ a maximal ideal of $C$ and $s_0\in S_\m$. Set 
\begin{align*}
B:=&\{\EU_{2n+1}(I(xs_0),\Omega(xs_0))\mid x\in R, \forall s\in S_\m:xs\not\in I\}\\
&\cup\{\EU_{2n+1}(I(a\circ s_0),\Omega(a\circ s_0))\mid a\in \Delta, \forall s\in S_\m:a\circ s\not\in \Omega\}
\end{align*}
where $(I(xs_0),\Omega(xs_0))$ (resp. $(I(a\circ s_0),\Omega(a\circ s_0))$ denotes the odd form ideal defined by $xs_0$ (resp. $a\circ s_0$), see Definition \ref{13}. Further set $B_\m:=F_\m(B)$ and let $A_\m$ be defined as in Lemma \ref{77}. Then the following is true.
\begin{enumerate}[(1)]
\item[\textnormal{(1)}] If $U \in A_\m$ and $g\in \EU_{2n+1}(R_\m,\Delta_\m)$ is an elementary matrix which is not $(I_\m,\Omega_\m)$-elementary, then 
\[V\subseteq{}^{U}g\]
for some $V\in B_\m$.
\item[\textnormal{(2)}] If $V\in B_\m$ and $d\in \EU_{2n+1}(R_\m,\Delta_\m)$ is an elementary matrix, then 
\[g\in{}^{d}V\]
for some elementary matrix $g\in \EU_{2n+1}(R_\m,\Delta_\m)$ which is not $(I_\m,\Omega_\m)$-elementary.
\end{enumerate}
\end{Lemma}
\begin{proof}
Follows from the relations in Lemma \ref{23}.
\end{proof}
\begin{Corollary}\label{80}
Let $(I,\Omega)$ be an odd form ideal of $(R,\Delta)$, $m$ a maximal ideal of $C$ and $s_0\in S_\m$. If $U \in A_\m$, $d\in \EU_{2n+1}(R_\m,\Delta_\m)$ and $g\in \EU_{2n+1}(R_\m,\Delta_\m)$ is an elementary matrix which is not $(I_\m,\Omega_\m)$-elementary, then 
\[V\subseteq{}^{U}(^{d}g)\]
for some $V\in B_\m$.
\end{Corollary}
\begin{proof}
If $d = e$, then we are done, by  Lemma \ref{79}(1). Assume $d\neq e$ and write $d$ as a product $d_k\dots d_1$ of nontrivial elementary matrices in $\EU_{2n+1}(R_\m,\Delta_\m)$. We proceed by induction on $k$.\\
\underline{case 1} Assume that $k = 1$. Since $A_\m$ is a base for $\EU_{2n+1}(R_\m,\Delta_\m)$, there is a $U_1 \in A_\m$ such that $^{d_1}\U_1\subseteq U$. Clearly
\begin{align*}
&^{U}(^{d_1}g)\\
\supseteq~~&^{U}(^{d_1} (^{\U_1}g))\\
\overset{L.\ref{79}(1)}{\supseteq}& ^{U}(^{d_1} V) \text{ (for some } V \in B_\m
 \text{)}\\
\overset{L.\ref{79}(2)}{\supseteq}&^{U}g'  \text{ (for some elementary matrix }g'\in \EU_{2n+1}(R_\m,\Delta_\m)\text{ which is not }(I_\m,\Omega_\m)\text{-elementary)} \\
\overset{L.\ref{79}(1)}{\supseteq}&V_1 \text{ (for some } V_1 \in B_\m \text{).}
\end{align*}
\underline{case 2} Assume that $k > 1$. Set $h := d_{k-1}\dots d_1$. Thus $d = d_k\dots d_1 = d_kh$.  We can assume by induction 
on $k$ that given $U_1 \in A_\m$, $^{U_1}(^{h}g)\supseteq V$ for some $V \in B_\m$. Now we proceed similarly to case 1, replacing $g$ by $^{h}g$ and $d_1$ by $d_k$. Here are the details. Choose $U_1\in A_\m$ such that $^{d_k}\phi(U_1)\subseteq U$. Clearly
\begin{align*}
&^{U}(^{d_kh}g)\\
\supseteq~~&^{U}(^{d_k} (^{\U_1}(^{h}g)))\\
\overset{I.A.}{\supseteq}~&^{U}(^{d_k} V)\\
\overset{L.\ref{79}(2)}{\supseteq}&^{U}g'  \text{ (for some elementary matrix }g'\in \EU_{2n+1}(R_\m,\Delta_\m)\text{ which is not }(I_\m,\Omega_\m)\text{-elementary)} \\
\overset{L.\ref{79}(1)}{\supseteq}&V_1 \text{ (for some } V_1 \in B_\m \text{).}
\end{align*}
\end{proof}
\begin{Lemma}\label{81}
Let $(I,\Omega)$ be an odd form ideal and $H$ an E-normal subgroup of level $(I,\Omega)$. If $H\not\subseteq \CU_{2n+1}((R,\Delta),(I,\Omega))$, then $H$ contains an elementary matrix which is not $(I,\Omega)$-elementary.
\end{Lemma}
\begin{proof}
Let $h\in H\setminus \CU_{2n+1}((R,\Delta),(I,\Omega))$. By Lemma \ref{75} there is a maximal ideal $\m$ of $C$ such that $h':=F_\m(h)\in \U_{2n+1}(R_\m,\Delta_\m)\setminus \CU_{2n+1}((R_\m,\Delta_\m),(I_\m,\Omega_\m))$. By Lemma \ref{56} there is an $s_0\in S_\m$ such that $\phi_\m$ is injective on $\pi(K)$ where $K=\U_{2n+1}((R,\Delta),(s_0R,s_0\Delta))\cap \NU_{2n+1}((R,\Delta),(I,\Omega))$. By Remark \ref{42}, $\NU_{2n+1}((R,\Delta),(s_0R,s_0\Delta))=\U_{2n+1}(R,\Delta)$ (since $J(s_0\Delta)\subseteq s_0R$). Hence $K$ is a normal subgroup of $\NU_{2n+1}((R,\Delta),(I,\Omega))$. Let $A$ and $A_\m$ be the bases defined in Lemma \ref{77} and choose an $U\in A_\m$. By Lemma \ref{78} there is a $k\in\mathbb{N}$ and elements $\xi\in \EU_{2n+1}((R_\m,\Delta_\m),(I_\m,\Omega_\m))$, $g_0,\dots,g_k\in F_\m(\EU_{2n+1}(R,\Delta))$, $\epsilon_0,\dots,\epsilon_k\in \EU_{2n+1}(R_\m,\Delta_\m)$ and $l_1,\dots,l_k\in \{\pm 1\}$ such that $g_k$ is an elementary matrix in $\EU_{2n+1}(R_\m,\Delta_\m)$ which is not $(I_\m,\Omega_\m)$-elementary,
\[ ^{\epsilon_{k}}([^{\epsilon_{k-1}}(\dots^{\epsilon_2}([^{\epsilon_1}([^{\epsilon_0}h',g_0]^{l_1}),g_1]^{l_2})\dots),g_{k-1}]^{l_{k}})\xi=g_k\tag{76.1}\]
and
\[^{d_i}g_i\in U~\forall i\in \{0,\dots,k\}\]
where $d_i=(\epsilon_{i}\hdots\epsilon_0)^{-1}\hspace{0.1cm}~(0\leq i\leq k)$. By conjugating (76.1) by $d_k$ we get
\[[\dots[[h',^{d_0}\hspace{-0.1cm}g_0]^{l_1},^{d_1}\hspace{-0.1cm}g_1]^{l_2}\dots,^{d_{k-1}}\hspace{-0.1cm}g_{k-1}]^{l_{k}}({}^{d_k}\xi)=^{d_k}\hspace{-0.1cm}g_k.\tag{76.2}\]
By Corollary \ref{80} there is a $V\in B_\m$ (where $B_\m$ is defined as in Lemma \ref{79}) such that 
\[V\subseteq {}^{U}(^{d_k}g_k).\tag{76.3}\]
By (76.2) and (76.3) we have 
\[V\subseteq {}^{U}(^{d_k}g_k)= {}^{U}([\dots[[h',^{d_0}\hspace{-0.1cm}g_0]^{l_1},^{d_1}\hspace{-0.1cm}g_1]^{l_2}\dots,^{d_{k-1}}\hspace{-0.1cm}g_{k-1}]^{l_{k}}({}^{d_k}\xi))\tag{76.4}.\]
By applying $\pi_\m$ to (76.4) we get 
\[\pi_\m(V)\subseteq \pi_\m({}^{U}([\dots[[h',^{d_0}\hspace{-0.1cm}g_0]^{l_1},^{d_1}\hspace{-0.1cm}g_1]^{l_2}\dots,^{d_{k-1}}\hspace{-0.1cm}g_{k-1}]^{l_{k}}))\tag{76.5}\]
since $\xi\in \EU_{2n+1}((R_\m,\Delta_\m),(I_\m,\Omega_\m))\subseteq \U_{2n+1}((R_\m,\Delta_\m),(I_\m,\Omega_\m))$.
One checks easily that 
\[{}^{U}([\dots[[h',^{d_0}\hspace{-0.1cm}g_0]^{l_1},^{d_1}\hspace{-0.1cm}g_1]^{l_2}\dots,^{d_{k-1}}\hspace{-0.1cm}g_{k-1}]^{l_{k}})\subseteq  F_\m(H\cap K)\tag{76.6}\]
since $^{d_i}g_i\in U~\forall i\in\{0,\dots,k\}$. Let $\hat V\in B$ such that $F_\m(\hat V)=V$. By (76.5) and (76.6) we have $\pi_\m(F_\m(\hat V))\subseteq \pi_\m(F_\m(H\cap K))$ which is equivalent to $\phi_\m(\pi(\hat V))\subseteq \phi_\m(\pi(H\cap K))$. Since $\phi_\m$ is injective on $\pi(K)$, it follows that $\pi(\hat V)\subseteq \pi(H)$ (note that $\hat V\subseteq K$). Hence $\hat V\subseteq H\cdot \U_{2n+1}((R,\Delta),(I,\Omega))$ and therefore $E(\hat V)\subseteq H\cdot \U_{2n+1}((R,\Delta),(I,\Omega))$ where $E(\hat V)$ denotes the normal closure of $\hat V$ in $\EU_{2n+1}(R,\Delta)$ (note that both $H$ and $\U_{2n+1}((R,\Delta),(I,\Omega))$ are E-normal). Thus
\begin{align*}
\hat V\subseteq E(\hat V)=&[\EU_{2n+1}(R,\Delta),E(\hat V)]\\
\subseteq&[\EU_{2n+1}(R,\Delta),H\cdot \U_{2n+1}((R,\Delta),(I,\Omega))]\\
\subseteq& [\EU_{2n+1}(R,\Delta),H]({}^H[\EU_{2n+1}(R,\Delta),\U_{2n+1}((R,\Delta),(I,\Omega))])\\
=&[\EU_{2n+1}(R,\Delta),H]({}^H\!\EU_{2n+1}((R,\Delta),(I,\Omega)))\\
\subseteq& H
\end{align*}
by Theorem \ref{51}. Clearly $\hat V$ contains elementary matrices which are not $(I,\Omega)$-elementary.
\end{proof}
\begin{Theorem}\label{82}
Let $H$ be a subgroup of $\U_{2n+1}(R,\Delta)$. Then $H$ is E-normal if and only if there is an odd form ideal $(I,\Omega)$ of $(R,\Delta)$ such that 
\[\EU_{2n+1}((R,\Delta),(I,\Omega))\subseteq H \subseteq \CU_{2n+1}((R,\Delta),(I,\Omega)).\]
Further $(I,\Omega)$ is uniquely determined, namely it is the level of $H$.
\end{Theorem}
\begin{proof}
Follows from Theorem \ref{51} and Lemma \ref{81} (see the proof of Theorem \ref{74}).
\end{proof}
\subsection{Quasifinite case}
In this subsection we assume that $R$ is {\it quasifinite}. By this we mean that $R$ is the direct limit of subrings $R_i~(i\in \Phi)$ which are almost commutative (i.e. finitely generated as modules over their centers), involution invariant and contain $\lambda$ and $\mu$. 
\begin{Lemma}\label{83}
Each $R_i$ is the direct limit of Noetherian, involution invariant subrings $R_{ij}~(j\in \Psi_i)$ containing $\lambda$ and $\mu$ such that for any $j\in \Psi_i$ there is a subring $C_{ij}$ of $\Center(R_{ij})$ with the properties
\begin{enumerate}[(1)]
\item[\textnormal{(1)}] $\bar c=c$ for any $c\in C_{ij}$,
\item[\textnormal{(2)}] if $\Delta_{ij}$ is an odd form parameter of $R_{ij}$, $(I_{ij},\Omega_{ij})$ an odd form ideal of $(R_{ij},\Delta_{ij})$, $(0,x)\in \Omega_{ij}$ and $c\in C_{ij}$, then $(0,cx)\in\Omega_{ij}$,
\item[\textnormal{(3)}] $(R_{ij})_\m$ is semilocal for any maximal ideal $\m$ of $C_{ij}$ and 
\item[\textnormal{(4)}] if $\Delta_{ij}$ is an odd form parameter of $R_{ij}$ and $(I_{ij},\Omega_{ij})$ an odd form ideal of $(R_{ij},\Delta_{ij})$, then $\Omega^{I_{ij}}_{max}/\Omega_{ij}$ is a Noetherian $C_{ij}$-module.
\end{enumerate}
\end{Lemma}
\begin{proof}
Denote the center of $R_i$ by $C_i$. Since $R_i$ is almost commutative, there are an $p\in\mathbb{N}$ and $x_1,\dots,x_q\in R_i$ such that $R_i = C_ix_1+\dots+C_ix_q$. For each $k,l\in\{1,\dots,q\}$ there are $a^{(kl)}_1,\dots, a^{(kl)}_q\in C_i$ such that $x_kx_l=\sum\limits_{p=1}^{q}a^{(kl)}_px_p$. Further for each $k\in\{1,\dots,q\}$ there are $b^{(k)}_1,\dots, b^{(k)}_q\in C_i$ such that $\bar x_k=\sum\limits_{p=1}^{q}b^{(k)}_px_p$. Finally there are $c_1,\dots, c_q\in C_i$ and $d_1,\dots,d_q\in C_i$ such that $\lambda =\sum\limits_{p=1}^{q}c_px_p$ and $\mu =\sum\limits_{p=1}^{q}d_px_p$. Set \[K:=\mathbb{Z}[a^{(kl)}_p,\overline{a^{(kl)}_p},b^{(k)}_p,\overline{b^{(k)}_p},c_p,\overline{c_p},d_p,\overline{d_p}\mid k,l,p\in\{1,\dots,q\}].\]
One checks easily that $C_i$ is a $K$-algebra and the direct limit of all involution invariant $K$-subalgebras $A_{ij}~(j\in \Psi_i)$ of $C_i$ which are finitely generated over $K$.
For any $j\in \Psi_i$ set $R_{ij}:=A_{ij}+A_{ij}x_1+\dots+A_{ij}x_q$. One checks easily that each $R_{ij}$ is an involution invariant subring of $R_i$ containing $\lambda$ and $\mu$. Further $\varinjlim\limits_{j} R_{ij}=R_i$. Fix a $j\in \Psi_i$ and let $C_{ij}$ denote the subring of $A_{ij}$ consisting of all finite sums of elements of the form $a\bar a$ and $-a\bar a$ where $a\in A_{ij}$. Cleary $C_{ij}$ is a subring of $\Center(R_{ij})$ which has the properties (1) and (2). Property (3) can be shown as in \cite[proof of Lemma 8.3]{bak-vavilov}.\\
Next we show that $R_{ij}$ is Noetherian. Clearly $A_{ij}$ is a finitely generated $\mathbb{Z}$-algebra and hence also a finitely generated $C_{ij}$-algebra. Since for any $a\in A_{ij}$
\[a + \bar a = (a + 1)(\bar a + 1) - a\bar a-1,\]
$C_{ij}$ contains all sums $a+\bar a$ where $a\in A_{ij}$. Since any $a\in A_{ij}$ is root of the monic polynomial $X^2-(a+\bar a)X+a\bar a$, $A_{ij}$ is an integral extension of $C_{ij}$. Since $A_{ij}$ is an integral extension of $C_{ij}$ and a finitely generated $C_{ij}$-algebra, $A_{ij}$ is a finitely generated module over $\C_{ij}$ by \cite[Chapter VII, Proposition 1.2]{lang}. Since $R_{ij}$ is finitely generated over $A_{ij}$, it is a finitely generated $C_{ij}$-module. Since $K$ is a Noetherian ring, $A_{ij}$ is a Noetherian ring (by Hilbert's Basis Theorem) and hence $C_{ij}$ is a Noetherian ring (by the Eakin-Nagata Theorem). Thus $R_{ij}$ is a Noetherian $C_{ij}$-module and hence a Noetherian ring.\\
It remains to show property (4). Let $\Delta_{ij}$ be an odd form parameter of $R_{ij}$ and $(I_{ij},\Omega_{ij})$ an odd form ideal of $(R_{ij},\Delta_{ij})$. Set $\Gamma_{ij}:=\Gamma(\Omega_{max}^{I_{ij}})$ and $J_{ij}:=J(\Omega_{max}^{I_{ij}})$. Clearly $J_{ij}$ and $\Gamma_{ij}$ are $C_{ij}$-submodules of $R_{ij}$. Since $R_{ij}$ is a Noetherian $C_{ij}$-module, $J_{ij}$ and $\Gamma_{ij}$ are finitely generated. Hence there are $x'_1,\dots, x'_r\in J_{ij}$ and $z_1,\dots,z_s\in \Gamma_{ij}$ such that $J_{ij}=\sum\limits_{k=1}^rC_{ij}x'_k$ and $\Gamma_{ij}=\sum\limits_{l=1}^sC_{ij}z_l$. Since $x'_1,\dots,x'_r\in J_{ij}$, there are $y_1,\dots,y_r\in R_{ij}$ such that $(x'_1,y_1),\dots,(x'_r,y_r)\in \Omega^{I_{ij}}_{max}$. It is an easy exercise to show that $\Omega^{I_{ij}}_{max}/\Omega_{ij}$ is generated by $\{(x'_k,y_k)\+\Omega_{ij}\mid k\in \{1,\dots,r\}\}\cup\{(0,\pm z_l)\+\Omega_{ij}\mid l\in \{1,\dots,s\}\}$ as $A_{ij}$-module. Since $A_{ij}$ is a finitely generated module over $C_{ij}$, it follows that $\Omega^{I_{ij}}_{max}/\Omega_{ij}$ is a finitely generated $C_{ij}$-module. Thus $\Omega^{I_{ij}}_{max}/\Omega_{ij}$ is a Noetherian $C_{ij}$-module since $C_{ij}$ is a Noetherian ring.
\end{proof}
\begin{Lemma}\label{84}
Let $(I,\Omega)$ be an odd form ideal and $H$ an E-normal subgroup of level $(I,\Omega)$. Then $H\subseteq \CU_{2n+1}((R,\Delta),(I,\Omega))$.
\end{Lemma}
\begin{proof}
Let $\sigma\in H$ and $\tau\in \EU_{2n+1}(R,\Delta)$. We have to show that $[\sigma,\tau]\in \U_{2n+1}((R,\Delta),(I,\Omega))$. Since $R$ is quasifinite, it is the direct limit of almost commutative, involution invariant subrings $R_i~(i\in \Phi)$ containing $\lambda$ and $\mu$. By Lemma \ref{83} each $R_i$ is the direct limit of Noetherian, involution invariant subrings $R_{ij}~(j\in \Psi_i)$ containing $\lambda$ and $\mu$ such that for any $j\in \Psi_i$ there is a subring $C_{ij}$ of $\Center(R_{ij})$ with the properties (1)-(4) in Lemma \ref{83}. If $i\in \Phi$ and $j\in \Psi_i$, set $\Delta_{ij}:=\Delta\cap (R_{ij}\times R_{ij})$. One checks easily that $((R_{ij},~\bar{}~,\lambda,\mu),\Delta_{ij})$ is a Hermitian form ring. Clearly there are an $i\in \Phi$ and a $j\in \Psi_i$ such that $\sigma\in \U_{2n+1}(R_{ij},\Delta_{ij})$ and $\tau\in \EU_{2n+1}(R_{ij},\Delta_{ij})$. Set $H_{ij}:=\U_{2n+1}(R_{ij},\Delta_{ij})\cap H$. Then $H_{ij}$ is an E-normal subgroup of $\U_{2n+1}(R_{ij},\Delta_{ij})$ and $\sigma\in H_{ij}$. Let $(I_{ij},\Omega_{ij})$ denote the level of $H_{ij}$. Then obviously $I_{ij}\subseteq I$ and $\Omega_{ij}\subseteq\Omega$. By Theorem \ref{82},
\[H_{ij} \subseteq \CU_{2n+1}((R_{ij},\Delta_{ij}),(I_{ij},\Omega_{ij})).\]
Hence $[\sigma,\tau]\in \U_{2n+1}((R_{ij},\Delta_{ij}),(I_{ij},\Omega_{ij}))
\subseteq \U_{2n+1}((R,\Delta),(I,\Omega))$.
\end{proof}
\begin{Theorem}\label{85}
Let $H$ be a subgroup of $\U_{2n+1}(R,\Delta)$. Then $H$ is E-normal if and only if there is an odd form ideal $(I,\Omega)$ of $(R,\Delta)$ such that 
\[\EU_{2n+1}((R,\Delta),(I,\Omega))\subseteq H \subseteq \CU_{2n+1}((R,\Delta),(I,\Omega)).\]
Further $(I,\Omega)$ is uniquely determined, namely it is the level of $H$.
\end{Theorem}
\begin{proof}
The proof is the same as that of Theorem \ref{74}.
\end{proof}

\nocite{hahn-meara}
\bibliographystyle{plain} 
\bibliography{literatur_odd}
\end{document}